\documentclass[12pt]{amsart}
\usepackage{amsmath,amstext,amsfonts,
amsthm,amssymb,hyperref,amscd,wasysym,stmaryrd,bbm}
\usepackage{eucal}
\usepackage[all]{xy}               
\usepackage[usenames]{color}                          
\swapnumbers
\newtheorem{thm}[subsubsection]{Theorem}
\newtheorem{lem}[subsubsection]{Lemma}
\newtheorem{prp}[subsubsection]{Proposition}
\newtheorem{crl}[subsubsection]{Corollary}

\newtheorem*{Lem}{Lemma}
\newtheorem*{Prp}{Proposition}

{  \theoremstyle{definition}
           \newtheorem{dfn}[subsubsection]{Definition}
           
           \newtheorem{rem}[subsubsection]{Remark}

           \newtheorem*{Dfn}{Definition}
           
           \newtheorem*{Rem}{Remark}
           
}

\newcommand{\act}{\mathrm{act}}

\newcommand{\Alg}{\mathtt{Alg}}

\newcommand{\Ass}{\mathtt{Ass}}
\newcommand{\BAR}{\mathtt{Bar}}

\newcommand{\BM}{\mathtt{BM}}
\newcommand{\BMod}{\mathtt{BMod}}
\newcommand{\BMOD}{\mathtt{BMOD}}
\newcommand{\Cat}{\mathtt{Cat}}
\newcommand{\CAT}{\mathtt{CAT}}

\newcommand{\colax}{\mathrm{colax}}

\newcommand{\colim}{\operatorname{colim}}
\newcommand{\Com}{\mathtt{Com}}
\newcommand{\cond}{\mathit{cond}}
\newcommand{\const}{\mathrm{const}}
\newcommand{\Coop}{\mathtt{Coop}}
\newcommand{\cart}{\mathrm{cart}}
\newcommand{\coc}{\mathrm{coc}}

\newcommand{\End}{\operatorname{End}}
\newcommand{\eq}{\mathit{eq}}

\newcommand{\Fam}{\mathtt{Fam}}
\newcommand{\Fin}{\mathit{Fin}}
\newcommand{\Fun}{\operatorname{Fun}}
\newcommand{\Funop}{\operatorname{Funop}}

\newcommand{\id}{\mathrm{id}}
\newcommand{\LMod}{\mathtt{LMod}}
\newcommand{\RMod}{\mathtt{RMod}}
\newcommand{\LM}{\mathtt{LM}}
\newcommand{\lax}{\mathrm{lax}}

\newcommand{\Lur}{\mathrm{Lur}}

\newcommand{\Mon}{\mathtt{Mon}}

\newcommand{\Map}{\operatorname{Map}}
\newcommand{\one}{\mathbbm{1}}

\newcommand{\ol}{\overline}

\newcommand{\op}{\mathrm{op}}
\newcommand{\Op}{\mathtt{Op}}

\newcommand{\PCat}{\mathtt{PCat}}

\newcommand{\Quiv}{\mathtt{Quiv}}
\newcommand{\quiv}{\mathtt{quiv}}
\newcommand{\red}{\mathit{red}}
\newcommand{\rlarrows}{\substack{\longrightarrow\\ \longleftarrow}}

\newcommand{\RM}{\mathtt{RM}}
\newcommand{\RT}{\mathtt{RT}}
\newcommand{\rev}{\mathit{rev}}

\newcommand{\Sh}{\mathtt{Sh}}

\newcommand{\Tw}{\mathtt{Tw}}
\newcommand{\TENS}{\mathtt{Tens}} 
\newcommand{\TEN}{\mathtt{Ten}} 
\newcommand{\EN}{\mathtt{En}} 
\newcommand{\ten}{\mathtt{ten}} %

\newcommand{\wt}{\widetilde}

\newcommand{\cA}{\mathcal{A}}
\newcommand{\cB}{\mathcal{B}}
\newcommand{\cC}{\mathcal{C}}
\newcommand{\cD}{\mathcal{D}}
\newcommand{\cE}{\mathcal{E}}
\newcommand{\cF}{\mathcal{F}}

\newcommand{\cK}{\mathcal{K}}

\newcommand{\cM}{\mathcal{M}}

\newcommand{\cO}{\mathcal{O}}
\newcommand{\cP}{\mathcal{P}}
\newcommand{\cQ}{\mathcal{Q}}
\newcommand{\cR}{\mathcal{R}}
\newcommand{\cS}{\mathcal{S}}
\newcommand{\cT}{\mathcal{T}}

\newcommand{\cX}{\mathcal{X}}

\newcommand{\fA}{\mathfrak{A}}
\newcommand{\fB}{\mathfrak{B}}
\newcommand{\fP}{\mathfrak{P}}
\newcommand{\fX}{\mathfrak{X}}

\begin{document}

\title[]{Colimits in enriched $\infty$-categories and Day convolution}
\author{Vladimir Hinich}
\address{Department of Mathematics, University of Haifa,
Mount Carmel, Haifa 3498838,  Israel}
\email{hinich@math.haifa.ac.il}

\begin{abstract}
Let $\cM$ be a monoidal $\infty$-category with colimits.
In this paper we study colimits of $\cM$-functors
$\cA\to\cB$ where $\cB$ is left-tensored over $\cM$ and $\cA$ is an $\cM$-enriched category in the sense of \cite{H.EY}.
We prove that the enriched Yoneda embedding 
$Y:\cA\to P_\cM(\cA)$ defined in {\sl loc. cit.}
yields a universal $\cM$-functor. In case when $\cA$ has 
a certain monoidal structure, the category of enriched presheaves $P_\cM(\cA)$ inherits the same monoidal structure and the enriched
Yoneda embedding acquires the structure of universal monoidal 
$\cM$-functor.
\end{abstract}
\maketitle

\section{Introduction}

In this paper we use the word {\sl category} to denote an 
$\infty$-category and the word {\sl operad} to denote an $\infty$-operad 
in the sense of Lurie~\cite{L.HA}, Section 2. On the contrary, if we want 
to stress that a certain $\infty$-category is a category in the classical 
sense, we call it a conventional category.
\subsection{}

Throughout
the paper we assume that $\cM$ is a monoidal category with colimits, such that the tensor product in $\cM$ preserves
colimits in both arguments. This means that $\cM\in\Alg_\Ass(\Cat^L)$,
where $\Cat^L$ denotes the category of categories with colimits, the 
arrows being the colimit-preserving functors.
We denote by $\LMod_\cM$ the category
of left $\cM$-modules in $\Cat^L$. In~\cite{H.EY} we constructed
a Yoneda embedding   $Y:\cA\to P_\cM(\cA)$ of an $\cM$-enriched category
$\cA$ into the category of enriched presheaves. In this paper
we prove that $Y$ is universal among the functors to a left $\cM$-module
$\cB$ with colimits: one has a natural equivalence
\begin{equation}
\label{eq:Y*}
Y^*:\Fun_{\LMod_\cM}(P_\cM(\cA),\cB)\to\Fun_\cM(\cA,\cB).
\end{equation}

\subsection{}
The existence of the functor (\ref{eq:Y*}) results from 
the functoriality of the assignement $\cB\mapsto\Fun_\cM(\cA,\cB)$. 
There is a functor in the opposite direction
that can be described in two ways: as an operadic left Kan extension, or using the notion of  weighted colimit.
\footnote{One can think of colimits
for functors of two kinds: functors from one $\cM$-enriched 
category to another, and functors from an $\cM$-enriched 
category to a category left tensored over $\cM$.  In this 
paper we deal with this second kind of functors.} 

Given an $\cM$-functor $f:\cA\to\cB$, where $\cB$ is a 
category with colimits left-tensored over $\cM$,
the weighted colimit $\colim(f): P_\cM(\cA)\to\cB$ is 
defined. This gives a functor
\begin{equation}
\colim:\Fun_\cM(\cA,\cB)\to\Fun_{\LMod_\cM}(P_\cM(\cA),\cB)
\end{equation}
quasi-inverse to $Y^*$.

\subsection{}
In the case when $\cM$ is an $\cO$-algebra in the category 
of monoidal categories (that is, $\cM$ is a 
$\cO\otimes\Ass$-monoidal category), one can define 
$\cO$-monoidal enriched $\cM$-categories, as well as 
$\cO$-monoidal left-tensored categories over $\cM$. 
In this case, if $\cA$ is an $\cO$-monoidal $\cM$-enriched 
category, $P_\cM(\cA)$  inherits an $\cO$-monoidal 
structure, Yoneda embedding becomes an 
$\cO$-monoidal $\cM$-functor (see   
\ref{sss:lm-version}), and (\ref{eq:Y*}) induces an equivalence
\begin{equation}
\label{eq:equivalence-O}
\Fun^{\cO}_{\LMod_\cM}(P_\cM(\cA),\cB)\to\Fun^\cO_\cM(\cA,\cB)
\end{equation}
of the corresponding categories of $\cO$-monoidal functors.

The $\cO$-monoidal structure on $P_\cM(\cA)$ is an
enriched version of the Day convolution defining a monoidal
structure on the presheaves on a monoidal category.

\subsection{}
The paper was started with the aim to prove universality of
enriched Yoneda embedding constructed in \cite{H.EY}. At first
the task seemed very easy: given an $\cM$-enriched category $\cA$
and an $\cM$-functor $f:\cA\to\cB$, where $\cB$ is a 
left $\cM$-module with colimits, one can define
the functor $\colim f:P_\cM(\cA)\to\cB$ as a weighted colimit, using 
 Lurie's general machinery \cite{L.HA} of relative tensor product 
(this is now explained in the beginning of Section~\ref{sec:wc}). 
One can easily prove that 
any map $F:P_\cM(\cA)\to\cB$ in $\LMod_\cM$ is equivalent to
the colimit of its composition with the Yoneda embedding.
But we have not found an easy argument to show that for any $f\in\Fun_\cM(\cA,\cB)$
the composition of $\colim f$ with the 
Yoneda embedding gives back $f$.

This is why we had to add Section~\ref{sec:lurie} 
comparing our working definition of enriched categories with the one 
given by Lurie in~\cite{L.HA}, 4.2.1.28. Now universality of $Y$
follows from the description of colimit preserving left $\cM$-module
maps $P_\cM(\cA)\to\cB$ as operadic left Kan extensions of their
restriction to $\bar\cA\subset P_\cM(\cA)$, the essential image of the 
Yoneda embedding. 

An $\cO$-monoidal version of the universality easily follows
from the interpretation of this operadic Kan extension in terms of 
weighted colimits.

Our work has a very considerable overlap with
the recent manuscript by Hadrian Heine~\cite{HH}. In it the category
of enriched $\cM$-categories is proven to be equivalent to 
the one defined by Lurie (we only prove that the functor 
$\cA\mapsto\bar\cA$ is fully faithful). Heine also proves 
universality of the Yoneda embedding. It seems, however, that his 
methods are insufficient to deduce the $\cO$-monoidal version of the 
universality.

\subsection{}
In Section~\ref{sec:digest} we provide a digest of the theory of enriched 
categories and enriched Yoneda lemma. The notion of 
enriched category used here is the one presented
in \cite{H.EY}, Sect. 3. Our definition of enriched
categories is practically equivalent to the earlier
definition of \cite{GH}. 

J.~Lurie defines in \cite{L.HA}, 4.2.1.28 another notion
of $\cM$-enriched category, as a category weakly tensored
over $\cM$, and satisfying some properties.

In Section~\ref{sec:lurie} we compare the notion of enriched categories
used in this paper with the one defined by Lurie. We 
construct a fully faithful functor from the category
$\Cat(\cM)$ of categories enriched over a monoidal category 
$\cM$ with colimits, to the category
of Lurie enriched $\cM$-categories~\footnote{H.~Heine has recently proven that these two notion of enrichment are equivalent, see~\cite{HH}.}.
What is more important to us, we interpret $\cM$-functors
$f:\cA\to\cB$ from an enriched category $\cA$ to a left-
tensored category $\cB$ as functors between the categories 
weakly tensored over $\cM$.

In Section~\ref{sec:rtp} we review the theory of relative 
tensor products \cite{L.HA}, 3.1 and 4.6.

In Section~\ref{sec:bar} we study bar resolutions
for enriched presheaves. This is a technical section whose
result is only needed in the characterization 
\ref{prp:olke} of morphisms
of $\cM$-modules $P_\cM(\cA)\to\cB$ as operadic left Kan extensions.

The notion of relative tensor product allows us to define
in Section~\ref{sec:wc} the weighted colimits. 
In Section~\ref{sec:functorialityQ} we study the functoriality of
the construction of Section~\ref{sec:wc}. This allows one to deduce
the multiplicative version of the main universality result in 
Section~\ref{sec:wc-mult}.

\subsection{Acknowledgement}
We are grateful to Greg Arone, Ilan Barnea and Tomer Schlank for their interest in this work and for sharing their
preprint~\cite{ABS}. We are also grateful to H. Heine for informing 
us about his work~\cite{HH}. The first version of the manuscript 
contained a number of mistakes and unproven claims, and we are
extremely indebted to the referee who helped to sort them out.

The work was supported by ISF 786/19 grant.

\section{Enriched categories and enriched Yoneda: digest}
\label{sec:digest}
In this section we recall some important constructions
of \cite{H.EY}. The notion of operadic left Kan extension
is reviewed in \ref{ss:OLKE}.

\subsection{}
The category of operads $\Op$ is a subcategory of
$\Cat_{/Fin_*}$, where $\Fin_*$ is the category of finite pointed sets. If $\cO$ is an operad, we denote
$\Op_{/\cO}$ or just $\Op_\cO$ the category of $\cO$-operads, that is operads endowed with a morphism to $\cO$. The terminal object in $\Op$
is the operad for commutative algebras. We denote it $\Com$ or $\Fin_*$.

The operad $\Ass^\otimes$ governs associative algebras, and 
$\Op_\Ass$ is the category of planar operads. 
We denote by $\LM^\otimes,\BM^\otimes$ the operads governing the left 
modules and the bimodules, respectively. Thus,  the operad $\BM^\otimes$
has three colors, so that the $\BM^\otimes$-algebras are the triples 
$(A,M,B)$ consisting of two associative algebras $A$ and $B$ acting from 
the left and from the right on $M$.
Similarly, a $\BM^\otimes$-operad has three components, two planar 
operads $A$ and $B$,  and a category $M$, with two compatible weak 
actions of $A$ and of $B$ on $M$.

Following~\cite{L.HA}, 2.3.3 and \cite{H.EY}, 2.7.1, we often
replace  operads with their strong approximations. In particular,
we use the approximation $\Ass$, $\LM$ and $\BM$ of $\Ass^\otimes$,
$\LM^\otimes$ and $\BM^\otimes$ as defined in \cite{H.EY}, 2.9.

\subsection{Quivers}
The notion of enriched category, as presented in \cite{H.EY}, is based on a functor
\begin{equation}\label{eq:quiv}
\Quiv:\Cat^\op\times\Op_\Ass\to\Op_\Ass
\end{equation}
 carrying a pair 
$(X,\cM)$ to a planar operad $\Quiv_X(\cM)$ whose colors
are {\sl $\cM$-quivers}, that is functors $A:X^\op\times X\to\cM$. 
 
The functor (\ref{eq:quiv}) has two variations.
The first is a functor
\begin{equation}\label{eq:quiv-lm}
\Quiv^\LM:\Cat^\op\times\Op_\LM\to\Op_\LM,
\end{equation}
and the second is
\begin{equation}\label{eq:quiv-bm}
\Quiv^\BM:\Cat^\op\times\Op_\BM\to\Op_\BM.
\end{equation}
The functors are compatible: the $\Ass$-component of
$\Quiv^\LM_X(\cM,\cB)$ is $\Quiv_X(\cM)$, and so on.

In good cases, the functors $\Quiv$ applied to monoidal categories with enough colimits, produce a monoidal category.

\subsubsection{More details}
\label{sss:detailsonlmx}

In Section~\ref{sec:bar} we will need a more detailed information about the functor (\ref{eq:quiv-lm})
\footnote{The functors (\ref{eq:quiv}), (\ref{eq:quiv-bm})
have a similar description.}.

In what follows $\Delta_{/\LM}$ denotes the category of simplices
in $\LM$.
For a fixed $X\in\Cat$, one defines an $\LM$-operad $\LM_X$
by a presheaf 
$$ (\Delta_{/\LM})^\op\to\cS$$
given by the formula
$$\LM_X(\sigma)=\Map(\cF(\sigma),X),$$
where $\cF:\Delta_{/\LM}\to\Cat$ is a functor with values in conventional 
categories combinatorially defined in \cite{H.EY}, 3.2.

The $\LM$-operad $\LM_X$ is always flat, \cite{H.EY}, 3.3. This means that the functor $\Op_\LM\to\Op_\LM$ given by product with $\LM_X$,
admits a right adjoint, which is denoted $\Funop_\LM(\LM_X,\_)$.

Finally, given $\cM=(\cM_a,\cM_m)\in\Op_\LM$, one defines
 the $\LM$-operad $\Quiv^\LM_X(\cM)$ as $\Funop_\LM(\LM_X,\cM)$.
 
Two other variations of the category of quivers, $\Quiv$ and 
$\Quiv^\BM$, have a similar description.

Given $\cM=(\cM_a,\cM_m,\cM_b)\in\Op_\BM$, the $\BM$-operad
$\Quiv^\BM_X(\cM)$ has components $(\Quiv_X(\cM_a),\Fun(X,\cM_m),\cM_b)$.

\subsubsection{$\Cat$-enrichment}
\label{sss:cat-enrichment-quiv}

Let $\cO$ be an operad or a strong approximation of an operad.
The category of $\cO$-operads $\Op_\cO$ has a $\Cat$-enrichment
that assigns to $\cP,\cQ\in\Op_\cO$ the category $\Alg_{\cP/\cO}(\cQ)$.

The functor $\Quiv^\LM_X:\Op_\LM\to\Op_\LM$ respects this enrichment.
This means that, given $\cP,\cQ\in\Op_\LM$, one has a functor
\begin{equation}
\label{eq:cat-enrichment-quiv}
\Alg_{\cP/\LM}(\cQ)\to\Alg_{\Quiv^\LM_X(\cP)/\LM}(\Quiv_X^\LM(\cQ))
\end{equation}
extending the map
$$
\Map_{\Op_\LM}(\cP,\cQ)\to\Map_{\Op_\LM}(\Quiv^\LM_X(\cP),\Quiv^\LM_X(\cQ)).
$$
The map (\ref{eq:cat-enrichment-quiv}) is defined as follows. Its target
is naturally equivalent, according to \cite{H.EY}, 2.8.6, to 
$\Alg_{\Quiv^\LM_X(\cP)\times_\LM\LM_X}(\cQ)$. The 
map~(\ref{eq:cat-enrichment-quiv}) can therefore be defined as the one
induced by the canonical evaluation map 
$$
\Quiv_X^\LM(\cP)\times_\LM\LM_X=\Funop_\LM(\LM_X,\cP)\times_\LM\LM_X\to\cP.
$$
 
\subsection{Algebras in quivers}

Fixing the second (operadic) argument, we will look at the functors 
{\ref{eq:quiv})--(\ref{eq:quiv-bm})  as cartesian families of (planar, $\LM$ or $\BM$) operads. Let us describe our interpretation for the categories of algebras
in various operads of quivers.

\subsubsection{Enriched precategories}

For a fixed planar operad $\cM$ with colimits, an associative algebra
in the family $\Quiv(\cM)$ is called $\cM$-enriched precategory. 
We denote $\PCat(\cM)=\Alg_\Ass(\Quiv(\cM))$ the category
of $\cM$-enriched precategories.

An enriched precategory $\cA$ has a category $X$ of objects, and an associative multiplication in the underlying quiver $\cA:X^\op\times X\to\cM$.

\subsubsection{$\cM$-functors}
\label{sss:Mfunctors}

Fix an $\LM$-operad, consisting of a planar operad $\cM$ 
and a category $\cB$ weakly tensored over $\cM$. For fixed 
$X\in\Cat$, the $\LM$-operad $\Quiv_X^\LM(\cM,\cB)$ 
consists of the planar operad $\Quiv_X(\cM)$ and a category 
$\Fun(X,\cB)$, weakly tensored over $\Quiv_X(\cM)$.

The $\LM$-operads $\Quiv_X^\LM(\cM,\cB)$ form a family
$\Quiv^\LM(\cM,\cB)$. 

An $\LM$-algebra in it consists of
a pair $(\cA,f)$ where $\cA$ is an $\cM$-enriched precategory, and $f$ is an $\cA$-module in $\Fun(X,\cB)$.

We denote 
$\PCat^\LM(\cM,\cB)=\Alg_\LM(\Quiv^\LM(\cM,\cB)).$

We interpret $\cA$-modules in $\Fun(X,\cB)$ as $\cM$-functors from $\cA$ to $\cB$, whence the notation
\begin{equation}
\label{eq:funM}
\Fun_\cM(\cA,\cB)=\LMod_\cA(\Fun(X,\cB)),
\end{equation}
the category of left $\cA$-modules in $\Fun(X,\cB)$.

\subsubsection{} Assume now $\cM$ is a monoidal category
with colimits. Applying the above to $\cB:=\cM$ considered 
as a right $\cM$-module (which is the same as left  
$\cM^\rev$-module), we can define the category of enriched 
presheaves $P_\cM(\cA)=$ \newline $\Fun_{\cM^\rev}(\cA^\op,\cM)$. It is 
left-tensored over $\cM$ and has colimits.

Yoneda embedding is an $\cM$-functor $Y:\cA\to P_\cM(\cA)$,
defined by $\cA$-bimodule structure on $\cA$, see details
in \cite{H.EY}, Section 6.

\subsubsection{Enriched categories} 
\label{sss:enrichedcategory}
An $\cM$-enriched 
category is an enriched precategory satisfying a certain
completeness condition. The full embedding $\Cat(\cM)\subset\PCat(\cM)$ is right adjoint to a localization
functor $L:\PCat(\cM)\to\Cat(\cM)$ 
which can be described as follows. Given 
$\cA\in\Alg_\Ass(\Quiv(\cM))$, we define 
$X$ as the subspace of $P_\cM(\cA)^\eq$ spanned by the representable functors,
and define $L(\cA)$ as the endomorphism object
in $\Quiv_X(\cM)$ of the tautological embedding
$X\to P_\cM(\cA)$, see \cite{H.EY}, 7.2.

\subsubsection{Restriction of scalars}
\label{sss:restriction}

Given a cartesian family $p:\cQ\to B\times\LM$ of 
$\LM$-operads, the embedding $\Ass\to\LM$ induces a functor
$$ \Alg_\LM(\cQ)\to\Alg_\Ass(\cQ)$$
which is a cartesian fibration. This result can be found in \cite{L.HA}, 4.2.3.2 or \cite{H.EY}, 2.11.

\subsubsection{}
\label{sss:inverseimage}

Let $(\cM,\cB)$ be an $\LM$-operad.
The assignment $\cA\mapsto\Fun_\cM(\cA,\cB)$ is contravariant in $\cA$. This is a special case of
a general setup presented in \ref{sss:restriction}.

Thus,  a map $f:\cA\to\cA'$ of $\cM$-enriched precategories 
gives rise to a functor 
$f^*:\Fun_\cM(\cA',\cB)\to\Fun_\cM(\cA,\cB)$.
The definition of $f^*$ allows one to compose a map 
of $\cM$-enriched precategories $f:\cA\to\cA'$ with an 
$\cM$-functor $\cA'\to\cB$.

\subsection{Operadic left Kan extensions}
\label{ss:OLKE}

Operadic colimits and operadic left Kan extensions 
defined in Lurie's \cite{L.HA}, Section 3.1, are a part of  the
construction of a free operad algebra. In this subsection we present 
necessary details connected to these notions.

Let $\cO$ be an operad and let $\cC\in\Alg_\cO(\Cat^L)$ be an
$\cO$-monoidal category with colimits. Given a morphism $f:\cP\to\cQ$
of $\cO$-operads, one has a forgetful functor 
$f^*:\Alg_\cQ(\cC)\to\Alg_\cP(\cC)$. In this context the operadic left Kan extension always exists and defines a functor
$$
f_!:\Alg_\cP(\cC)\to\Alg_\cQ(\cC)
$$
left adjoint to $f^*$.

\subsubsection{Operadic colimit, see~\cite{L.HA}, 3.1.1}
For $K\in\Cat$ we denote $K^\triangleright=(K\times[1])\sqcup^K[0]$,
the category obtained by adding to $K$ the terminal object $*$.

Let $p:\cC\to\cO$ be an $\cO$-operad. For any functor $f:K\to\cC^\act$
to the subcategory of $\cC$ spanned by the active arrows we denote
\begin{equation}
\label{eq:act-to}
\cC^\act_{f/}=\{f\}\times_{\Fun(K,\cC^\act)}\Fun(K^\triangleright,\cC^\act)\times_{\cC^\act}\cC_{\langle 1\rangle},
\end{equation}
where the arrows from $\Fun(K^\triangleright,\cC^\act)$ are defined by the embeddings $K\to K^\triangleright$ and $*\in K^\triangleright$.

A functor $\bar f:K^\triangleright\to\cC^\act$ with $f:=\bar f|_K$
is called a weak operadic colimit diagram
if the natural map
$$
\cC^\act_{\bar f/}\to\cC^\act_{f/}\times_{\cO^\act_{q\circ f/}}
\cO^\act_{p\circ\bar f/}
$$
is an equivalence.

One says that $\bar f$ is an operadic colimit diagram
if for any $C\in\cC$ the composition
$$
K^\triangleright\stackrel{\bar f}{\to}\cC^\act
\stackrel{\oplus C}{\to}\cC^\act
$$
is a weak operadic colimit diagram.

Let $f:K\to\cC^\act$ be a diagram in an $\cO$-operad $\cC$ and let
$\bar g:K^\triangleright\to\cO$ be an extension of
$g=p\circ f$. We say that the diagram $f$ has an operadic colimit over
$\bar g$ if there exists $\bar f$ over $\bar g$ that is an operadic colimit diagram.

In the case when $\cC$ is an $\cO$-monoidal category with colimits, such that the monoidal structure is compatible with the colimits, any
diagram $f:K\to\cC^\act$ has an operadic colimit.

\subsubsection{Operadic left Kan extensions}

We present below a definition of operadic left Kan extension
due to Lurie, \cite{L.HA}, 3.1.2. We restrict ourselves to a special 
case of what Lurie calls ``free algebra'', see \cite{L.HA}, 3.1.3.1.

Given $f:\cP\to\cQ$  a morphism of $\cO$-operads and an $\cO$-operad
$\cC$, one has an obvious functor
$$
f^*:\Alg_\cQ(\cC)\to\Alg_\cP(\cC).
$$

For any $q\in \cQ$ we define $K_q=\cP\times_\cQ\cQ^\act_{/q}$. 

Given $A\in\Alg_\cP(\cC)$ and $B\in\Alg_{\cQ}(\cC)$,   a morphism
$ j:A\to f^*(B)$ in $\Alg_\cP(\cC)$  determines a morphism of functors 
$\alpha_q\to\const_{B(q)}:K_q\to\cC^\act$,
with $\alpha_q$ being the composition of $A$ with the projection $K_q\to\cP$ and $\const_{B(q)}$ being the constant functor with the value $B(q)$. Equivalently, this translates into a functor  
\begin{equation}
\label{eq:alphaq}
\bar\alpha_q:K_q^\triangleright\to\cC^\act.
\end{equation}

A morphism $j:A\to f^*(B)$ 
is called an operadic left Kan extension of $A$ with respect to $f$
if for any $q\in\cQ$ the functor
$
\bar\alpha_q:K_q^\triangleright\to\cC^\act
$ 
is an operadic colimit diagram.

In the case when $\cC\in\Alg_\cO(\Cat^L)$, the operadic left Kan extension exists and defines a functor $f_!:\Alg_\cP(\cC)\to\Alg_\cQ(\cC)$
left adjoint to $f^*$.

\section{Lurie's enriched categories}
\label{sec:lurie}

In this section we compare the notion of
$\cM$-enriched category and of an $\cM$-enriched functor,
as presented in~\ref{sss:enrichedcategory} and
\ref{sss:Mfunctors},
with the similar (but simpler) notions of Lurie, \cite{L.HA}, 4.2.1. 

\subsection{$\LM$-operads and Lurie's enriched categories}

Probably, the simplest way to define an enriched $\infty$-category
over a monoidal category $\cM$ is presented in Lurie's
\cite{L.HA}, 4.2.1.28. 

\subsubsection{}
The map $a:\Ass\to\LM$ of operads induces a base change
functor
\begin{equation}
\label{eq:lmtoass}
 a^*:\Op_{\LM}\to\Op_\Ass
\end{equation}
assigning to each $\LM$-operad $\cO$ its planar component
$\cO_a\to\Ass$. The fiber $\cO_m$ at $m\in\LM$ is a category that is called {\sl weakly enriched} over $\cO_a$.

We denote by $\LMod^w_\cM$ 
the fiber of $a^*$ at $\cM\in\Op_\Ass$. This is the 
category of categories weakly enriched over $\cM$.

An object of $\LMod^w_\cM$ is an $\LM$-operad $\cO$
together with an equivalence $\cM=\Ass\times_\LM\cO$.
We will sometimes denote it as a pair
$(\cM,\cA)$, where $\cA=\{m\}\times_{\LM}\cO$,
or (when $\cM$ is fixed) as  $\cA$.

\subsubsection{}
\label{sss:funlmodw}

For $\cO,\cO'\in\LMod^w_\cM$ we define
$\Fun_{\LMod^w_\cM}(\cO,\cO')$ as the fiber of the map
$$\Alg_{\cO/\LM}(\cO')\to\Alg_{\cM/\Ass}(\cM)$$
at $\id_\cM$.

\

Let now  $\cM$ be a monoidal category.
Here is the definition of Lurie's $\cM$-enriched category.

\begin{dfn}
Let $\cM$ be a monoidal category. A Lurie $\cM$-enriched category $\cA$
is an $\LM$-operad $\cO$ with the equivalences
$\cM=\Ass\times_\LM\cO$, $\cA=\{m\}\times_{\LM}\cO$,
satisfying the following properties.
\begin{itemize}
\item[1.] The natural map $\oplus m_i\to \otimes m_i$
induces an equivalence
$\Map((\otimes m_i)\oplus a,b)
\to\Map(\oplus m_i\oplus a,b)$ for 
any $m_i\in \cM$ and $a,b\in\cA$. Here we use the sign
$\oplus$ as in \cite{L.HA}, 2.1.1.15
~\footnote{The first property is a {\sl pseudo-enrichment} in Lurie's terminology. see~\cite{L.HA}, 4.2.1.25. The condition makes sense
for the number of factors $n\geq 0$.}.
\item[2.] For any $a,b\in\cA$ the weak enrichment functor
$\hom_\cA(a,b):\cM^\op\to\cS$,
defined by the formula
$$\hom_\cA(a,b)(m)=\Map(m\oplus a,b),$$
 is representable.
\end{itemize}
\end{dfn}
Lurie $\cM$-enriched categories form a category denoted 
$\Cat^\Lur(\cM)$. This is a full subcategory of
$\LMod^w_\cM$ spanned by the Lurie $\cM$-enriched categories.

In this section we assign to any
 $\cM$-enriched category $\cA$ a Lurie $\cM$-enriched category 
$\bar\cA$. We prove  that the category
$\Cat(\cM)$ defined in~\ref{sss:enrichedcategory} 
is equivalent to a full subcategory 
of $\Cat^\Lur(\cM)$, see~\ref{crl:ff}.  
Note that H.~Heine has recently proven ~\cite{HH} that the two notions
are equivalent.

\subsubsection{}
Let $\cM$ be a monoidal category with colimits.

For any $\cM$-enriched category $\cA$ we define $\bar\cA\subset P_\cM(\cA)$, the full subcategory of $P_\cM(\cA)$ 
spanned by the representable presheaves. Obviously, 
$\bar\cA$ is an $\cM$-enriched category in the sense of
Lurie. By~\cite{H.EY}, 6.1.4, this defines a functor 
$$\lambda:\Cat(\cM)\to\Cat^\Lur(\cM).$$

Corollary~\ref{crl:ff} asserts that this functor is fully 
faithful. As H. Heine shows in \cite{HH}, the functor $\lambda$ is 
actually an equivalence, see Remark~\ref{rem:lambdaequivalence}.

\subsection{Baby Yoneda functor}
\label{ss:babyY}

In what follows we denote by $(\cM,\bar\cA)^\otimes$ 
the $\LM$-operad formed by the category $\bar\cA$ weakly tensored over a monoidal category $\cM$.

Let $\cM$ be a monoidal category with colimits, 
$\cA$ be an $\cM$-enriched category and $\cB$ be a 
category with colimits left-tensored over $\cM$.

Let $(\cM,\bar\cA)$ be the corresponding Lurie enriched category.

Our aim is to construct an equivalence 
\begin{equation}
\label{eq:funfun}
\Fun_{\LMod_\cM^w}(\bar\cA,\cB)\to\Fun_\cM(\cA,\cB).
\end{equation}
 
\subsubsection{}
Functoriality of the assignment
$$\cB\mapsto\Fun_\cM(\cA,\cB),$$
as defined by the formula (\ref{eq:funM}),
and the preservation of $\Cat$-enrichment by $\Quiv^\LM$, see
\ref{sss:cat-enrichment-quiv},
yields a canonical functor
\begin{equation}
\label{eq:funalgtofun0}
\Fun_\cM(\cA,\cB')\times
\Fun_{\LMod_\cM^w}(\cB',\cB)\to\Fun_\cM(\cA,\cB).
\end{equation}
In particular, for $\cB':=\bar\cA$, we get

\begin{equation}
\label{eq:funalgtofun}
\Fun_\cM(\cA,\bar\cA)\times
\Fun_{\LMod_\cM^w}(\bar\cA,\cB)\to\Fun_\cM(\cA,\cB).
\end{equation}

We claim that the  Yoneda embedding 
$Y:\cA\to P_\cM(\cA)$ factors uniquely 
through the full embedding 
$\bar  Y:\bar\cA\to P_\cM(\cA)$;  
the natural $\cM$-functor $y:\cA\to\bar\cA$ so defined
will be denoted $y$ and called {\sl the baby Yoneda functor}.
The existence (and uniqueness) of the baby Yoneda 
immediately follows from the lemma below.
\begin{lem}
Let $\cB$ be a full subcategory in $\cC$ that is weakly enriched over a monoidal category $\cM$. Then, for an associative
algebra $\cA$ in $\cM$, one has
$$\LMod_\cA(\cB)=\LMod_\cA(\cC)\times_\cC\cB.$$
\end{lem}\qed

Note the following obvious property of the $\cM$-functor
$y:\cA\to\bar\cA$.
\begin{lem}\label{lem:forget}

The forgetful functor $\Fun_\cM(\cA,\bar\cA)\to\Fun(X,\bar\cA)$ carries $y$ to a map $i:X\to\bar\cA$ identifying
$X$ with the maximal subspace of $\bar\cA$.
\end{lem}\qed

\subsection{The functor (\ref{eq:funfun}) is an equivalence}

In Proposition~\ref{prp:y-equivalence} we prove that the functor 
(\ref{eq:funfun}) is an equivalence. The idea is to present the source and the target of the map by a monad on $\Fun(X,\cB)$ and to verify the equivalence of the monads.

\subsubsection{}
Recall that $(\cM,\cB)$ is an $\LM$-monoidal category with colimits. 
Let $\cA\in\Alg(\Quiv_X(\cM))$. Let
$\Phi:X\to\cB$ be an $\cA$-module in $\Fun(X,\cB)$ and let 
$a:\phi\to\Phi$ be an arrow in $\Fun(X,\cB)$. 
In Lemma~\ref{lem:Afree} below we formulate the condition for $a$ to represents $\Phi$ as a free $\cA$-module generated by $\phi$. 

The $\cA$-module structure on $\Phi$ defines
an active arrow $(\cA,\Phi)\to\Phi$ in 
$\Quiv_X^\LM(\cM,\cB)$
that is explicitly described in~\cite{H.EY}
4.2.1 and 4.3.1, case (w2) with $n=2$, $k=1$. 

Here is the description. The active arrow above is given by a
map $A:=[1]\times_\LM\LM_X\to(\cM,\cB)^\act$ where 
$$A=A_0\sqcup^C(C\times[1])\sqcup^CA_1,$$
with $A_0=X\times X^\op\times X$, $A_1=X$, $C=\Tw(X)^\op\times X$,
the map $C\to A_0$ is given by the projection $\Tw(X)^\op\to X\times X^\op$, whereas $C\to A_1$ is the projection to the last factor.

This yields, for any $x\in X$, a functor
\begin{equation}
\label{eq:Amodulepart}
\bar\theta^{\Phi}_x:(\Tw(X)^\op)^\triangleright\to(\cM,\cB)^\act
\end{equation}
carrying the terminal object $*\in(\Tw(X)^\op)^\triangleright$
to $\Phi(x)$ and the arrow $\alpha:z\to y$ from $\Tw(X)$ to 
the pair $(\cA(y,x),\Phi(z))$.

We denote by $\theta^\Phi_x$ the restriction  of
$\bar\theta^{\Phi}_x$ to $\Tw(X)^\op$. The functor 
$\theta^\phi_x:\Tw(X)^\op\to(\cM,\cB)^\act$ is define in the same way
and the map $a:\phi\to\Phi$
gives rise to a map of functors
$a:\theta^\phi_x\to\theta^\Phi_x$.

\subsubsection{}
\label{sss:obs1}
A functor $f:K^\triangleright\to\cC$
can be uniquely presented by a map $f|_K\to \const_{f(*)}$ in 
$\Fun(K,\cC)$, where $\const_{f(*)}$ is the constant functor with the 
value $f(*)\in\cC$.
In particular, given $f$ as above and an arrow $\alpha:f_1\to f|_K$
in $\Fun(K,\cC)$, we get a canonically defined functor $f':K^\triangleright\to\cC$ with $f'|_K=f_1$ and $f'(*)=f(*)$. 

This allows one to define 
\begin{equation}
\label{eq:thetaax}
\bar\theta^a_x:
(\Tw(X)^\op)^\triangleright\to(\cM,\cB)^\act
\end{equation}
as the functor induced by $a:\phi\to\Phi$  whose
restriction to $\Tw(X)^\op$ is $\theta^\phi_x$
and the value at the terminal object is $\Phi(x)$.

\subsubsection{Cocartesian shift}
\label{sss:obs2}
Let $p:\cC\to B$ be a cocartesian fibration
and let, as above, $f:K^\triangleright\to\cC$ be a functor.

As above, $f$ gives rise to a map $f|_K\to \const_{f(*)}$ in
$\Fun(K,\cC)$, as well as to its image 
$p\circ f|_K\to\const_{p\circ f(*)}$ in $\Fun(K,B)$. Since the map 
$\Fun(K,p):\Fun(K,\cC)\to\Fun(K,B)$ is a cocartesian fibration, we get a 
map of functors
$$
f\to\Sh(f):K^\triangleright\to\cC,
$$
such that $p\circ\Sh(f)=\const_{p\circ f(*)}$ and
for each $x\in K$ the arrow $f(x)\to\Sh(f)(x)$ is $p$-cocartesian. 
In this case we will say that $\Sh(f)$ is obtained from $f$ by a 
{\sl cocartesian shift}.

\subsubsection{} Applying the cocartesian shift to (\ref{eq:thetaax}), we get
$$
\Sh(\bar\theta^a_x):(\Tw(X)^\op)^\triangleright\to\cB.
$$

One has the following.
\begin{lem}
\label{lem:Afree}
A map $a:\phi\to\Phi$ in $\Fun(X,\cB)$ presents $\Phi$
as a free $\cA$-module if and only if for any $x\in X$
the diagram $\bar\theta^a_x$ is an operadic colimit diagram
(or, equivalently, if $\Sh(\bar\theta^a_x)$ is a colimit diagram).
\end{lem}
\begin{proof}
The map $a:\phi\to\Phi$ presents $\Phi$ as a free $\cA$-module generated by $\phi$ if $a$ induces a cocartesian
arrow $(\cA,\phi)\to\Phi$ in $\Quiv_X^\LM(\cM,\cB)$.
This easily translates to our condition.
\end{proof}

The evaluation of (\ref{eq:funalgtofun}) at $y$
defines the canonical functor
$$y^*:\Fun_{\LMod_\cM^w}(\bar\cA,\cB)\to\Fun_\cM(\cA,\cB).
$$

Lemma~\ref{lem:forget} asserts that $G\circ y^*=G'$
where $G:\Fun_\cM(\cA,\cB)\to\Fun(X,\cB)$ is the forgetful 
functor and $G':\Fun_{\LMod_\cM^w}(\bar\cA,\cB)\to
\Fun(X,\cB)$ is given by the composition with $X\to\bar\cA$.

\subsubsection{} We will now prove that $y^*$  is an equivalence. 
Our proof will use the description of the source and the target of $y^*$ 
by monads on $\Fun(X,\cB)$.

Let us consider the following diagram.

 \begin{equation}
\xymatrix{
&{\Fun_{\LMod_\cM^w}(\bar\cA,\cB)} \ar^{y^*}[rr]
\ar^{G'}[rd]&{} &{ \Fun_\cM(\cA,\cB)  }\ar^{G}@<1ex>[ld] \\
&{} &{\Fun(X,\cB)}\ar^{F}@{.>}[ru]\ar^{F'}@{.>}@<1ex>[lu] &{}
}
\end{equation}
Here the functor $F$, left adjoint to $G$, is the free 
$\cA$-module functor.
The functor $F'$, left adjoint to $G'$, is defined by the 
operadic left Kan extension with respect to the map of 
$\LM$-operads
$$\epsilon:\cM\sqcup X\to(\cM,\bar\cA).$$

The equivalence $G'=G\circ y^*$ gives rise to a morphism
of functors $\eta:F\to y^*\circ F'$.  
Here is the main result of this section.
\begin{prp}
\label{prp:y-equivalence}
The functor $y^*$ defined above is an equivalence.
\end{prp}
\begin{proof}
According to~\cite{L.HA}, 4.7.3.16, we have to verify the
following conditions.
\begin{itemize}
\item[1.] The functors $G'$ and $G$ preserve geometric realizations. 
\item[2.] The functors $G$ and $G'$ are conservative.
\item[3.] $\eta(\phi)$ is an equivalence for any 
$\phi\in\Fun(X,\cB)$.
\end{itemize}

The functor $G$ is conservative by \cite{L.HA}, 3.2.2.6 and 
preserves colimits by \cite{L.HA}, 4.2.3.5.
The functor $G'$  is conservative (by \cite{L.HA}, 3.2.2.6) 
and preserves geometric realizations by~\cite{L.HA}, 
3.2.3.1.

It remains to verify that 
$\eta(\phi):F(\phi)\to y^*\circ F'(\phi)$ is an equivalence 
for any $\phi:X\to\cB$. We will do so by verifying that the 
unit of the adjunction
\begin{equation}
\label{eq:a}
a:\phi\to G'\circ F'(\phi)
\end{equation}
satisfies the condition of~\ref{lem:Afree} with
$\Phi=G'\circ F'(\phi)$.

The map (\ref{eq:a}) defines, for any 
$x\in X$, $\Phi(x)$ as an operadic colimit, see~\cite{L.HA}, 3.1.1.20, 3.1.3.5, which we are now going to describe.  

To shorten the notation, we denote $\fX=\cM\sqcup X$
and $\fA=(\cM,\bar\cA)$, both considered as $\LM$-operads.
Similarly, $\fP$ will denote the $\LM$-operad 
$(\cM,P_\cM(\cA))$.
 We denote
$\cF_x=\fX\times_\fA\fA^\act_{/x}$ and define
the functor
$$\bar\phi_x:\cF_x\to(\cM,\cB)$$
as the composition of the projection $\cF_x\to\fX$ and the
map $\fX\to(\cM,\cB)$ induced by $\phi$. 

The map (\ref{eq:a}) defines $\Phi(x)$ as the operadic colimit of $\bar\phi_x$.

In~\ref{sss:taux} and ~\ref{lem:composition} 
below we define a functor 
$\tau_x:\Tw(X)^\op\to\cF_x$
and prove that   
$\theta^{\phi}_x:\Tw(X)^\op\to(\cM,\cB)$ 
factors as $\theta^\phi_x=\bar\phi_x\circ \tau_x$.

Then we prove (see~\ref{lem:cofinal}) that $\tau_x$ is cofinal.
This implies that the diagram $\bar\theta_x^a$
is an operadic colimit  diagram. This, by
\ref{lem:Afree},  proves the assertion.
\end{proof}

\subsubsection{Construction of $\tau_x$, a general idea}

Here is how $\tau_x$ looks like. An object $f\in\cF_x$ is given by a collection
$(m_1,\ldots,m_n,z,\beta)$ where $m_i\in \cM$, $z\in X$
and $\beta:(m_1,\ldots,m_n,z)\to x$ is an arrow
in $\cQ$ over an active arrow $(a^nm)\to m$ in $\LM$.
Note that an arrow $\beta$ can be equivalently described
(by the Yoneda lemma, see~\cite{H.EY}, Sect.~6) by an arrow $\otimes m_i\to\cA(z,x)$ in $\cM$.

The functor $\tau_x:\Tw(X)^\op\to\cF_x$ will carry
an arrow $\alpha:z\to y$
to
$$\tau_x(\alpha)=(\cA(y,x),z,\alpha^*:\cA(y,x)\to\cA(z,x)).$$

We present below a more accurate description of $\tau_x$.
 
\subsubsection{Construction of $\tau_x$}
\label{sss:taux}

The functor $\tau_x$ is defined by its components
$\tau^\fX:\Tw(X)^\op\to\fX$ and
$\tau^\fA:\Tw(X)^\op\to\fA^\act_{/x}$
and an equivalence of their compositions
$\Tw(X)^\op\to\fA$. The functor $\tau^\fX$ is the composition
\begin{equation}
\label{eq:taup}
\Tw(X)^\op\to X\times X^\op\to\cM\times X,
\end{equation}
where the second map carries $(z,y)$ to $(\cA(y,x),z)$.

Since $\bar\cA$ is a full subcategory of $P:=P_\cM(\cA)$,
$\fA^\act_{/x}$ is a full subcategory of $\fP^\act_{/Y(x)}$. The functor $\tau^\fA$ is therefore uniquely defined
by a functor $\Tw(X)^\op\to\fP^\act_{/Y(x)}$
whose composition with the forgetful functor to $\fP$
is given by~(\ref{eq:taup}).

Since $\fP$ is $\LM$-monoidal,  $\tau^\fA$ is 
determined by a functor 
$\Tw(X)^\op\to P_{/Y(x)}$
assigning to $\alpha:z\to y$ an arrow
$\cA(y,x)\otimes Y(z)\to Y(x)$ in $P$.

The right fibration $p:\Tw(X)^\op\to X^\op\times X$ carrying $\alpha:z\to y$ to the pair $(y,z)$, is classified by the functor $h:X\times X^\op\to\cS$
given by the formula $h(z,y)=\Map_X(z,y)$.

One has a functor $X^\op\times X\to P$ carrying $(y,z)$
to $\cA(y,x)\otimes Y(z)$. In order to lift it to a functor $\Tw(X)^\op\to P_{/Y(x)}$, it is enough to 
present a morphism of functors $\Map_X(z,y)\to\Map_\cM(\cA(y,x),\cA(z,x))$. This latter comes from functoriality of $\cA$.

We have the following.

\begin{lem}
\label{lem:composition}
$\theta^\phi_x=\bar\phi_x\circ \tau_x$.
\end{lem}
\begin{proof}
The functor  
$\theta^{\phi}_x:\Tw(X)^\op\to(\cM,\cB)$ 
factors as $\theta^\phi_x=\bar\phi_x\circ \tau_x$
as $\bar\phi_x$ factors through $\fX$, so that the composition $\bar\phi_x\circ \tau_x$ can be expressed
through $\tau^\fX$ given by the formula (\ref{eq:taup}).

\end{proof}

\begin{lem}
\label{lem:cofinal}
The functor $\tau_x:\Tw(X)^\op\to\cF_x$ is cofinal.
\end{lem}
\begin{proof}
We use Quillen's Theorem A, see~\cite{L.T}, 4.1.3.1.

We claim that the comma category
\begin{equation}
\label{eq:fibercategory}
\Tw(X)^\op\times_{\cF_x}(\cF_x)_{f/}
\end{equation}
has a terminal object for any $f\in\cF_x$. In fact, let
$f=(m_1,\ldots,m_n,z,u)$ where $u:\otimes m_i\to\cA(z,x)$
is an arrow in $\cM$. The terminal object of (\ref{eq:fibercategory}) is given by $\tau_x(\id_z)=
(\cA(z,x),z,\id_{\cA(z,x)})$.
\end{proof}

\subsection{Enriched categories and Lurie enriched categories}

As an easy consequence of the above, we have the following.

\begin{crl}
\label{crl:ff}
The functor $\lambda:\Cat(\cM)\to\Cat^\Lur(\cM)$
is fully faithful.
\end{crl}
\begin{proof}
Let $\cA,\cA'\in\Cat(\cM)$. 
The map
$$
\Map_{\Cat(\cM)}(\cA,\cA')\to\Fun_\cM(\cA,P_\cM(\cA'))^\eq
$$
is embedding, identifying the left-hand side with the
subspace of the right-hand side consisting of
$f:\cA\to P_\cM(\cA')$ with representable images. In other words, it induces an equivalence
$$
\Map_{\Cat(\cM)}(\cA,\cA')\to\Fun_\cM(\cA,\bar\cA')^\eq.
$$
According to the theorem, one has an equivalence
$$
\Fun_\cM(\cA,P_\cM(\cA'))\to
\Fun_{\LMod^w_\cM}(\cA,P_\cM(\cA'))
$$
which identifies $\Map_{\Cat(\cM)}(\cA,\cA')$
with $\Map_{\LMod^w_\cM}(\bar\cA,\bar\cA')$.
\end{proof}

Note that H.~Heine has recently proven ~\cite{HH} that $\lambda$
is an equivalence.

\begin{rem}
\label{rem:lambdaequivalence}
Let $\cB$ be weakly enriched over a monoidal category $\cM$. Here is a reasonable way to assign to $\cB$ an $\cM$-enriched category $\cA$. Let $X=\cB^\eq$ and let
$i:X\to\cB$ be the natural embedding. The $\LM$-operad $\Quiv^\LM_X(\cM,\cB)$ defines a weak enrichment of $\Fun(X,\cB)$ over $\Quiv_X(\cM)$.

The $\cM$-enriched category $\cA$ can now be defined as the endomorphism
object of $i\in\Fun(X,\cB)$ (if it exists).  
Apparently, this is precisely how ~\cite{HH} proves that
$\lambda$ is an equivalence.
\end{rem}
\subsubsection{}
The forgetful functor
$p:\Alg_\LM(\Cat)\to\Alg_\Ass(\Cat)$ is a cartesian fibration.
In particular, given a monoidal functor $f:\cA\to\cB$ and a category 
$\cX$ left-tensored over $\cB$, we have an $\LM$-monoidal cartesian 
lifting $f^!:(\cA,\cX)\to(\cB,\cX)$ in $\Alg_\LM(\Cat)$. 

\begin{Lem}
\label{lem:cartesianarrow}
The arrow $f^!:(\cA,\cX)\to(\cB,\cX)$ is also $p'$-cartesian, where
$$ p':\Op_\LM\to\Op_\Ass$$
is the forgetful functor.
\end{Lem}
\begin{proof}
The embedding $\Alg_\LM(\Cat)\to\Op_\LM$ has a left adjoint functor
denoted $P_\LM$ (monoidal envelope functor). Similarly,
$P_\Ass:\Op_\Ass\to\Alg_\Ass(\Cat)$ is left adjoint to the embedding.
The lemma immediately follows from the equivalence
$$P_\LM(\cO)_a=P_\Ass(\cO_a)$$
valid for any $\cO\in\Op_\LM$.
\end{proof}

\section{Relative tensor product and duality}
\label{sec:rtp}

\subsection{Introduction}
This section is mostly an exposition of (parts of) Lurie's \cite{L.HA}, 4.6 and 3.1. In it we do the following.
\begin{itemize}
\item[1.] Starting with a monoidal category $\cC$ with geometric realizations, we construct a $2$-category $\BMOD(\cC)$
~\footnote{More precisely, a category object in $\Cat$, that is, a simplicial object in $\Cat$ satisfying the Segal condition.} 
called the Morita $2$-category of $\cC$, whose objects are associative algebras in $\cC$, so that 
the category of morphisms $\Fun(A,B)$ is the category of 
$A$-$B$-bimodules. Composition of arrows in $\BMOD(\cC)$ 
is given by the relative tensor product of bimodules. 
The description of $\BMOD(\cC)$ is based on a study of an 
operad $\TENS$ over $\Ass^\otimes_{\Delta^\op}$, see~\cite{H.EY}, 
2.10.5 (3) \footnote{Lurie \cite{L.HA} describes it as a family
of operads based on $\Delta^\op$.} describing  collections of 
bimodules and multilinear maps between them.
\item[2.] Duality for bimodules describes adjunction 
between morphisms in $\BMOD(\cC)$. The special case, 
when the unit and the counit of the adjunction are 
equivalences, describes a Morita equivalence between the 
corresponding associative algebras in $\cC$.
\item[3.] A more general type of the relative tensor 
product of bimodules, with different bimodules belonging
to different categories, is described
using the same operad  $\TENS$ and the ones obtained from it by a base change.
\end{itemize}

\subsection{Morita $2$-category and the operad $\TENS$}

We now present a construction of the Morita $2$-category
$\BMOD(\cC)$ for a monoidal category $\cC$ with geometric realizations.
We describe $\BMOD(\cC)$ as a Segal simplicial object
in $\Cat$, carrying $[n]$ to a category
$\BMOD_n(\cC)$. The category $\BMOD_n(\cC)$ can be 
described as the category of algebras in $\cC$ over a 
certain planar operad (in sets).

We will denote this operad $\TENS_n$. Algebras over it
are collections $(A_0,\ldots,A_n)$ of associative algebras
in $\cC$, together with a collection of $A_{i-1}$-$A_i$-bimodules
$M_i$ for $i=1,\ldots,n$. So, $\BMOD_n(\cC)=\Alg_{\TENS_n}(\cC)$. To define the simplicial object
$\BMOD_\bullet(\cC)$, we have to provide a compatible collection of functors $s^*:\BMOD_n(\cC)\to\BMOD_m(\cC)$ defined for each $s:[m]\to[n]$ in $\Delta$, together with
the coherence data.

For a map $s:[m]\to[n]$ the functor
$s^*:\BMOD_n(\cC)\to\BMOD_m(\cC)$
comes from a correspondence between the operads $\TENS_m$ and $\TENS_n$.
 
We will present an operad
$\TENS_s$ over $\Ass^\otimes_{[1]}$ with fibers $\TENS_m$ and 
$\TENS_n$ over $0$ and $1$ respectively.
\footnote{Recall that $\Ass^\otimes$ denotes the operad governing 
associative algebras. The operad $\Ass^\otimes_K$ for $K\in\Cat$ governs 
$K$-diagrams of associative algebras.}

The functors $i^*_0:\Alg_{\TENS_s}(\cC)\to\BMOD_m(\cC)$
and $i^*_1:\Alg_{\TENS_s}(\cC)\to\BMOD_n(\cC)$ are defined by the embeddings $i_0:\TENS_m\to\TENS_s$ and 
$i_1:\TENS_n\to\TENS_s$.

The map $s^*:\BMOD_n(\cC)\to\BMOD_m(\cC)$
will be defined as the composition
$i^*_0\circ i_{1!}$ where $i_{1!}$ is left adjoint
to $i^*_1$.

In order to describe the compatibility of $s^*$ with respect to composition, we will define an operad $\TENS$ over 
$\Ass^\otimes_{\Delta^\op}=\Ass^\otimes\times\Com_{\Delta^\op}$. We will have
$\TENS_s=\TENS\times_{\Ass^\otimes_{\Delta^\op}}\Ass^\otimes_{[1]}$ where 
$s:[1]\to\Delta^\op$ defines the map
$\Ass^\otimes_{[1]}\to\Ass^\otimes_{\Delta^\op}$.

\subsubsection{The operad $\TENS$}
\label{sss:TENS}
$\TENS$ is the operad in sets governing the following 
collection of data.

\begin{itemize}
\item[1.] For each $n\geq 0$ the collection of monoids
$A_{0,n},\ldots,A_{n,n}$ and $A_{i-1}$-$A_i$ bimodules
$M_{i,n}$ for $i=1,\ldots,n$.
\item[2.] For each map $s:[m]\to[n]$ in $\Delta$ the collection of arrows:
\begin{itemize}
\item[a.] Morphism of algebras $A_{s(i),n}\to A_{i,m}$, for $i=0,\ldots,m$.
\item[b.] Multilinear morphisms (see remark below)
$$M_{s(i-1)+1,n}\times\ldots\times M_{s(i),n}\to M_{i,m}$$
of $A_{s(i-1),n}$-$A_{s(i),n}$-bimodules.
\end{itemize}
\item[3.] The collections of arrows defined in (2) for
each $s:[m]\to[n]$ compose in an obvious way.
\end{itemize}
\begin{Rem}Multilinearity in the last sentence means that, in case
$s(i-1)+1<s(i)$, the map is compatible with the actions of
all intermidiate $A_{j,n}$, $j=s(i-1)+2,\ldots,s(i)-1$; 
it means nothing if $s(i)=s(i-1)+1$; and it means an
$A_{s(i),n}$-bimodule map $A_{s(i),n}\to M_{i,m}$ if
$s(i-1)=s(i)$.
\end{Rem}

\subsubsection{The map to $\Ass^\otimes_{\Delta^\op}$}

An $\Ass^\otimes_{\Delta^\op}$-algebra in an $\Ass^\otimes$-operad $\cC$
is given by a functor $A:\Delta^\op\to\Alg_\Ass(\cC)$. 
This functor defines a canonical $\TENS$-algebra defined by the formulas $A_{i,n}=A([n])=M_{i,n}$. This gives a functor
$\pi^*:\Alg_{\Ass_{\Delta^\op}}(\cC)\to\Alg_\TENS(\cC)$
realized as the inverse image with respect to the map
$$\pi:\TENS\to\Ass^\otimes_{\Delta^\op}.$$

\subsubsection{}
\label{sss:BMOD}
For any $\phi:S\to\Delta^\op$ one defines
$\TENS_S$ (or $\TENS_\phi$) as $\Com_S\times_{\Com_{\Delta^\op}}\TENS$.

One defines $p:\BMOD(\cC)\to\Delta^\op$ as a category 
over $\Delta^\op$ representing the functor
\begin{equation}
\Fun_{\Delta^\op}(B,\BMOD(\cC))=
\Alg_{\TENS_B}(\cC).
\end{equation}

In the case when $\cC$ has geometric realizations and a monoidal
structure preserving geometric realizations, $\BMOD(\cC)$
is a cocartesian fibration over $\Delta^\op$, so it defines
a simplicial object $\BMOD_\bullet(\cC)$ in $\Cat$,
see \cite{L.HA}, 4.4.3.12. It satisfies the Segal condition
by \cite{L.HA}, 4.4.3.11.

\begin{rem}
Note that $\BMOD(\cC)$ is not complete. The zero component
$\BMOD_0(\cC)$ is the category of algebras in $\cC$ which is not a space.
An equivalence defined by $\BMOD_1(\cC)$ is a Morita equivalence which is 
not equivalence in $\BMOD_0(\cC)$.
\end{rem}

\subsection{Duality}

In this subsection we apply the general notion of
adjunction in a 2-category to the Morita 2-category 
described in the previous subsection.

\begin{dfn}(see \cite{L.HA}, 4.6.2.3)
Let $\cC$ be a monoidal category with geometric realizations. Let $A,B$ be two associative algebras in $\cC$, $M\in_A\!\!\BMod_B(\cC)$ and $N\in_B\!\!\BMod_A(\cC)$.
A map $c:B\to N\otimes_AM$ is said to exhibit $N$ as
left dual of  $M$ (or $M$ as a right dual of $N$) if
there exists $e:M\otimes_BN\to A$ in $_A\BMod_A(\cC)$
such that the compositions
$$
M= M\otimes_BB\stackrel{\id_M\otimes c}{\to} M\otimes_BN\otimes_AM\stackrel{e\otimes\id_M}{\to}M
$$
and
$$
N= B\otimes_BN\stackrel{c\otimes\id_N}{\to} N\otimes_AM\otimes_BN\stackrel{\id_N\otimes e}{\to}N
$$
are equivalent to $\id_M$ and $\id_N$, respectively.
\end{dfn}
 
Let $\cM$ be a left $\cC$-tensored category with geometric realizations.
A dual pair of bimodules $M\in_A\!\!\BMod_B(\cC)$
and $N\in_B\!\!\BMod_A(\cC)$ determines an adjunction
(see [L.HA], 4.6.2.1)
\begin{equation}
F:\LMod_B(\cM)\rlarrows\LMod_A(\cM):G
\end{equation} 
given by the formulas $F(X)=M\otimes_BX$
and $G(Y)=N\otimes_AY$. This adjunction deserves the name
{\sl Morita adjunction}. 

A Morita adjunction is called a Morita equivalence if the
arrows $c$ and $e$ are equivalences.

Two properties of Morita adjunction are listed below.
The first one, Proposition~\ref{prp:composition},
describes a good behavior of Morita adjunctions under composition. The second one, Proposition~\ref{prp:noB},
claims that the left dualizability of 
$M\in_A\!\!\BMod_B(\cC)$ is independent of the algebra $B$.

\begin{prp}(see~\cite{L.HA}, 4.6.2.6)
\label{prp:composition}
let $\cC$ be a monoidal category with geometric realizations, $A,B,C$ three associative algebras in $\cC$.
If $c:B\to N\otimes_AM$ exhibits $N$ as a left dual to 
$M\in_A\!\!\BMod_B$ and $c':C\to N'\otimes_BM'$
exhibits $N'$ as a left dual to $M'\in_B\!\!\BMod_C$ then the composition
$$ C\stackrel{c'}{\to} N'\otimes_BM'=N'\otimes_BB\otimes_BM'\stackrel{c}{\to} 
N'\otimes_B N\otimes_AM\otimes_BM'$$
exhibits $N'\otimes_BN$ as a left dual to $M\otimes_BM'$.
\end{prp}

\begin{prp}(see~\cite{L.HA}, 4.6.2.12, 4.6.2.13)
\label{prp:noB}
Let $\cC$ be as above. A bimodule $M\in_A\!\!\BMod_B(\cC)$
is left dualizable if and only if its image $M'$ in 
\newline
$\LMod_A(\cC)= _A\!\!\BMod_\one(\cC)$ is left dualizable. 
Moreover, if $N\in _B\!\!\BMod_A(\cC)$ is left dual to $M$, its image in $\RMod_A(\cC)$ is a left dual of $M'\in\LMod_A(\cC)$.
\end{prp}\qed

\begin{rem}
In the classical context of associative rings, an
$(A,B)$-bimodule $N$ is right-dualizable iff it is finitely
generated projective as a right $A$-module. This property is 
independent of $B$ and right dualizability of $N$  is sufficient to 
have an adjunction between the categories of left $A$ and $B$-
modules. This adjunction is an equivalence, for $B=\End_A(N)$, if
$N$ is a generator in $\RMod_A$. It would be very nice
to describe in our general context a condition on a right dualizable 
module $N\in\RMod_A$ leading to Morita equivalence. 
\end{rem}

\subsubsection{}

We can fix a right-dualizable $A$-module $N\in\RMod_A(\cC)$
and try to reconstruct a would-be Morita equivalence.

Let $M\in\LMod_A(\cC)$ be the right dual of $N$.

The category $\RMod_A(\cC)$ is left-tensored over $\cC$.
So, given $N\in\RMod_A(\cC)$, one can define an  
endomorphism object $\End_A(N)$ which, if it exists, 
acquires an associative algebra structure. Since $N$ is right dualizable, this object does exist, as one has a canonical
equivalence
\begin{equation}
\Map_\cC(X,N\otimes_AM)=\Map_{\RMod_A(\cC)}(X\otimes N,N)
\end{equation}
by~\cite{L.HA}, 4.6.2.1 (3), so that $\End_A(N)=N\otimes_AM$ as an object of $\cC$.

\begin{crl}
\label{crl:morita-semi}
Let $\cC$ be a monoidal category with geometric realizations, $A$ an associative algebra in $\cC$, $N\in\RMod_A(\cC)$ a right dualizable $A$-module. Then $M\in\LMod_A(\cC)$, the right dual of $N$, has a canonical structure of  $A$-$\End_A(N)$-bimodule and the pair $(M,N)$ defines a Morita adjunction
$$F:\LMod_{\End_A(N)}(\cC)\rlarrows\LMod_A(\cC),$$
with $F(X)=M\otimes_{\End_A(N)}X$ and
$G(Y)=N\otimes_AY$, for which the coevaluation
$$ c:\End_A(N)\to N\otimes_AM$$
is an equivalence.
\end{crl}
\begin{proof}
See ~\cite{L.HA}, 4.6.2.1 (2).
\end{proof}

Note  that this construction produces the $A$-$A$-bimodule 
evaluation map  
\begin{equation}
\label{eq:eval}
e:M\otimes_{\End_A(N)}N\to A.
\end{equation} 

\begin{rem}
\label{rem:viam}
The algebra $B=\End_A(N)$ can be also described in terms of
the left $A$-module $M$. In fact, the category 
$\LMod_A(\cC)$ is right-tensored over $\cC$, so it is left-tensored over the reversed monoidal category 
$\cC^\rev$. The endomorphism object of $M\in\LMod_A(\cC)$
in $\cC^\rev$ exists, and it coincides with the algebra
$B^\op$.
\end{rem}

\subsection{Relative tensor product}
\label{ss:relativetensor}

The relative tensor product of bimodules with values
in a monoidal category $\cC$ is encoded in the composition
of arrows of $\BMOD(\cC)$. There exists a slightly
more general relative tensor product, for the bimodules
having values in different categories.

Let now $\cC\in\Alg_{\TENS_S}(\Cat^L)$ where, as before, $\Cat^L$ 
denotes the category of categories with small colimits, with the arrows
being the colimit preserving functors. 

We wish to study tensor  product of bimodules with values 
in $\cC$.  

We define, slightly generalizing \ref{sss:BMOD},  
$p:\BMOD^\phi(\cC)\to S$ as a category over $S$ representing the functor
$$
\Fun_S(B,\BMOD^\phi(\cC))=
\Alg_{\TENS_B/\TENS_S}(\cC).
$$

One has
\begin{prp} 
\label{prp:S-family}
\begin{itemize}
\item[1.]The map $p:\BMOD^\phi(\cC)\to S$  is
a cocartesian fibration. 
\item[2.]An arrow $\tilde\alpha$ in 
$\BMOD^\phi(\cC)$ over $\alpha:x\to y$  in $S$  
 is $p$-cocartesian iff the corresponding
$F_\alpha\in\Alg_{\TENS_{\phi(\alpha)}/\TENS_S}(\cC)$ is 
an operadic left Kan extension of its restriction
$F_{\phi(x)}:\TENS_{\phi(x)}\to\cC$.
\item[3.] Let $f:\cC\to\cD$ be a $\TENS_\phi$-monoidal functor preserving geometric realizations. Then the induced
map $\BMOD^\phi(f):\BMOD^\phi(\cC)\to\BMOD^\phi(\cD)$
preserves cocartesian arrows.
\end{itemize}
\end{prp}
\begin{proof}
The first two claims are just \cite{L.HA}, Corollary 4.4.3.2, with
$\cO=\TENS_S$. The condition (*) is fulfilled as
$\cC$ is $\TENS_S$-monoidal category with geometric realizations, commuting with the tensor product.
Claim 3 follows from Claim 2.
\end{proof}

\subsubsection{}
Here is an important example of the above construction.

Let $\succ:[1]\to\Delta^\op$ be defined by the arrow
$\partial^1:[1]\to[2]$ in $\Delta$. We have then
$\TENS_\succ=\Com_{[1]}\times_{\Com_{\Delta^\op}}\TENS$.

One has natural embeddings 
$i_1:\TENS_1\to\TENS_\succ$ and 
$i_2:\TENS_2\to\TENS_\succ$ induced by the
embedding of the ends $\{1\}\to[1]$ and $\{0\}\to[1]$.

Note that $\TENS_1=\BM^\otimes$ and $\TENS_2=\BM^\otimes\sqcup^{\Ass^\otimes}\BM^\otimes$.

\subsubsection{}
\label{sss:tens-succ}

Let $\cC$ be a $\TENS_\succ$-monoidal category.
Up to equivalence, $\cC$ is uniquely described by a collection of five monoidal categories
$\cC_a,\cC_b,\cC_c,\cC_{a'},\cC_{c'}$,
three bimodule categories $\cC_m
\in_{\cC_a}\!\!\BMod_{\cC_b}(\Cat)$, 
$\cC_n\in_{\cC_b}\!\!\BMod_{\cC_c}(\Cat)$, 
$\cC_k\in_{\cC_{a'}}\!\!\BMod_{\cC_{c'}}(\Cat)$, 
monoidal functors $\phi_a:\cC_a\to\cC_{a'}$ and 
$\phi_c:\cC_c\to\cC_{c'}$, and a $\cC_b$-bilinear functor
$$ \cC_m\times\cC_n\to\cC_k$$
of $\cC_a$-$\cC_c$-bimodule categories.

The embedding $i_2:\TENS_2\to\TENS_\succ$ 
induces
$$i_2^*:\Alg_{\TENS_\succ}(\cC)\to
\Alg_{\TENS_2/\TENS_\succ}(\cC).$$
The relative tensor product functor 
\begin{equation}
\label{eq:relativetensor}
\RT:\Alg_{\TENS_2/\TENS_\succ}(\cC)\to
\Alg_{\TENS_\succ}(\cC)
\end{equation}
is defined as the functor left adjoint to $i_2^*$.

The functor $\RT$ exists if $\cC\in\Alg_{\TENS_\succ}(\Cat^L)$.

It makes sense to fix associative algebras
$A\in\Alg_\Ass(\cC_a)$, $B\in\Alg_\Ass(\cC_b)$,
$C\in\Alg_\Ass(\cC_c)$,  and restrict (\ref{eq:relativetensor}) to 
$\TENS_2$-algebras in $\cC$
having algebra-components $A,B$ and $C$. 
If $A'=\phi_a(A)\in\Alg_\Ass(\cC_{a'})$
and $C'=\phi_c(C)\in\Alg_\Ass(\cC_{c'})$, this gives
\begin{equation}
\label{eq:relativetensor2}
\RT_{A,B,C}:_A\!\!\BMod_B(\cC_m)\times_B\!\!\BMod_C(\cC_n)\to
_{A'}\!\!\BMod_{C'}(\cC_k).
\end{equation}

\subsubsection{Two-sided bar construction}
\label{sss:bar}

The following explicit formula for the calculation
of relative tensor product explains why does it exist
for categories with geometric realizations.

Recall that $\TENS_2$ governs 5-tuples of objects, $(A,M,B,N,C)$,
where $A,B,C$ are associative monoids, $M$ is an $A$-$B$-bimodule
and $N$ is a $B$-$C$-bimodule. We denote the colors of
$\TENS_2$ by $a,m,b,n,c$. The operad $\TENS_1$ has colors
$a',k,c'$.
 
Define a functor $u:\Delta^\op\to\TENS_2$ carrying 
$[i]$ to $mb^in\in\TENS_2$, where the action of $u$
on the arrows is defined as follows.
\begin{itemize}
\item Faces correspond to the action maps $mb\to m$, $bb\to b$
or $bn\to n$.
\item Degeneracies correspond to the unit maps $1\to b$.
\end{itemize}

We extend the map $u:\Delta^\op\to\TENS_2$
to $u_+:\Delta_+^\op\to\TENS_\succ$ carrying the
terminal object of $\Delta_+^\op$ to $k\in\TENS_1$.  

Let $q:\cC\to\TENS_\succ$ present a $\TENS_\succ$-monoidal 
category. The map 
$\Fun(\Delta^\op,q):\Fun(\Delta^\op,\cC)\to\Fun(\Delta^\op,\TENS_\succ)$ is a cocartesian fibration.
The functor $u_+$ defines an arrow $\beta:u\to u_*$
in $\Fun(\Delta^\op,\TENS_\succ)$, $u_*$ being the constant
functor with the value $k\in\TENS_\succ$. Therefore, any
$\phi\in\Alg_{\TENS_2/\TENS_\succ}(\cC)$ gives rise to a unique lift 
$\beta_!:\phi\circ u\to X$, where $X$
is a simplicial object in $\cC_k$. We will denote $X=\Sh(\phi\circ u)$ and call it the 
two-sided bar construction, $\Sh(\phi\circ u)=\BAR(\phi)$.

The following explicit description of relative tensor product is a reformulation of~\cite{L.HA}, 4.4.2.8.

\begin{prp} 
\label{prp:colimitdescription}
Let $\cC$ be a $\TENS_\succ$-monoidal category with geometric realizations and the tensor structure commuting
with the geometric realizations, and let
$q:\cC\to\TENS_\succ$ be the corresponding cocartesian fibration.
Given a commutative diagram
\begin{equation}
\label{eq:tensorproduct}
\xymatrix{
&\TENS_2\ar^\phi[r]\ar[d] &\cC\ar^q[d]\\
&\TENS_\succ\ar@{=}[r]\ar@{-->}^\Phi[ru]&\TENS_\succ
}
\end{equation}
of marked categories, with $\phi$ corresponding to a pair
of bimodules $M\in_A\!\!\BMod_B$ and 
$N\in_B\!\!\BMod_C$ with values in $\cC$.  Then there exists $\Phi$ presenting  a relative tensor product of $M$ and $N$.  Vice versa, any 
extension $\Phi$
of $\phi$ presents a relative tensor product of $M$ with $N$if and only if the following conditions are fulfilled.
\begin{itemize}
\item $\Phi$ carries the maps $A_{02}\to A_{01}$ and
$A_{22}\to A_{11}$ to $q$-cocartesian arrows in $\cC$.
\item The functor $\Phi$ induces an equivalence
$$|\BAR(\phi)|\to \Phi(k).$$
\end{itemize}
\end{prp}
 
\subsubsection{Associativity}
\label{sss:asso}

To formulate associativity, we need to use
Proposition~\ref{prp:S-family} applied to the family 
$\phi:S\to\Delta^\op$ defined by the commutative square
\begin{equation}
\xymatrix{
&{[1]}\ar^{\partial^1}[d]\ar^{\partial^1}[r] &{[2]}\ar^{\partial^2}[d] \\
&{[2]}\ar^{\partial^1}[r] &{[3]}
}
\end{equation}
in $\Delta$, and a $\TENS_\phi$-monoidal category $\cC$ 
with geometric realizations, see~\cite{L.HA}, 4.4.3.14. 

\subsection{Variants}
The tensor product of bimodules (\ref{eq:relativetensor2}) commutes with the functor forgetting the left $A$-module structure
and the right $C$-module structure.

We would like to formulate this observation as follows.
Let $\TEN_\succ$ be the full suboperad of $\TENS_\succ$ spanned by the colors $a,b,a', m,n,k\in[\TENS_\succ]$. There is an obvious
embedding $i:\TEN_\succ\to\TENS_\succ$ and the functor
$i^*:\Alg_{\TENS_\succ}\to\Alg_{\TEN_\succ}$ forgets the 
right module structure on the bimodules described by the colors $n$ and $k$.

Similarly, it makes sense to describe a yet smaller suboperad 
$\EN_\succ$ spanned by the colors $b,m,n,k\in[\TENS_\succ]$.
We denote $j:\EN_\succ\to\TENS_\succ$ the obvious embedding 
that forgets both the right module structure on bimodules described by $n,k$ and the left module structure on bimodules described by
$m,k$.

We define $\TEN_2$ and $\EN_2$ as for $\TENS_\succ$; this yields
the functors $i^*_2$ and the left adjoints $\RT$ exactly as for 
$\TENS_\succ$-monoidal categories. One has

\begin{prp}
The forgetful functors $i^*$ and $j^*$ commute with the relative tensor product.
\end{prp}
\begin{proof}
The commutative square
\begin{equation}
\xymatrix{
&\Alg_{\TENS_\succ}(\cC)\ar^{i^*}[d]\ar^{i_2^*}[r]
&\Alg_{\TENS_2/\TENS_\succ}(\cC)\ar^{i^*}[d]\\
&\Alg_{\TEN_\succ/\TENS_\succ}(\cC)\ar^{i_2^*}[r]
&\Alg_{\TEN_2/\TENS_\succ}(\cC)
}
\end{equation}
defines a morphism of functors 
$$ \RT\circ i^*\to i^*\circ\RT.$$
To prove that this functor is an equivalence, we use the description 
of $\RT$ in terms of the two-sided bar construction. The functor
$u_+:\Delta^\op\to\TENS_\succ$ factors through 
$i:\TEN_\succ\to\TENS_\succ$, so the bar construction used to calculate $\RT$ as a colimit, is the same for both setups.

The version for $j:\EN_\succ\to\TENS_\succ$ is proven in the same way.
\end{proof}

\subsection{Reduction}
\label{ss:reduction}
An $A$-$B$-bimodule  in $\cC$ can be 
equivalently described as a left $A$-module in the category 
$\RMod_B(\cC)$. 

We present below a similar transformation of $\TENS_\succ$-monoidal categories compatible with the formation of the weighted colimit.

The construction is based on the notion of bilinear map of operads
and their tensor product as presented in \cite{H.EY}, 2.10.

\subsubsection{}
We define a map $p:\TENS_\succ\to\BM^\otimes$ as the obvious map
carrying the colors
$a,a',b,m$ to $a\in[\BM]$, $n,k$ to $m\in[\BM]$ and $c,c'$
to $b\in[\BM]$.
We have $\TEN_\succ=\LM^\otimes\times_{\BM^\otimes}\TENS_\succ$.  

One has a standard bilinear map
$\Pr:\LM^\otimes\times\RM^\otimes\to\BM^\otimes$ defined in~\cite{L.HA}, 4.3.2.1 and \cite{H.EY}, 2.10.7. There is
a lifting of $\Pr$ to a bilinear map
\begin{equation}
\label{eq:mu}
\mu:\TEN_\succ\times\RM^\otimes\to\TENS_\succ
\end{equation}
uniquely defined by its action on the colors.

\begin{itemize}
\item $\mu(*,m)=*$ where $*$ is any color of $\TEN_\succ$.
\item $\mu(n,b)=c,\ \mu(k,b)=c'$.
\end{itemize}

\subsubsection{}
\label{sss:red-1}
Let $\cC$ be a $\TENS_\succ$-operad. Following a general pattern~\cite{H.EY}, 2.10.1, we define a  $\TEN_\succ$-operad 
$\cC^\red:=\Alg^{\mu}_{\RM/\TENS_\succ}(\cC)$ as the one
representing the functor
$$
K\in\Cat_{/\TEN_\succ}\mapsto\Map_{\Cat^+_{/\TENS_\succ^\natural}}
(K^\flat\times\RM^\natural,\cC^\natural).
$$
We call $\cC^\red$ the reduction of $\cC$. 

Here is a more convenient description of $\cC^\red$ in the case 
when $\cC$ is a $\TENS_\succ$-monoidal category. In this case 
$\cC$ is classified by a lax cartesian structure
$$ \wt\cC:\TENS_\succ\to\Cat.$$
Composing it with $\mu$, we get a functor
$$ \wt\cC\circ\mu:\TEN_\succ\times\RM^\otimes\to\Cat,$$
defining
\begin{equation}
\label{eq:olC}
\ol\cC:\TEN_\succ\to\Fun^\lax(\RM^\otimes,\Cat),
\end{equation}
that is, a functor with the values in $\RM$-monoidal categories.
Composing it with the functor $\Alg_\RM$, we get a functor
$\wt\cC^\red:\TEN_\succ\to\Cat$ classifying $\cC^\red$.

Here is a more detailed information about the functor $\ol\cC$.
For $x=a,a',m,b$, $\ol\cC(x)$ is the $\RM$-monoidal category
$(\cC_x,[0])$ describing the trivial action of the trivial monoidal category on $\cC_x$. Thus, one has
$\cC^\red_x=\cC_x$ for these values of $x$. Furthermore, 
$\ol\cC(n)=(\cC_n,\cC_c)^\otimes$ and
$\ol\cC(k)=(\cC_n,\cC_{c'})^\otimes$, so that  
$\cC^\red_n=\Alg_\RM(\cC_n,\cC_c)$ and 
$\cC^\red_k=\Alg_\RM(\cC_k,\cC_{c'})$.

The standard embedding $i:\TEN_\succ\to\TENS_\succ$ identifies
the $m$-component of $\ol\cC$ with $i^*(\cC)$.
 
This defines a functor
$G:\cC^\red\to i^*(\cC)$ forgetting the right module structure
 on the components $\cC^\red_n,\cC^\red_k$.

The restriction with respect to $\mu$ (\ref{eq:mu})  defines a natural map
\begin{equation}
\label{eq:reduction}
\theta:\Alg_{\TENS_\succ}(\cC)\to\Alg_{\TEN_\succ}(\cC^\red),
\end{equation}
whose composition with the map induced by $G$
is the obvious restriction
\begin{equation}
\label{eq:res}
\Alg_{\TENS_\succ}(\cC)\to\Alg_{\TEN_\succ/\TENS_\succ}(\cC).
\end{equation}
We believe that the map (\ref{eq:reduction}) is an equivalence,
that is that $\mu$ presents $\TENS_\succ$ as a tensor product.

We will actually verify a somewhat weaker statement 
Proposition~\ref{prp:rt-reduction} that will be used in Section~\ref{sec:wc}.

\begin{lem}
\label{lem:tens2appr}
The bilinear map
$$\mu_2:\TEN_2\times\RM^\otimes\to\TENS_2$$
obtained by restriction of $\mu$, presents 
$\TENS_2$ as a tensor product of $\TEN_2$ with $\BM^\otimes$.
\end{lem}
\begin{proof}
We compose $\mu_2$ with the standard strong approximations
$\RM\to\RM^\otimes$, $\ten_2\to\TEN_2$ as described in 
\cite{H.EY}, 2.9, with $\ten_2=\BM\sqcup^{\Ass}\LM$.
We get a bilinear map $\ten_2\times\RM\to\TENS_2$ that is 
easily seen to be a strong approximation.
\end{proof}

\begin{prp}
\label{prp:rt-reduction}
Let $\cC$ be a $\TENS_\succ$-monoidal category with colimits. 
Then $\mu$ induces a commutative diagram
\begin{equation}
\label{eq:reduction-prp}
\xymatrix{
&{\Alg_{\TENS_2/\TENS_\succ}(\cC)}\ar^(0.5){\RT}[rr] 
\ar_{\theta_2}^\sim[d]&&{\Alg_{\TENS_\succ}(\cC)}\ar^\theta[d]\\
&{\Alg_{\TEN_2/\TEN_\succ}(\cC^\red)}\ar^(0.55){\RT^\red}[rr]&&{\Alg_{\TEN_\succ}(\cC^\red)},
}
\end{equation}
where $\RT^\red$ is the relative tensor product defined for
the $\TEN_\succ$-monoidal category $\cC^\red$.
\end{prp}
\begin{proof}
If $\cC$ is a $\TENS_\succ$-monoidal category with colimits,
$\cC^\red$ is a $\TEN_\succ$-monoidal category with colimits.
This implies that $\RT^\red$ is defined as the functor left adjoint to the restriction
$i^{\red*}_2:\Alg_{\TEN_\succ}(\cC^\red)\to\Alg_{\TEN_2/\TEN_\succ}(\cC^\red)$.
The equivalence $\theta_2$ is defined by the universal property of 
$\cC^\red=\Alg^\mu_{\RM/\TENS_\succ}(\cC)$ and Lemma~\ref{lem:tens2appr}.
The equivalence $\theta_2\circ i^*_2=i^{\red*}_2\circ\theta$
induces a morphism of functors
\begin{equation}
\label{eq:rttheta}
 \RT^\red\circ\theta_2\to\theta\circ\RT.
 \end{equation}
We claim that this morphism is an equivalence.

Let $\phi\in\Alg_{\TENS_2/\TENS_\succ}(\cC)$ be given by a pair
of bimodules $M\in_A\!\!\BMod_B(\cC_m),\ N\in_B\!\!\BMod_C(\cC_n)$
and let $\Phi=\RT(\phi)$. By~\ref{prp:colimitdescription}, the composition $\Phi\circ u_+:\Delta^\op_+\to\cC$ is an operadic colimit
diagram. This is equivalent to saying that the cocartesian shift
$\Sh(\Phi\circ u_+):\Delta^\op_+\to\cC_k$ is a colimit diagram.

The map $u_+:\Delta^\op_+\to\TENS_\succ$ factors
through $i:\TEN_\succ\to\TENS_\succ$.
By~\ref{prp:colimitdescription}, the claim of Proposition~\ref{prp:rt-reduction} will be proven once we
verify that $\theta(\Phi)\circ u_+$ is an operadic colimit 
diagram in $\cC^\red$, or, equivalently, that the cocartesian shift
$\Sh(\theta(\Phi)\circ u_+):\Delta^\op_+\to\cC^\red_k$
is a colimit diagram.

The composition $G\circ\Sh(\theta(\Phi)\circ u_+):\Delta^\op_+\to\cC_k$
is a  colimit diagram as the composition $G\circ\theta$ is the
restriction (\ref{eq:res}). According to \cite{L.HA}, 3.2.3.1,
$G$ creates colimits. This proves the proposition.
\end{proof}

\begin{rem}
\label{rem:red-c}
Let $C\in\Alg_\Ass(\cC_c)$ be the $c$-component of $\phi$.
We define a $\TEN_\succ$-monoidal subcategory $\cC^\red_C$  
of $\cC^\red$ as follows. The restriction 
$\Alg_\RM\to\Alg_{\Ass/\RM}$ applied to the functor 
$\ol\cC$ (\ref{eq:olC}) yields a morphism of functors
$$\wt\cC^\red\to\Alg_{\Ass/\RM}\circ\ol\cC.$$
The $\TEN_\succ$-monoidal subcategory $\cC^\red_C$
is defined as the fiber of this functor at the 
object of $\Alg_{\Ass/\RM}\circ\ol\cC$ determined by the cocartesian
arrow $C\to\phi_c(C)$, in the notation of \ref{sss:tens-succ}.
The functor $\theta(\Phi)\circ u_+:\Delta^\op_+\to\cC^\red$ 
canonically factors through $\cC^\red_C$. By \cite{L.HA}, 3.2.3.1,
the functor $\Delta^\op_+\to\cC^\red_C$ so defined is also
an operadic colimit diagram.

\end{rem}

\section{Bar resolutions for enriched presheaves}
\label{sec:bar}

The aim of this very technical section is to construct a certain
operadic colimit diagram, see~\ref{prp:barquiv}, used later in the
proof of the important result \ref{prp:olke}.

Given a monadic adjunction $\cC\rlarrows\LMod_A(\cC)$, any
$A$-module $M$ acquires a standard resolution $\BAR_\bullet(A,M)$
(sometimes called {\sl bar resolution}).
This is a  simplicial resolution of $M$ consisting of free $A$-modules.
If one forgets the $A$-module structure on $\BAR_\bullet(A,M)$, 
one will get a special case of the bar construction described 
in~\ref{sss:bar}. It should not surprise us that the $A$-module structure
that we have just forgotten, can be reconstructed from a monadic action.

Let $\cC=(\cC_a,\cC_m)$ be an $\LM$-monoidal category with colimits,
let $A$ be an associative algebra in $\cC_a$ and $M$ be an $A$-module in 
$\cC_m$. The pair $(A,M)$ is given by a map of operads
$\gamma:\LM\to\cC$. Its composition with a functor 
$u_+:\Delta^\op_+\to\LM$ (a variant of the $u_+$ defined in \ref{sss:bar}) defines an operadic colimit diagram; its 
cocartesian shift $\Sh(\gamma\circ u_+)$ is equivalent to $G(\BAR_\bullet(A,M))$,
where $\BAR(A,M)$ is the $G$-split simplicial objects defined by the monad $G\circ F$ on $\cC$ and  $G:\LMod_A(\cC_m)\to\cC_m$ is the forgetful functor.

In Subsection \ref{ss:barmodule} we explain how to reconstruct 
$\BAR_\bullet(A,M)$ from $\Sh(\gamma\circ u_+)$, in terms an 
action of the monad corresponding to $A$.

We use a similar reasoning to describe the bar resolution
for enriched presheaves in Subsection~\ref{ss:barpresheaves}.
Here we are able to say more than for general modules. In general, we have no chance to introduce the monad action on $\gamma\circ u_+$
instead of $\Sh(\gamma\circ u_+)$,
as the category $\LMod_A(\cC_m)$ is not left-tensored over $\cC_a$.

As for the enriched presheaves that are defined as $\LMod_{\cA^\op}(\Fun(X^\op,\cM))$, they have a left $\cM$-module structure. 
This allows us to encode the bar resolution for a presheaf 
$f\in P_\cM(\cA)$ into
an operadic colimit diagram
\begin{equation}
\label{eq:barpresheavesintro}
K^\triangleright\to(\cM,P_\cM(\cA))^\otimes,
\end{equation}
with an appropriate choice of a category $K$, see~\ref{lem:bar-pre}
for the precise formulation.

\subsection{Bar resolution of a module}
\label{ss:barmodule}

Let $\cO$ be an $\LM$-operad and let $\gamma:\LM\to\cO$ 
be an $\LM$-algebra in $\cO$ defined by a pair $(A,M)$,
where $A$ is an associative algebra in the planar operad
$\cO_a$ and $M\in\cO_m$ is a left $A$-module. 

A very special case of the two-sided bar construction~\ref{sss:bar}
gives the following simplicial resolution of a module.

We define the functor $u_+:\Delta_+^\op\to\LM$ by the formula
$$ 
u_+([n])=a^{n+1}m,\ 
n\geq -1,
$$
with the face maps defined by the multiplication in $a$ and by its action on $m$.
Note that the image of $u_+$ belongs to $\LM^\act$,
the active part of $\LM$.
\begin{lem}
\label{lem:bar-mod}
The composition 
$$\gamma\circ u_+:
\Delta^\op_+\to\cO
$$
is an operadic colimit diagram. In the case when $\cO$
is a monoidal $\LM$-category with geometric realizations,
it induces an equivalence
$|\Sh(\gamma\circ u)|\to M$ in 
$\cO_m$. 
\end{lem}
\begin{proof}
This is a direct consequence of \cite{L.HA}, 4.4.2.5, 4.4.2.8, applied to the tensor product $A\otimes_AM=M$.
\end{proof}

\subsubsection{}
\label{sss:bar-elev}

The  functor 
$\Sh(\gamma\circ u_+)$ has a canonical lifting
to an augmented simplicial object in $\LMod_A(\cO_m)$ that we call
{\sl the bar resolution of $A$-module $M$} and denote 
$\BAR_\bullet(A,M)$.
In \ref{sss:bar-elev}---\ref{sss:bar-elev-end}   we show how this canonical lifting can be described in terms of a monadic action.

\subsubsection{An action of $\cO_a$ on $\Fun(K,\cO_m)$}
\label{sss:klm}

An $\LM$-monoidal category $\cO=(\cO_a,\cO_m)$ encodes an action of 
a monoidal category $\cO_a$ on a category $\cO_m$, or, in other 
words, a monoidal functor $\cO_a\to\End_{\Cat}(\cO_m)$. 

Fix $K\in\Cat$. The functor $C\mapsto\Fun(K,C)$ defines a monoidal
functor $\End(C)\to\End(\Fun(K,C))$. Thus,  any $\LM$-monoidal category $\cO$ defines a monoidal functor $\cO_a\to\End(\cO_m)\to
\End(\Fun(K,\cO_m))$, that is an $\LM$-monoidal category 
$(\cO_a,\Fun(K,\cO_m))$.

We wish to  present two more  constructions of the $\LM$-monoidal
category \newline $(\cO_a,\Fun(K,\cO_m))$. 

1. Applying the functor $\Fun^\LM(K,\_)$ defined in \cite{H.EY},
6.1.6, we get an $\LM$-monoidal category with the $a$-component
$\Fun^\LM(K,\cO)_a=\Fun(K,\cO_a)$ and \newline
$\Fun^\LM(K,\cO)_m=\Fun(K,\cO_m)$. The forgetful functor
$\Alg_\LM(\Cat)\to\Alg_\Ass(\Cat)$ being a cartesian fibration,
the $\LM$-monoidal category $(\cO_a,\Fun(K,\cO_m))$
described above, is equivalent to  
$i^*(\Fun^\LM(K,\cO))$, where $i:\cO_a\to\Fun(K,\cO_a)$ is induced by the
map $K\to[0]$.

2. Let $K\in\Cat$. Denote by $K^\LM$  the $\LM$-monoidal
category describing the action of the terminal monoidal category 
$[0]$ on $K$.  We claim that $\Funop_\LM(K^\LM,\cO)$ gives yet 
another 
presentation of $(\cO_a,\Fun(K,\cO_m))$.

We start with the map of cocartesian fibrations over $\LM$
\begin{equation}
\label{eq:toklm}
q:K\times\LM\to K^\LM
\end{equation}
constructed as an obvious natural transformation of the classifying 
functors $\LM\to\Cat$. The map $q$ induces
\begin{equation}
\label{eq:toklm2}
K\times\Fun_\LM(K^\LM,\cO)\to K^\LM\times_\LM\Fun_\LM(K^\LM,\cO)\to\cO, 
\end{equation}
and, therefore, an $\LM$-operad map
$$Q:\Fun_\LM(K^\LM,\cO)\to\Fun^\LM(K,\cO).$$
The monoidal component of $Q$ is the monoidal functor $i:\cO_a\to\Fun(K,\cO_a)$ mentioned above. Therefore, the map $Q$ factors through
$$Q':\Fun_\LM(K^\LM,\cO)\to i^*(\Fun^\LM(K,\cO)).$$
One can easily see that $Q'$ is an equivalence.

\begin{crl}
\label{crl:monadaction}
Let $\cO=(\cO_a,\cO_m)$ be $\LM$-monoidal category, $A\in\Alg_\Ass(\cO_a)$, $f:K\to\cO_m$ a functor. Let
$$
F:\cO_m\rlarrows\LMod_A(\cO_m):G
$$
be the adjunction.
There is an equivalence between decompositions $f=G\circ f'$, $f':K\to\LMod_A(\cO_m)$,
and $A$-module structures on $f$.
\end{crl}
\qed

\subsubsection{}
\label{sss:bar-elev-end}
We apply Corollary~\ref{crl:monadaction} to the functor 
$\Sh(\gamma\circ u_+):\Delta^\op_+\to\cO_m$.

According to~\ref{sss:klm} and \ref{crl:monadaction}, we have to produce
a map of operads
$\LM\to\Funop_\LM((\Delta^\op_+)^\LM,\cO)$ or, equivalently,
an $\LM$-operad map $(\Delta^\op_+)^\LM\to\cO$, whose 
$a$-component is given by $A$ and
$m$-component is $\gamma\circ u_+:\Delta^\op_+\to\cO_m$.

Let $P_\LM$ be the monoidal envelope of $\LM$. 
The monoidal part of $P_\LM$ has objects  $a^n$, $n\geq 0$,
and the arrows are generated by the unit $\one\to a$, 
the product $aa\to a$, subject to the standard identities.
The module part of $P_\LM$ consists of the objects $a^nm$
with the obvious action of the monoidal part.

One has an $\LM$-monoidal functor 
$i:(\Delta^\op_+)^\LM\to P_\LM$ that is the unit on the 
monoidal part, and $u_+:\Delta^\op_+\to P_\LM$ carrying
$[n-1]$ to $a^nm$.

Since $\cO$ is $\LM$-monoidal, the map $\gamma:\LM\to\cO$
uniquely extends to $\Gamma: P_\LM\to\cO$. The composition
with $i$ yields the required $\LM$-operad map
$(\Delta^\op_+)^\LM\to\cO$.

\subsection{Bar resolution for presheaves }
\label{ss:barpresheaves}

We apply the reasoning of the previous subsection to
enriched presheaves. Given a monoidal category $\cM$
and an $\cM$-enriched category $\cA$ with the space of 
objects $X$, we want to describe a bar resolution for
$f\in P_\cM(\cA)=\Fun_{\cM^\rev}(\cA^\op,\cM)$.

The pair $\gamma=(\cA^\op,f)$ is given by the functor
\begin{equation}
\label{eq:gamma0}
\gamma:\LM\to\Quiv_{X^\op}^\LM(\cM^\rev,\cM).
\end{equation}
The information we need about the bar resolution of $f$ is contained
in the composition
$$
\gamma\circ u_+:\Delta^\op_+\to\LM\to\Quiv_{X^\op}^\LM(\cM^\rev,\cM).
$$
Since 
$$\Quiv^\LM_{X^\op}(\cM^\rev,\cM)=
\Funop_\LM(\LM_{X^\op},(\cM^\rev,\cM)),$$
the functor $\gamma$ defines (and is uniquely defined by)
\begin{equation}
\label{eq:gammaprime}
\gamma':\LM_{X^\op}\to(\cM^\rev,\cM),
\end{equation}
where the $\LM$-operad $\LM_{X^\op}$ is the one discussed in
\ref{sss:detailsonlmx}.
\subsubsection{}
\label{sss:lmcircdash}
We will need to know more about the $\LM$-operad $\LM_X$ and its base change $\LM^\circ_X:=\Delta^\op_+\times_{\LM}\LM_X$.

The explicit description of $\LM_X$ is given in \cite{H.EY}, 3.2. According to this description,
$\LM_X$ is presented by a functor
$$(\Delta_{/\LM})^\op\to\cS$$
carrying $\sigma:[n]\to\LM$ to $\Map(\cF(\sigma),X)$
where $\cF:\Delta_{/\LM}\to\Cat$
has values in conventional categories described by certain diagrams, see ~\cite{H.EY}, 3.2, especially the diagrams (51), (55), (60).

The base change $\LM_X^\circ=\Delta^\op_+\times_\LM\LM_X$ 
is described by the collection of $\cF(\sigma)$
for $\sigma:[n]\to\LM$ that factor through 
$u_+:\Delta^\op_+\to\LM$.

The categories $\cF(\sigma)$ for these values of $\sigma$
canonically decompose $\cF(\sigma)=\cF^-(\sigma)\sqcup [n]$, where $[n]$ 
appears as the rightmost component of $\cF(\sigma)$ in the graphic
presentation \cite{H.EY}, (55), (60), the component containing the vertex
$y_1$, see {\sl op. cit.}, diagram (51).

This means that $\LM^\circ_X=\LM^{\circ,-}_X\times X$,
so that the canonical projection $\LM^\circ_X\to\Delta^\op_+$ factors through the projection to $\LM^{\circ,-}_X$.

\subsubsection{}
The restriction of~(\ref{eq:gammaprime}) to 
$\LM^\circ_{X^\op}$ gives therefore
\begin{equation}
\label{eq:gammacirc}
\gamma^\circ:\LM^{\circ,-}_{X^\op}\times X^{\op}\to
u^*_+(\cM^\rev,\cM),
\end{equation}
where $u^*_+(\cM^\rev,\cM)$ is the base change of $(\cM^\rev,\cM)$ considered as a category over $\LM$.

We compose $\gamma^\circ$ with the equivalence 
$\op:u^*_+(\cM^\rev,\cM)\to u^*_+(\cM,\cM)$
\footnote{Note that this is an equivalence over $\op:\Delta^\op_+\to\Delta^\op_+$.}.

We get a functor

$$
\gamma^{\circ,-}:\LM^{\circ,-}_{X^\op}\to
\Fun(X^\op,u^*_+(\cM,\cM)).
$$
Since the projection $\LM^\circ_{X^\op}\to\Delta^\op_+$
factors through $\LM^{\circ,-}_{X^\op}\to\LM$, the map
$\gamma^{\circ,-}$ defines a map
\begin{equation}
\label{eq:gammacirc-}
\gamma^{\circ,-}:\LM^{\circ,-}_{X^\op}\to
u^*_+(\Fun^\LM(X^\op,(\cM,\cM))),
\end{equation}
where we use the notation of~\cite{H.EY}, 6.1.6 to define
the target of the map.

\subsubsection{}
The right-hand side of (\ref{eq:gammacirc-}) has, as $\Ass$-component,
the monoidal category $\Fun(X^\op,\cM)$. There is a monoidal functor
$$c:\cM\to\Fun(X^\op,\cM)$$
carrying $m\in\cM$ to the corresponding constant functor.
The arrow
\begin{equation}
c^!:(\cM,\Fun(X^\op,\cM))\to\Fun^\LM(X^\op,(\cM,\cM)).
\end{equation}
induced by $c$ is cartesian in $\Op_\LM$ by Lemma~\ref{lem:cartesianarrow}. Since the $\Ass$-component
of $\gamma^{\circ,-}$ factors through the map $\Ass_{X^\op}\to\cM$
defining $\cA^\op$, the map (\ref{eq:gammacirc-}) factors through $c^!$ 
giving the map that we denote by the same letter
\begin{equation}
\label{eq:gamma-circ-dash}
\gamma^{\circ,-}:\LM^{\circ,-}_{X^\op}\to
u^*_+(\cM,\Fun(X^\op,\cM)).
\end{equation}

\subsubsection{}
The category $\LM^{\circ,-}_{X^\op}$ has one object over the terminal
object $[-1]$ of $\Delta^\op_+$. We will denote this object by $*$
(note that it is not a terminal object). The functor $\gamma^{\circ,-}$
applied to $*$ gives $G(f)\in\Fun(X^\op,\cM)$ (once more, $G$
is the forgetful functor $P_\cM(\cA)\to\Fun(X^\op,\cM)$).

An object of $\LM^{\circ,-}_{X^\op}$ over $[n-1]$, $n\geq 1$, is given by a collection of objects $(y,x_n,y_n,\ldots,x_1)$ of $X$.

The functor $\gamma^{\circ,-}$ carries $(y,x_n,y_n,\ldots,x_1)$
to 
$$(f(y),\cA(y_n,x_n),\ldots,\cA(y_2,x_2),Y(x_1))\in\cM^n\times
\Fun(X^\op,\cM),$$
where $Y$ is the Yoneda embedding.

\subsubsection{} It is interesting to see what does
$\gamma^{\circ,-}$ do with the arrows.

An arrow $\alpha$ in $\LM^{\circ,-}_{X^\op}$ from 
$(y,x_n,y_n,\ldots,x_1)$ to $*$ is given by a collection of maps 
$\alpha_i: x_i\to y_{i+1}$ (or $\alpha_n:x_n\to y$).

The functor $\gamma^{\circ,-}$ carries $\alpha$ to the arrow 
$$
(f(y),\cA(y_n,x_n),\ldots,\cA(y_2,x_2),Y(x_1))\to f
$$
defined by the map
$$
f(y)\otimes\cA(y_n,x_n)\otimes\ldots\otimes\cA(y_2,x_2)\to f(x_1),
$$
defined by the $\cA^\op$-module structure on $P_\cM(\cA)\ni f$.

\

We are now ready to formulate the enriched presheaf analog of
Lemma~\ref{lem:bar-mod}. 

\begin{lem}
\label{lem:bar-pre}
The functor 
\begin{equation}
\label{eq:bar-pre}
(\LM^{\circ,-}_{X^\op})_{/*}\to(\cM,\Fun(X^\op,\cM))
\end{equation}
induced by $\gamma^{\circ,-}$, is an operadic colimit diagram.
\end{lem}
\begin{proof}
Note that the source of the functor~(\ref{eq:bar-pre}) has form $K^\triangleright$, where $K=\Delta^\op\times_{\Delta^\op_+}
(\LM^{\circ,-}_{X^\op})_{/*}$.
The evaluation map $e_x:\Fun(X^\op,\cM)\to\cM$, $x\in X^\op$, commutes 
with the left $\cM$-module structure, so, in order to prove
the lemma, it is sufficient to verify that for any $x\in X^\op$ the
composition of (\ref{eq:bar-pre}) with the evaluation map
is an operadic colimit diagram $K^\triangleright\to (\cM,\cM)$.

We know that the composition $\gamma\circ u_+:\Delta^\op_+\to\Quiv_{X^\op}^\LM(\cM^\rev,\cM)$ is an operadic colimit diagram.
Since $\cM$ is a monoidal category with colimits, the restriction functor
\begin{equation}
\label{eq:restriction}
\Fun_\LM(\Delta^\op_+,\Quiv_{X^\op}^\LM(\cM^\rev,\cM))\to
\Fun_\LM(\Delta^\op,\Quiv_{X^\op}^\LM(\cM^\rev,\cM))
\end{equation}
has a left adjoint carrying $\gamma\circ u$ to $\gamma\circ u_+$.
Since, by adjunction, the arrow (\ref{eq:restriction}) can be rewritten
as 
\begin{equation}
\label{eq:restriction2}
\Fun_\LM(\LM_{X^\op}^\circ,(\cM,\cM))\to
\Fun_\LM(\Delta^\op\times_{\Delta^\op_+}\LM_{X^\op}^\circ,(\cM,\cM)),
\end{equation}
this implies that the functor $\gamma^\circ$ is an operadic left Kan extension of its restriction to $\Delta^\op\times_{\Delta^\op_+}\LM_{X^\op}^\circ$. This easily implies the claim.

\end{proof}
\subsubsection{}
We now intend to show that (\ref{eq:gamma-circ-dash}) canonically 
factors through the forgetful functor
\begin{equation}
\label{eq:lifting-gamma-circ-dash}
(\cM,P_\cM(\cA))\stackrel{G}{\to}(\cM,\Fun(X^\op,\cM)).
\end{equation}

This implies the main result of this section, 
Proposition~\ref{prp:barquiv}.

To deduce the factorization, we need, according to~\ref{crl:monadaction},
to present the action of the monad $G\circ F$ defined by $\cA^\op$ on 
the functor (\ref{eq:gamma-circ-dash}). As a first step, we will describe a left $\Quiv_{X^\op}(\cM^\rev)$-tensored structure on the target of
(\ref{eq:gamma-circ-dash}), the category 
$u^*_+(\cM,\Fun(X^\op\cM))$.

Here is a general construction in the context of $\BM$-monoidal categories.

\subsubsection{Condensation}

Let $\cO=(\cO_a,\cO_m,\cO_b)$ be a $\BM$-monoidal category.
We can look at $\cO_m$ as an object of 
$\RMod_{\cO_b}(\LMod_{\cO_a}(\Cat))$. Its bar construction gives an augmented simplicial object in $\LMod_{\cO_a}(\Cat)$ that classifies
a cocartesian fibration over $\Delta^\op_+$. The total category of
this cocartesian fibration in $v^*_+(\cO_m,\cO_b)$ where
the functor $v_+:\Delta^\op\to\RM$ is defined as the composition of
$u_+$ with $\op:\LM\to\RM$. The resulting $\LM$-monoidal category
will be called {\sl the condensation} of $\cO$ and denoted by $\cond(\cO)$.

Applying the condensation functor to $\cO=\Quiv^\BM_{X^\op}(\cM^\rev)$,
we get an $\LM$-monoidal category whose monoidal part is
$\Quiv_{X^\op}(\cM^\rev)$ and whose left-tensored part is
$v^*_+(\Fun(X^\op,\cM),\cM^\rev)=u^*_+(M,\Fun(X^\op,\cM))$.

\subsubsection{}
Therefore, in order to show that (\ref{eq:gamma-circ-dash}) 
canonically factors through~(\ref{eq:lifting-gamma-circ-dash}), 
we have to extend $\gamma^{\circ,-}$ to a map of operads
\begin{equation}
\label{eq:Gamma}
\Gamma:(\LM^{\circ,-}_{X^\op})^\LM\to
\cond(\Quiv_{X^\op}^\BM(\cM^\rev)). 
\end{equation}
\subsubsection{} We will construct (\ref{eq:Gamma}) using the presentation of $\LM$-operads by simplicial spaces over $\LM$, 
or, equivalently, by presheaves on $\Delta_{/\LM}$. We will describe
the functors corresponding to the source and the target of $\Gamma$,
and we will show that (\ref{eq:gamma0}) allows one to construct a map
of these presheaves.

We will now describe the source of (\ref{eq:Gamma}).
First of all, $\LM^{\circ,-}_{X^\op}$ is a category over $\Delta^\op_+$,
so the space of its $k$-simplices decomposes
$$
(\LM^{\circ,-}_{X^\op})_k=\coprod_{\tau:[k]\to\Delta^\op_+}
\LM^{\circ,-}_{X^\op}(u_+\circ\tau),
$$
where $\LM^{\circ,-}_{X^\op}(u_+\circ\tau)=
\Map(\cF^-(u_+\circ\tau),X^\op)$,
see the notation of \ref{sss:lmcircdash}.

Let $\sigma:[n]\to\LM$ be an object of $\Delta_{/\LM}$
presented by a sequence
$$
\sigma: a^{d_0}m\to\ldots a^{d_k}m\to a^{d_{k+1}}\to
\ldots a^{d_n}.
$$
We will assume $k=-1$ if $\sigma$ factors through 
$\Ass\to\LM$.

Then the source of (\ref{eq:Gamma}) is described by the following 
formula.

\begin{equation}
(\LM^{\circ,-}_{X^\op})^\LM(\sigma)=
\begin{cases}
[0], & \mathrm{ if\  }k=-1,\\
(\LM^{\circ,-}_{X^\op})_k=\coprod_{\tau:[k]\to\Delta^\op_+}
\LM^{\circ,-}_{X^\op}(u_+\circ\tau) & \mathrm{ otherwise}.
\end{cases}
\end{equation}
The description of the target of (\ref{eq:Gamma}) will be presented
in \ref{sss:target}, after a certain digression.

\subsubsection{Internal mapping operad, reformulated}
\label{sss:internalmappingoperad}

We need some detail on internal operad objects, \cite{H.EY}, 2.8.

The direct product in the category $P(C)$ of presheaves 
has a right adjoint assigning to a pair $F,G\in P(C)$ a presheaf
$\Fun_{P(C)}(F,G)$ whose value at $c\in C$ is calculated as the limit
\begin{equation}
\label{eq:homp}
\Fun_{P(C)}(F,G)(c)=\lim_{a\to b\to c}\Map(F(b),G(a)).
\end{equation}

Fix a category $B$ and let us look for a similar description of the internal Hom in $\Cat_{/B}$. The latter is a full subcategory
of $P(\Delta_{/B})$, so we can try to use the formula (\ref{eq:homp})
with $C=\Delta_{/B}$. Given $F,G\in\Cat_{/B}$, we are looking for
an object $\Fun_{\Cat_{/B}}(F,G)$ furnishing an equivalence
\begin{equation}
\Map_{\Cat_{/B}}(H,\Fun_{\Cat_{/B}}(F,G))=\Map_{\Cat_{/B}}(H\times F,G).
\end{equation}
Since $\Cat_{/B}$ is a full subcategory of $P(\Delta_{/B})$, and since the representable presheaves belong to $\Cat_{/B}$, the object
$\Fun_{\Cat_{/B}}(F,G)$, if it exists, is equivalent to
the presheaf
$\Fun_{P(\Delta_{/B})}(F,G)$. This provides a very easy criterion for the
existence of $\Fun_{\Cat_{/B}}(F,G)$: it exists if and only if 
$\Fun_{P(\Delta_{/B})}(F,G)$ is a category over $B$.
Note that, by definition, for $b\in B$, the fiber
$\Fun_{\Cat_{/B}}(F,G)_b$ identifies with $\Fun(F_b,G_b)$.

Let us now apply the above reasoning to operads. 
If $\cP$ is a flat $\cO$-operad, one defines a marked category
$\pi:\cP'\to\cP$ over $\cP$ by the formulas
\begin{equation}
\label{eq:pprime}
\cP'=\Fun^{in}([1],\cO)\times_\cO\cP,
\end{equation}
where $\Fun^{in}([1],\cO)$ denotes the category of inert arrows in $\cO$.
Let $s,t:\Fun^{in}([1],\cO)\to\cO$ be the standard projections.
An arrow in $\cP'$ is marked iff its projections to $\cP$ and to $\cO$
(via $s$) are inert. $\cP'$ considered as a category over $\cO$ is
flat. One defines $\Fun_\cO^\sharp(\cP',\cQ)$ as 
the full subcategory of $\Fun_{\Cat_{/\cO}}(\cP',\cQ)$ spanned by the 
arrows $\alpha:\cP'_o\to\cQ_o$, for some $o\in\cO$, carrying 
the marked arrows in $\cP'_o$ to equivalences.

Proposition 2.8.3 of \cite{H.EY} claims that, for any $\cO$-operad $\cQ$,
$\Fun^\sharp_\cO(\cP',\cQ)$ is an $\cO$-operad representing
$\Funop_\cO(\cP,\cQ)$.
 
In particular, for $s:[n]\to\cO$,  
$$
\Funop_\cO(\cP,\cQ)(s)\subset\lim_{u\to v\to s}\Map(\cP'(v),\cQ(u)),
$$
consists of collections whose each component corresponding to
$u=v:[1]\to[0]\stackrel{k}{\to}[n]\stackrel{s}{\to}\cO$ carries the 
marked arrows of $\cP'_{s(k)}$ to equivalences in $\cQ_{s(k)}$.

\begin{rem}
\label{rem:pprimep}
Note that a map $\cP'_x\to\cQ_x$ factoring through $\pi:\cP'_x\to\cP_x$ automatically carries the marked arrows to equivalences.
\end{rem}

\subsubsection{Condensation, for $\BM$-operads}

We need more explicit formulas for the condensation of a $\BM$-monoidal category. This operation can be defined in the greater 
generality of $\BM$-operads.

Given a $\BM$-operad $p:\cO\to\BM$, we will define its ``condensation'' 
$q:\cO'=\cond(\cO)\to\LM$ so that
\begin{itemize}
\item $\cO'\times_\LM\Ass=\cO\times_\BM\Ass_-$.
\item $\cO'_m=v^*_+(\cO\times_\BM\RM)$,
\end{itemize}
where $v_+:\Delta^\op_+\to\RM$ is defined as the composition
$v_+=\op\circ u_+$.

We will be using a presentation of operads
by presheaves on $\Delta_{/\BM}$ and $\Delta_{/\LM}$.

Recall that  $\BM=(\Delta_{/[1]})^\op$; its objects are arrows
$s:[n]\to[1]$ and an arrow from $s:[n]\to[1]$ to $t:[m]\to[1]$
is given by $f:[m]\to[n]$ such that $s\circ f=t$.

An object $s:[n]\to[1]$ of $\BM$ defined by the formulas
$s(i)=0, i=0,\ldots,k, \ s(i)=1, i>k$, is otherwise denoted 
$a^kmb^{n-k-1}$, see~\cite{H.EY}, 2.9.2.

We denote by $\BM^0$ the full subcategory of $\BM$ spanned by 
$s:[n]\to[1]$ with $s(0)=0$. The category $\LM$ is the full subcategory 
of $\BM^0$ spanned by the arrows $s:[n]\to[1]$  having at most one 
value of $1$. The subcategories $\BM^1$ and $\RM$ are defined as
images of $\BM^0$ and $\LM$ under $\op:\BM\to\BM$.

The full embeddings $\LM\to\BM^0$ and $\RM\to\BM^1$
have left adjoint functors 
$\ell:\BM^0\to\LM$ and $r:\BM^1\to\RM$ erasing superfluous values
of $1$ and $0$ respectively. The functors $\ell$ and $r$ induce
functors $\Delta_{/\BM^0}\to\Delta_{/\LM}$ and $\Delta_{/\BM^1}
\to\Delta_{/\RM}$.
We will denote by $\Delta_{/\RM}^\act$ the category of simplices in
$\RM$ whose all arrows are active.

For $\tau:[n]\to\BM^0$ given by
$$\tau: a^{c_0}mb^{d_0}\to\ldots\to a^{c_k}mb^{d_k}\to a^{c_{k+1}}
\to\ldots\to a^{c_n}$$
or
$$\tau: a^{c_0}mb^{d_0}\to\ldots\to a^{c_k}mb^{d_k}\to b^{d_{k+1}}
\to\ldots\to b^{d_n},$$
we denote by $\tau_-$ the $k$-simplex
$$a^{c_0}mb^{d_0}\to\ldots\to a^{c_k}mb^{d_k}$$

(if $\tau$ is a simplex in $\Ass_-\subset\BM^0$, we put $k=-1$
and $\tau_-$ is empty in this case).

For $\sigma\in\Delta_{/\LM}$ we define
\begin{equation}
\Pi(\sigma)=\{\tau\in\Delta_{/\BM^0}|\ell{\tau}=\sigma,
r(\tau_-)\in\Delta^\act_{/\RM}\}.
\end{equation}
Note that for $\sigma\in\Delta_{/\Ass}$ $k=-1$ and $\Pi(\sigma)=\{\sigma\}$.

The assignment $\sigma\mapsto\Pi(\sigma)$ is obviously functorial, with a map
$\alpha:[m]\to[n]$ defining $\Pi(\sigma)\to\Pi(\sigma\circ\alpha)$ carrying $\tau$ to $\tau\circ\alpha$.

\

Let a $\BM$-operad $\cO$ be described by a presheaf
$F\in P(\Delta_{/\BM})$. We define a presheaf
$F'\in P(\Delta_{/\LM})$ describing $\cond(\cO)$ by the formula
\begin{equation}
\label{eq:condense}
 F'(\sigma)=\coprod_{\tau\in\Pi(\sigma)}F(\tau).
\end{equation}

\begin{Lem}
The presheaf $F'$ defined above represents an $\LM$-operad.
\end{Lem}
\begin{proof}
1. Segal condition follows from the definition of $\Pi(\sigma)$. 
To verify completeness, we can fix $w=a^km\in\LM$ and study the simplicial space $n\mapsto\coprod_{\tau\in\Pi(w_n)}F(\tau)$,
where $w_n$ is the degenerate $n$-simplex determined by $w\in\LM$.
This simplicial space is equivalent to the product of the fibers
$\{a^k\}\times_\BM\cO$ and $v^*_+(\cO)$, which is, of course, complete
as a simplicial space. Thus, $F'$ represents a category over $\LM$
which we will denote by $\cond(\cO)$.

2. It remains to prove that $\cond(\cO)$ is fibrous.

Let us describe cocartesian liftings of the inerts in 
$\LM$. These are of two types
\begin{itemize}
\item Erasing $m$. Such inert has form $\alpha:a^nm\to a^l$.
\item Not erasing $m$: either $\alpha:a^n\to a^l$ or 
$\alpha:a^nm\to a^lm$.
\end{itemize}
In the first case the cocartesian lifting of $\alpha$ in $F'(\alpha)$ having a source in $F(t)$ for $t=a^nmb^k$ is an inert arrow in
$F(\tau)$, where $\tau\in\Pi(\alpha)$ is (the only) inert arrow
from $a^nmb^k$ to $a^l$ such that $\ell(\tau)=\alpha$.

In the second case, for $\alpha:a^nm\to a^lm$, with the source
$t=a^nmb^k$, is the inert arrow in $F(\tau)$ where $\tau$ is defined
by the conditions $\ell(\tau)=\alpha$, $r(\tau)=\id_{mb^k}$.

The rest of the fibrousness conditions \cite{L.HA}, 2.3.3.28 or 
\cite{H.EY}, 2.6.3 easily follow from the above description.

\end{proof}

\subsubsection{}
\label{sss:target}

We will now present the target of (\ref{eq:Gamma}) by a presheaf on
$\Delta_{/\LM}$.

For $\sigma:[n]\to\LM$ we have
\begin{eqnarray}
\cond(\Quiv_{X^\op}^\BM(\cM^\rev))(\sigma)=\coprod_{\hat\sigma\in\Pi(\sigma)}\Quiv_{X^\op}^\BM(\cM^\rev)(\hat\sigma)\\
\nonumber\subset\coprod_{\hat\sigma\in\Pi(\sigma)}\lim_{u\to v\to\hat\sigma}
\Map(\BM'_{X^\op}(v),\cM(\op\circ u)),
\end{eqnarray}
where the inclusion means that we have to choose the connected components preserving the inerts, and  $\BM'_{X^\op}$ is defined by the formula~(\ref{eq:pprime}).

\subsubsection{}

We are now ready to construct $\Gamma$. It consists of a compatible collection of maps

\begin{equation}
\label{eq:Gamma-det}
\Gamma_{\sigma,\tau,s,t}:\LM^{\circ,-}_{X^\op}(u_+\circ\tau)\times \BM'_{X^\op}(t)\to\cM(\pi\circ\op\circ s)
\end{equation}
for each $\sigma\in\Delta_{\LM}$, $\hat\sigma=\sigma*
(v_+\circ\op\circ\tau)\in\Pi(\sigma)$
and $s\to t\to\hat\sigma$ in $\Delta_{/\BM}$.

Equivalently, this can be described by a compatible collection
\begin{equation}
\label{eq:Gamma-det-1}
\Gamma_{\sigma,\tau}:\LM^{\circ,-}_{X^\op}(u_+\circ\tau)\times \BM'_{X^\op}(\hat\sigma)\to\cM(\pi\circ\op\circ\hat\sigma)
\end{equation}
for each $\sigma\in\Delta_{\LM}$ and 
$\hat\sigma=\sigma*(v_+\circ\op\circ\tau)\in\Pi(\sigma)$.

\subsubsection{} The collection of maps (\ref{eq:Gamma-det-1}) 
we are going to present will factor through the natural projections
$\BM'_{X^\op}\to\BM_{X^\op}$, so, by Remark~\ref{rem:pprimep},
it will induce a map to $\Fun^\sharp(\BM'_{X^\op},\cM^\rev)$
as needed.

The explicit formulas
for $\BM_X$ show that $\BM_{X^\op}(\hat\sigma)=\BM_{X^\op}(\sigma)$,
so we will rewrite the source of (\ref{eq:Gamma-det-1}) as
\begin{eqnarray}
\nonumber
\Map(\cF^-(u_+\circ\tau),X^\op)\times\Map(\cF(\sigma),X^\op)=
\Map(\cF^-(u_+\circ\tau)\sqcup\cF(\sigma),X^\op)=\\
\nonumber \Map(\cF(u_+\circ\tau)\sqcup^{\cF(\sigma_k)}\cF(\sigma),X^\op)=
\Map(\cF(\sigma),X^\op)\times_{\Map(\cF(\sigma_k),X^\op)}\Map(\cF(\tau),X^\op),
\end{eqnarray}
where  $\sigma_k$ denotes the $k$-simplex in $\LM$ with
the constant value $m\in\LM$. Similarly, one has
a canonical presentation
\begin{equation}
\cM(\pi\circ\op\circ\hat\sigma)=\cM(\pi\circ\op\circ\sigma)\times_{\cM(\pi\circ\sigma_n)}\cM(\pi\circ\op\circ\tau).
\end{equation}

\subsubsection{}
The map $\gamma'$ (\ref{eq:gammaprime})
is described by a map of presheaves, given, for each
$\sigma:[n]\to\LM$, by a map
\begin{equation}
\label{eq:gammasigma}
\gamma_\sigma:\Map(\cF(\sigma),X^\op)\to M(\pi\circ\op\circ
\sigma).
\end{equation}
The collection of maps $\Gamma_{\sigma,\tau}$ is defined as the fiber product of $\gamma_\sigma$ and $\gamma_\tau$.

\subsection{Conclusion}
We have just  constructed the functor $\Gamma$ (\ref{eq:Gamma}) 
defined by the collection of maps $\Gamma_{\sigma,\tau}$
(\ref{eq:Gamma-det-1}) obtained as the fiber product of
$\gamma_\sigma$ and $\gamma_{u_+\tau}$ (\ref{eq:gammasigma}).
This implies the main result of this section.
\begin{prp}
\label{prp:barquiv}
The functor $\Gamma$ (\ref{eq:Gamma}) defines an operadic colimit diagram
\begin{equation}
(\LM^{\circ,-}_{X^\op})_{/*}\to(\cM,P_\cM(\cA)).
\end{equation}
\end{prp}

\section{Weighted colimits}
\label{sec:wc}

In this paper we study weighted colimits of $\cM$-functors.

Given an $\cM$-functor $f:\cA\to\cB$ where $\cA$ is $\cM$-enriched category 
and $\cB$ is a left $\cM$-module, and an enriched presheaf
$W\in P_\cM(\cA)$, we will define a weighted colimit 
 $\colim_W(f)\in\cB$. The construction is functorial in 
$W$, so that $\colim_W(f)$ is the evaluation at $W$
of a certain colimit preserving functor 
$\colim(f): P_\cM(\cA)\to\cB$ preserving the left $\cM$-module structure. The composition 
$\colim(f)\circ Y$ with the enriched Yoneda embedding
yields $f:\cA\to\cB$.

The construction of weighted colimit is also functorial in
$f:\cA\to\cB$. This will imply Theorem~\ref{thm:universal}
claiming that the Yoneda embedding $Y:\cA\to P_\cM(\cA)$
is a universal $\cM$-functor. 

\subsection{Internal Hom}
\label{ss:funak}

We keep the notation of \ref{sss:tens-succ}.
Let $\cC=\Cat^L$.

Given $A,B,C$ associative algebras in $\cC$, we have a functor
$$
\RT_{A,B,C}:_A\BMod_B(\cC)\times{_B\BMod_C}(\cC)\to
_A\!\!\BMod_C(\cC)
$$
defined by the relative tensor product. This functor has a right
adjoint
$$
\Fun^L_A:{_A\BMod_B}(\cC)^\op\times_A\!\!\BMod_C(\cC)\to
_B\!\!\BMod_C(\cC) 
$$
which we are now going to describe.

\subsubsection{}

The formula
$$
\Map_{_B\BMod_C}(N,\Fun^L_{A,B,C}(M,K))=\Map_{_A\BMod_C}(M\otimes_BN,K)
$$
determines $\Fun^L_{A,B,C}$~\footnote{This is a temporary notation.} as presheaf. In the case when $M$ is a free
$A-B$-bimodule $M=A\otimes X\otimes B$, this presheaf is represented 
by the $B-C$-bimodule $\Fun^L(X\otimes B,K)$; since any bimodule is a colimit of free bimodules and since the Yoneda embedding commutes with the limits, this proves the existence of $\Fun^L_{A,B,C}(M,K)$ in general~\footnote{The forgetful functor $\Cat^L\to\CAT$ to the category of (big) categories creates the limits.}.
\subsubsection{}
The functor $\Fun^L_{A,B,C}$, as defined above, depends on 
three algebras $A,B,C$. We will omit $B$ and $C$ from the notation for the following reason.

Let $b:B\to B'$ and $c:C\to C'$ be algebra maps. Then 
the $B'-C'$-bimodule $\Fun^L_{A,B',C'}(M,K)$ identifies with the
restriction of scalars of the $B-C$-bimodule 
$\Fun^L_{A,B,C}(M,K)$. 

\subsubsection{}
Thus, we now assume that $B=C=\one$. For $M,K\in\LMod_A(\cC)$,
we can define
$\Fun^L_A(M,K)$  
as the full subcategory of 
$\Fun_{\LMod^w_A}(M,K)$, see~\ref{sss:funlmodw},
 spanned by the lax $\LM$-monoidal functors
$f=(\id_A,f_m):(A,M)\to(A,K)$ satisfying two extra properties:
\begin{itemize}
\item $f$ is $\LM$-monoidal.
\item $f_m$ preserves small colimits.
\end{itemize}

 As we explained above, in the case when $M$ is an $A-B$-bimodule and
 $K$ is an $A-C$-bimodule, $\Fun^L_A(M,K)$ acquires a natural
 $B-C$-bimodule structure.

\subsection{Weighted colimit}
\label{ss:RLQuiv}

We apply the notion of tensor product described 
in~\ref{ss:relativetensor} to the following context.

Fix a monoidal category 
$\cM\in\Alg_\Ass(\Cat^L)$
and a $X\in\Cat$.
Let $\cB$ be a left $\cM$-module. Let $N=\Fun(X,\cM)$. This is 
a right $\cM$-module.  
\begin{lem}
\label{lem:Nrightdualizable}
$N$ is right dualizable. Its right dual is $M=\Fun(X^\op,\cM)$ considered as a left $\cM$-module.
\end{lem}
\begin{proof}
We deduce the duality between $M$ and $N$ from the special case 
$\cM=\cS$. In this case $N=P(X^\op)$, 
$M=P(X)$ and
the duality is given by the maps
$$ c:\cS\to P(X^\op)\otimes P(X)=
P(X^\op\times X)$$
defined as the colimit-preserving map preserving the terminal objects, and
$$ e:P(X)\otimes P(X^\op)\to\cS$$
extending the Yoneda map $X^\op\times X\to\cS$ to preserve colimits.

To get the duality for arbitrary $\cM$, we use 
Proposition~\ref{prp:composition}. We have three associative algebras in $\Cat^L$, 
$A=B=\cS$ and $C=\cM$, the adjoint pair 
$(P(X^\op),P(X))$ that we just constructed, 
and the one defined by $(\cM,\cM)$. The composition of these gives an adjunction for $(N,M)$.
\end{proof}

\subsubsection{}
\label{sss:Q}
Let us now apply Corollary~\ref{crl:morita-semi}
to $A=\cM$ and $N=\Fun(X,\cM)$. We get the right dual module $M=\Fun(X^\op,\cM)$ and the endomorphism ring
$B=\Quiv_X(\cM)$, see~\cite{H.EY}, 4.5.3. We also get a canonical $\cM$-$\cM$-bimodule map
\begin{equation}
\label{eq:eval-quiv}
e:\Fun(X^\op,\cM)\otimes_{\Quiv_X(\cM)}\Fun(X,\cM)\to\cM.
~\footnote{In fact $e$ is an equivalence. We do not use this fact.}
\end{equation}

We would like to comment about the right 
$\Quiv_X(\cM)$-module structure on $\Fun(X^\op,\cM)$. 

According to Remark~\ref{rem:viam}, the reverse monoidal category $\Quiv_X(\cM)^\rev$ identifies with the endomorphism object of the left $\cM$-module $\Fun(X^\op,\cM)$, that is, of the right $\cM^\rev$-module
$\Fun(X^\op,\cM)$. This means that the right 
$\Quiv_X(\cM)$-action on $\Fun(X^\op,\cM)$ is defined by the same construction as the left action on $\Fun(X,\cM)$,
but with $\cM^\rev$ replacing $\cM$ and $X^\op$ replacing $X$.

\
 
Tensoring (\ref{eq:eval-quiv}) with $\cB$ over $\cM$
and using associativity of the relative tensor product ~\ref{sss:asso},
we get a map of left $\cM$-modules
\begin{equation}
\label{eq:eval-quiv-2}
e_\cB:\Fun(X^\op,\cM)\otimes_{\Quiv_X(\cM)}\Fun(X,\cB)\to\cB.
\end{equation}

This gives a $\TEN_\succ$-algebra 
$\cQ=\cQ_{X,\cM,\cB}$ in $\Cat^L$ consisting of the following 
categories and
operations between them described above.
\begin{itemize}
\item Monoidal categories $\cQ_a=\cQ_{a'}=\cM$,$\cQ_b=\Quiv_X(\cM)$,
\item A bimodule category $\cQ_m=\Fun(X^\op,\cM)$, a left
$\Quiv_X(\cM)$-module $\cQ_n=\Fun(X,\cB)$ and a left $\cM$-module
$\cQ_k=\cB$. 
\end{itemize}

\subsubsection{} 
Fix $X\in\Cat$, $(\cM,\cB)\in\Alg_\LM(\Cat^L)$.
We have  a relative tensor product functor
\begin{equation}
\label{eq:rtp-q}
\RT:\Alg_{\TEN_2/\TEN_\succ}(\cQ_{X,\cM,\cB})\to
\Alg_{\TEN_\succ}(\cQ_{X,\cM,\cB}).
\end{equation}

Let $\cA\in\Alg_\Ass(\Quiv_X(\cM))$ be an $\cM$-enriched
precategory with space of objects $X$.

The restriction $\RT_{\one,\cA}$ of (\ref{eq:rtp-q})
is a functor
 
\begin{equation}
\label{eq:rtp-q-2}
\RT_{\one,\cA}:\RMod_\cA(\Fun(X^\op,\cM))\times
\LMod_\cA(\Fun(X,\cB))\to\cB.
\end{equation} 

Taking into account that $\RMod_\cA(\Fun(X^\op,\cM))=
P_\cM(\cA)$ and $\LMod_\cA(\Fun(X,\cB))=\Fun_\cM(\cA,\cB)$,
we finally get the functor called {\sl  weighted colimit},

\begin{equation}
\label{eq:wc}
\colim:P_\cM(\cA)\times\Fun_\cM(\cA,\cB)\to\cB,
\end{equation}
carrying a pair $(W\in P_\cM(\cA),f:\cA\to\cB)$ 
to $\colim_W(f):=f\otimes W\in\cB$. This functor
preserves colimits separately in both arguments, as well
as left $\cM$-tensored structure in the first argument.

In particular, for a fixed $f:\cA\to\cB$ the functor
$\colim(f):P_\cM(\cA)\to\cB$ preserving the colimits and 
the left $\cM$-tensored structure, is defined.

Since the bifunctor (\ref{eq:wc}) is a special case of the
relative tensor product, it defines, using the notation (\ref{ss:funak}),
a canonical functor (that we also denote as $\colim$)
 
\begin{equation}
\label{eq:adjcolim}
\colim:\Fun_\cM(\cA,\cB)\to\Fun_{\LMod_\cM}(P_\cM(\cA),\cB)
\end{equation}

\subsection{Properties of the weighted colimit}

\begin{lem}
\label{lem:colimY}
Let $Y:\cA\to P_\cM(\cA)$ be the Yoneda embedding. Then
$\colim(Y)=\id_{P_\cM(\cA)}$.
\end{lem}
\begin{proof}
Look at the $\TENS_\succ$-monoidal category $\cF$ in 
$\Cat^L$ having the following components.
\begin{itemize}
\item Monoidal categories $\cF_a=\cF_{a'}=\cM$,
$\cF_b=\cF_c=\cF_{c'}=\Quiv_X(\cM)$, 
\item $\cF_m=\cF_k=\Fun(X^\op,\cM)$, $\cF_n=\Quiv_X(\cM)$,
\end{itemize}
with the standard $\cM$-$\Quiv_X(\cM)$-bimodule structure
on $\Fun(X^\op,\cM)$ and the unit 
$\Quiv_X(\cM)$-$\Quiv_X(\cM)$-bimodule structure
on $\Quiv_X(\cM)$.

We will study the relative tensor product defined by $\cF$. Let $\cA$ be an associative algebra in $\Quiv_X(\cM)$.
The relative tensor product with $\cA$-$\cA$-bimodule 
$\cA$ defines the identity functor on
$\RMod_\cA(\Fun(X^\op,\cM))=P_\cM(\cA)$.

We will now show that the calculation of $\colim(Y)$ has the same answer.

We apply the reduction procedure described in \ref{ss:reduction}. 
Let $\cF'=\cF^\red_\cA$  be a $\TEN_\succ$-monoidal category 
obtained from $\cF$ by reduction, see \ref{sss:red-1} and \ref{rem:red-c}.
 It has the following components.

\begin{itemize}
\item Monoidal categories $\cF'_a=\cF'_{a'}=\cM$,
$\cF'_b=\Quiv_X(\cM)$.
\item $\cF'_m=\Fun(X^\op,\cM)$, 
$\cF'_n=\RMod_\cA(\Quiv_X(\cM))$,
$\cF'_k=\RMod_\cA(\Fun(X^\op,\cM))$.
\end{itemize}

The category $\Fun(X,\Fun(X^\op,\cM))$ has a structure of
$\Quiv_X(\cM)$--$\Quiv_X(\cM)$-bimodule described by
the $\BM$-monoidal category
$$\Quiv^\BM_X(\Quiv^\BM_{X^\op}(\cM^\rev)^\rev),$$
as in~\cite{H.EY}, 6.1.7. This bimodule identifies with 
$\Quiv_X(\cM)$. This implies that $\cF'_n$ as a left $\Quiv_X(\cM)$-module identifies with $\Fun(X,P_\cM(\cA))$, so that
$\cF'$ is equivalent to the category 
$\cQ_{X,\cM,\cB}$ with $\cB=P_\cM(\cA)$.

\end{proof}

\

The construction $(X,\cM,\cB)\mapsto\cQ_{X,\cM,\cB}$ described in
\ref{sss:Q} is functorial in $(X,\cM,\cB)\in\Cat\times\Alg_\LM(\Cat^L)$. The precise expression of this functoriality is given in Section~\ref{sec:functorialityQ}. We now need only a small (and obvious) fragment of it.

\begin{lem}
A map $g:\cB\to\cB'$ of $\cM$-modules in  $\Cat^L$ induces
a $\TEN_\succ$-monoidal colimit-preserving functor 
$\cQ_{X,\cM,\cB}\to\cQ_{X,\cM,\cB'}$.
\end{lem}\qed

By Proposition~\ref{prp:colimitdescription} this induces a map preserving weighted colimits.
This implies that the colimit functor defined by formula(\ref{eq:wc}) is functorial  in $\cB$.
\begin{crl}
For $g:\cB\to\cB'$ an arrow in $\LMod_\cM$,
the following  diagram
\begin{equation}
\xymatrixcolsep{5pc}\xymatrix{
&{P_\cM(\cA)\times\Fun_\cM(\cA,\cB)}\ar^{\ \ \ \ \ \colim}[r]\ar[d]&{\cB}
\ar[d]\\
&{P_\cM(\cA)\times\Fun_\cM(\cA,\cB')}\ar^{\ \ \ \ \ \ \colim}[r]&{\cB'}
},
\end{equation} 
with the vertical arrows defined by $g$,
is commutative. 
\end{crl}\qed

\begin{crl}
\label{crl:colim-ext}
Let $\cA$ be a $\cM$-enriched category and let $\cB$ 
be a left $\cM$-modules with colimits. For any 
colimit-preserving map $F:P_\cM(\cA)\to\cB$ of left $\cM$-modules 
there is a natural equivalence 
$$F\stackrel{\sim}{\to}\colim(F\circ Y).$$
\end{crl}
\begin{proof}
Follows from \ref{lem:colimY} and \ref{crl:colim-ext}.
\end{proof}

\subsection{Universality}

We define the map
\begin{equation}
\label{eq:y*}
Y^*:\Fun_\cM^L(P_\cM(\cA),\cB)\to
\Fun_\cM(\cA,\cB)
\end{equation}
as the composition with the Yoneda embedding $Y:\cA\to P_\cM(\cA)$.

In this subsection we will show that $Y^*$ is an equivalence.  

\subsubsection{}
The weighted colimit defines a map (\ref{eq:adjcolim}) in the opposite 
direction. According to Corollary~\ref{crl:colim-ext}, the composition
$\colim\circ Y^*$ is equivalent to identity. 

We deduce that the other composition is also equivalent to
identity providing an interpretation of the weighted colimit
as an operadic left Kan extension.

Recall \ref{ss:babyY} that $\bar\cA\subset P_\cM(\cA)$ is
the full subcategory spanned by the representable functors.

In what follows we denote
$$ P=P_\cM(\cA),\  \fP=(\cM,P)\in\Alg_\LM(\Cat^L),\ 
\fA=(\cM,\bar\cA)\subset\fP,
\fB=(\cM,\cB).
$$
Thus,  $\fA$ is  an $\LM$-suboperad
of $\fP$. For $F\in\Fun_{\LMod_\cM}(P,\cB)$
we will denote by the same letter $F:\fP\to\fB$ the corresponding 
$\LM$-monoidal functor.

We claim the following.

\begin{prp}
\label{prp:olke}
Any $F\in\Fun_{\LMod_\cM}(P,\cB)$
is an operadic left Kan extension of $\bar F=F|_{\fA}$.
\end{prp}

\begin{prp}
\label{prp:olke-2}
Let $\bar F:\fA\to\fB$ be a map of $\LM$-operads. Then
the map $F:\fP\to\fB$ defined as an operadic left Kan extension 
of $\bar F$ preserves operadic colimits.
\end{prp}

\begin{thm}
\label{thm:universal}
The map (\ref{eq:y*}) is an equivalence.
\end{thm}
\begin{proof}
Let $\bar f:\cA\to\cB$ be a morphism in $\LMod^w_\cM$ defined by 
$f:\cA\to\cB$. By \ref{prp:olke-2}, its operadic left Kan extension
$F:\fP\to\fB$ preserves operadic colimits, and, therefore,
defines an arrow in $\LMod_\cM$. Therefore, by \ref{prp:olke},
$F=\colim(f)$. This implies that the composition $Y^*\circ\colim$
is equivalent to identity.
\end{proof}

The rest of this subsection is devoted to proving \ref{prp:olke}
and \ref{prp:olke-2}.

\subsubsection{Proof of Proposition~\ref{prp:olke}}
Let $f\in P$. Denote
$\cF_f=\fA\times_\fP\fP^\act_{/f}$.

We have to verify that the composition
$\cF_f\to\fA\stackrel{\bar F}{\to}\cB$ extends to 
an operadic colimit diagram
$$\cF_f^\triangleright\to\cB$$
carrying the terminal object of $\cF_f^\triangleright$
to $F(f)$.

The general case is immediately reduced to the case
$\cB=P$ and $F=\id_P$. Thus, we have to verify that
any $f\in P$ is an operadic colimit of the functor 
$\bar y:\cF_f\to\fP$ defined as the composition of the projection $\cF_f\to\fA$ with the embedding to $\fP$.

The plan of our proof is as follows.

Following to \ref{prp:barquiv} , the pair $\gamma=(\cA^\op,f)$ gives rise to a factorization of the functor
$\gamma^{\circ,-}$, see~(\ref{eq:gamma-circ-dash}),
\begin{equation}
\LM^{\circ,-}_{X^\op}\stackrel{\Gamma}{\to}
\fP\stackrel{G}{\to}(\cM,\Fun(X^\op,\cM))
\end{equation}
through the forgetful functor $G$, so that $\Gamma(*)=f$
where $*$ is a (unique) object of $\LM_{X^\op}^{\circ,-}$
over $[-1]\in\Delta^\op_+$. The restriction 
of $\Gamma$ to $\Delta^\op\times_{\Delta^\op_+}\LM^{\circ,-}_{X^\op}$ factors through the full embedding 
$\fA\to\fP$. 

Recall that
$$
K=
\Delta^\op\times_{\Delta^\op_+}
(\LM^{\circ,-}_{X^\op})_{/*}
$$
induces an equivalence
 
$$
K^\triangleright=(\LM^{\circ,-}_{X^\op})_{/*}.
$$

This yields a functor

\begin{equation}
\label{eq:toFf}
\tau:K\to(\LM^{\circ,-}_{X^\op})_{/*}\to\cF_f
\end{equation}
and an operadic colimit diagram
(by Proposition~\ref{prp:barquiv})
\begin{equation}
\Gamma:K^\triangleright\to\fP
\end{equation}
extending the composition of (\ref{eq:toFf}) with the
projection $\cF_f\to\fP$.

It remains to
verify that $\tau$ is cofinal.

\subsubsection{Cofinality of $\tau$}

To prove that $\tau$ is cofinal, we use Quillen's Theorem A
in the form of \cite{L.T}, 4.1.3.1. We have to
verify that, for any $\phi\in\cF_f$, the comma category
$$K_\phi=K\times_{\cF_f}(\cF_f)_{\phi/}$$
is weakly contractible.

The proof goes as follows. In \ref{sss:dphi1} below
we present an object
$t_\phi$ of $K_\phi$ and a functor $F:K_\phi\to
\Fun(\Lambda^2_0, K_\phi)$ whose evaluation at $1\in\Lambda^2_0$
is $\id_{K_\phi}$ and at $2\in\Lambda^2_0$ 
is the composition $K_\phi\to\{t_\phi\}\hookrightarrow K_\phi$.
The functor $F$ provides a null-homotopy for the identity
map on $K_\phi$. This proves cofinality of $\tau$ and, finally,
Proposition~\ref{prp:olke}.

\subsubsection{}
\label{sss:dphi0}
We denote by $p:K\to\Delta^\op$ the obvious projection. We also denote $p:K_\phi\to\Delta^\op$ the composition  
 $K_\phi\to K\to\Delta^\op$. The functor 
$F:K_\phi\to\Fun(\Lambda^2_0,K_\phi)$ will be defined as 
the one assigning to $t\in K_\phi$ a $p$-product diagram
$t\leftarrow t'\to t_\phi$ in the sense of \cite{L.T}, 4.3.1.1.
The latter means that for any $x\in K_\phi$ the diagram
$$
\xymatrix{
&\Map_{K_\phi}(x,t')\ar[r]\ar[d]&\Map_{K_\phi}(x,t)\times\Map_{K_\phi}(x,t_\phi)\ar[d]\\
&\Map_{\Delta^\op}(p(x),p(t'))\ar[r]&\Map_{\Delta^\op}(p(x),p(t))\times\Map_{\Delta^\op}(p(x),p(t_\phi))
}
$$
is cartesian. To construct $F$, we first construct
a functor $\bar F:\Delta^\op\to\Fun(\Lambda^2_0,\Delta^\op)$
(this is very easy), and then prove that for any $t$ there exists
a $p$-product diagram $t\leftarrow t'\to t_\phi$ whose image under $p$
is $\bar F(p(t))$. By the uniqueness of relative product diagrams,
the functor $\bar F$ lifts to a functor
$F:K_\phi\to\Fun(\Lambda^2_0,K_\phi)$.
 
\subsubsection{}
\label{sss:dphi1}
We define $q:\cF_f\to\Delta^\op_+$ so that the composition
$u_+\circ\op\circ q$ is the projection to $\cF_f\to\fP^\act_{/f}
\to\LM$. This condition uniquely determines $q$ as any object in 
$\fP^\act_{/f}$ has image in $u_+(\Delta^\op_+)$. The definition
of $q$ is chosen so that for $t\in K$ the equality $p(t)=q(\tau(k))$ holds.

Let $\phi\in\cF_f$ be defined by the collection
$(m_1,\ldots,m_n,z,\beta)$ with $m_i\in\cM$, $z\in X^\op$ and  $\beta:(m_1,\ldots,m_n,z)\to f$ defined by a map
$\otimes m_i\to f(z)$ (we will also denote it by $\beta$). 

Given  $d\in K$ defined by a sequence
$(y,x_k,y_k,\ldots,x_2,y_2,x_1)$ together with arrows $\alpha_i:
x_i\to y_{i+1}$ ($\alpha_k:x_k\to y$), an object
$t:\phi\to\tau(d)$ of $K_\phi$ is defined by a collection of maps 
\begin{eqnarray}
\label{eq:mtotau}
\nonumber\otimes_{i-1}^{r_k} m_i \to f(y);\quad\quad\quad\quad\quad\quad
\quad\quad\ \ \\
\otimes_{i=r_j+1}^{r_{j-1}} m_i \to \cA(y_j,x_j), j=k,\ldots,2;\\
\nonumber\otimes_{i=r_1+1}^{n} m_i \to\cA(z,x_1),
\quad\quad\quad\quad\quad\quad 
\end{eqnarray}
for a certain sequence of numbers $1\leq r_k\leq\ldots\leq r_1\leq n$ defining the arrow $a^nm\to a^km$  in $\LM$ that is the image of
$t$.

We define an object $t_\phi$ of $K_\phi$
by the arrow $t_\phi:\phi\to\tau(d_\phi)$ where
$d_\phi=(z,z,\id_z)\in K$, so that $\tau(d_\phi)$ is given by
$(f(z),z,\id_{f(z)})$, and $t_\phi:\phi\to\tau(d_\phi)$,
an arrow in $\cF_f$ over the map $a^nm\to am$ induced by $a^n\to a$
in $\LM$, that is the identity on $z$ and 
$\beta:\otimes m_i\to f(z)$ on the $a$-component.

Note that $q(t_\phi)$ is the map 
$\langle n-1\rangle\to\langle 0\rangle$ corresponding to 
$\{0\}\in[n-1]$.

We can now define $\bar F:\Delta^\op\to\Fun(\Lambda^2_0,\Delta^\op)$. Its opposite carries $[k-1]\in\Delta$ to the diagram $[k-1]\to[k]
\leftarrow[0]$ in $\Delta$ with the arrows $\partial^0:[k-1]\to[k]$ and $[0]\to\{0\}\in[k]$.

We claim that for any $t\in K_\phi$ with $p(t)=\langle k-1\rangle$,
there is a $p$-product diagram $t\leftarrow t'\to t_\phi$
over $\bar F(\langle k\rangle)$.

We will now define an arrow $d'\to d$ in $K$ with a decomposition 
of $t:\phi\to\tau(d)$ via $\tau(d')\to\tau(d)$,
as well as a map $d'\to d_\phi$ decomposing $t_\phi$.

We put
$d'=(y,x_k,y_k,\ldots,x_1,z,z)$ together with $\alpha_i:x_i\to y_{i+1}$ and $\id_z:z\to z$. We have
$\tau(d')=(f(y),\cA(y_k,x_k),\ldots,\cA(y_2,x_2),\cA(z,x_1),z)$,
so that the collection of maps (\ref{eq:mtotau}), together with
the unit $\one\to\cA(z,z)$, yields a map that we denote $t':\phi
\to\tau(d')$.
 
The arrow $d'\to d$ in $K$ is given by the commutative diagram
\begin{equation}
\xymatrix{
&\stackrel{y}{\bullet} &\stackrel{x_k}{\circ}\ar_{\alpha_k}[l]
\ar[d]&\stackrel{y_k}{\bullet}&\ldots\ar_{\alpha_{k-1}}[l] &\stackrel{x_1}{\circ}\ar_{\alpha_1}[l]
\ar[d]&\stackrel{z}{\bullet}
& \stackrel{z}{\circ}\ar[l]\\
&\stackrel{y}{\bullet} \ar[u]&\stackrel{x_k}{\circ}\ar_{\alpha_k}[l]&\stackrel{y_k}{\bullet}\ar[u]&\ldots\ar_{\alpha_{k-1}}[l] &\stackrel{x_1}{\circ}\ar_{\alpha_1}[l]& & 
}
\end{equation}
where all unnamed arrows appearing in the diagram are the identity maps.  
The arrow $d'\to d_\phi$ is given by the commutative diagram
\begin{equation}
\xymatrix{
&\stackrel{y}{\bullet} &\stackrel{x_k}{\circ}\ar_{\alpha_k}[l]&\stackrel{y_k}{\bullet}&\ldots\ar_{\alpha_{k-1}}[l] &\stackrel{x_1}{\circ}\ar_{\alpha_1}[l]&\stackrel{z}{\bullet}
& \stackrel{z}{\circ}\ar[l]\ar[d]\\
& & & & & &\stackrel{z}{\bullet} \ar[u]&\stackrel{z}{\circ}\ar[l]
}
\end{equation}
where, once more, all unnamed  arrows are the identity maps.

Thus, for a fixed map $p(d'')\to p(d')$ in $\Delta^\op$, its
lifting in $d''\to d'$ in $K$ is described by the same collection of data as a pair of maps $d''\leftarrow d$
and $d''\to d_\phi$. Therefore, the diagram 
$d\leftarrow d'\to d_\phi$ is a $p$-product diagram in $K$.
For the same reason the diagram
$t\leftarrow t'\to t_\phi$ is a $p$-product diagram.

This proves the cofinality of $\tau:K\to\cF_f$, and, therefore,
Proposition~\ref{prp:olke}.

\subsubsection{Proof of Proposition~\ref{prp:olke-2}}

We prove the claim in two steps. 

1. The first step is to prove that 
$F:\fP\to\fB$ is a map of $\LM$-monoidal categories, that is, 
that it preserves cocartesian arrows.

Let $X=m\oplus x\in\cM\times P=\fP_{lm}$ and let $y=m\otimes x\in P$.
We want to show that $F$ carries the cocartesian arrow $\alpha:X\to y$
to a cocartesian arrow in $\fB$. By definition of $F$,
$F(x)$ is the operadic colimit of the composition
$$
\fA^\act_{/x}\to\fA\stackrel{f}{\to}\fB,
$$
that is $m\oplus F(x)$ is the operadic colimit of the composition
$$
\fA^\act_{/x}\to\fA\stackrel{m\oplus\_}{\to}\fA
\stackrel{f}{\to}\fB.
$$
One has an equivalence 
$\fA^\act_{/X}\to\fA^\act_{/m}\times\fA^\act_{/x}=
\cM^\act_{/m}\times\fA^\act_{/x}$
(see Lemma~\ref{lem:dec} below) defined by the decomposition 
$X=m\oplus x$, so that the map 
$\fA_{/x}\stackrel{m\oplus\_}{\to}\fA_{/X}$ 
is cofinal. 

This implies that $m\oplus F(x)$ is the operadic colimit of
the composition 
$$
\fA^\act_{/X}\to\fA\stackrel{f}{\to}\fB,
$$
that implies that the arrow $F(X)\to F(y)$ is cocartesian.

2. Let $\bar p:K^\triangleright\to\fP^\act$ be an operadic colimit
diagram. We keep in mind that $\fP$ is a $\LM$-monoidal category 
with colimits, so operadic colimits can be expressed in terms of 
colimits. Since we already know that $F$ preserves cocartesian arrows,
it is sufficient to verify the claim in the case $p$ factors through
$P\subset\cP^\act$. Thus, from now on we can assume that 
$\bar p:K^\triangleright\to P$, is a colimit diagram; we have to verify
that its composition with $F$ remains a colimit diagram in $\cB$.
Here we follow the proof of the similar claim in~\cite{L.T}, 5.1.5.5.

For any $p:K\to P$ 
denote $\cE(p)=\fA^\act\times_{\fP^\act}\fP^{\act[1]}\times_{\fP^\act}K$. We have two projections
$s_p:\cE(p)\to\fA^\act$ and $d_p:\cE(p)\to K$. One has a canonical map 
$u_p:Y\circ s_p\to p\circ d_p$ induced by the projection 
$\cE(p)\to \fP^{\act[1]}$.
We assert that the composition  $K\stackrel{p}{\to} P\stackrel{F}{\to}\fB^\act$ 
is a left Kan extension of the composition $f\circ s_p:\cE(p)\to\fA^\act\to\fB^\act$ along $d$.
In fact, the map 
$d_p:\cE(p)\to K$ is a cocartesian fibration so, by \cite{L.T}, 4.3.3.10,
applied to $d_p:\cE(p)\to K$ over $K$,  the assertion can be verified 
fiberwise for each $x\in K$ where it follows from the definition of
$F$ as the left Kan extension of $f$. This establishes a canonical equivalence $\colim(F\circ p)=\colim(f\circ s_p)$.

Note that,  by~\ref{prp:olke}, $\id_\fP$ is an operadic left Kan 
extension of $Y:\fA\to\fP$, so everything we 
said above is valid, in particular, for $F=\id_\fP$.

Let $\bar p:K^\triangleright\to P$ be a colimit diagram and $p=\bar p|_K$. We have to prove that $F\circ\bar p$ is a colimit diagram in $\cB_m$. In what follows we denote $\cE=\cE(p)$ and $\bar\cE=\cE(\bar p)$,
as well as $s=s_p$, $\bar s=s_{\bar p}$, $d=d_p$, $\bar d=d_{\bar p}$,
$u=u_p$ and $\bar u=u_{\bar p}$.

The map $s:\cE(p)\to\fA^\act$ is a cartesian fibration, so, in 
particular, the map $s^\op:\cE(p)^\op\to(\fA^{\act})^\op$, is a smooth map.

Given $a\in\fA^\act$, the fiber $\cE(p)_a=\fP^\act_{Y(a)/}
\times_{\fP^\act}K$ is a cocartesian fibration over $K$ classified by the 
composition 
$K\stackrel{p}{\to}\fP^\act\stackrel{ev_{Y(a)}}{\to}\cS.$

For any $a\in\fA$ the induced map $\cE_a\to\bar\cE_a$ is a weak homotopy equivalence by \cite{L.T}, 3.3.4.5. 
 Therefore, \cite{L.T}, 4.1.2.18 implies that the map 
$\cE\to\bar\cE$ is a contravariant equivalence over $\fA^\act$.

We will now apply 
\cite{L.T}, 5.1.5.4 to deduce the required result. We have a map of 
$\LM$-operads  $f:\fA\to\fB$ and a contravariant equivalence 
$\cE\to\bar\cE$ over $\fA^\act$. 

The morphism of functors $\bar u:Y\circ\bar s\to\bar p\circ\bar d$
gives rise to a composition
$$
Y\circ\bar s\stackrel{\bar u}{\longrightarrow}\bar p\circ\bar d
\to\bar p(*)
$$
to the constant functor with the value at $\bar p(*)\in\fP$, that 
is a functor $v:\bar\cE^\triangleright\to\fP^\act$ carrying the terminal
object to $\bar p(*)$. Let us look at the composition 
$F\circ v:\bar\cE^\triangleright\to\fB$.

This is a colimit diagram as $\colim(f\circ\bar s:\bar\cE\to
\fB)=\colim(F\circ\bar p:K^\triangleright\to\fB)$ that is obviously
$F(\bar p(*))$.

Therefore, the restriction of $F\circ v$ to $\cE^\triangleright$ is 
colimit diagram. This means that $\colim(F\circ p:K\to\fB^\act)=F(*)$.
This completes the proof of~\ref{prp:olke-2}.

\begin{lem}
\label{lem:dec}
Let $\cP\to\cQ$ be a morphism of operads and let $x,y\in\cQ_1$. Then
one has an equivalence
$$
\cP^\act_{/x\oplus y}=\cP^\act_{/x}\times\cP^\act_{/y}.
$$
\end{lem}
\begin{proof}
The category $\cP^\act$ has a symmetric monoidal structure, see
\cite{L.HA}, 2.2.4.3. One has a symmetric monoidal functor 
$\cP^\act\to\cQ^\act$ so that the monoidal operation induces a functor
$\otimes:\cP^\act_{/x}\times\cP^\act_{/y}\to\cP^\act_{/x\oplus y}$.
This is clearly an equivalence.
\end{proof}

\section{Functoriality of $\cQ_{X,\cM,\cB}$}
\label{sec:functorialityQ}

In Section~\ref{sec:wc-mult} we present a monoidal version of
the equivalence~(\ref{eq:y*}). This requires a better understanding of
the functorial properties of this equivalence.  
Our aim is to present the collection of 
$\TEN_\succ$-monoidal categories $\cQ_{X,\cM,\cB}$ defined in
\ref{sss:Q} as a symmetric monoidal functor from
$\Cat\times\Alg_{\LM}(\Cat^L)$ to a certain category of
$\TEN_\succ$-monoidal categories, see~\ref{prp:Q-SM}. The rest of the section is devoted to justifying Corollary~\ref{crl:moncolax-K}
and Lemma~\ref{lem:duality-K} used in the proof of~\ref{prp:Q-SM}.
The proof of Proposition~\ref{prp:cat-coc-K-SM} was suggested to us by the referee.

\subsection{Families of monoidal categories}
\subsubsection{}
\label{sss:toMonlax}
Let $\cP$ be an operad. Recall that the category $\Mon_\cP^\lax$ 
is defined as the full subcategory of the category of $\cP$-operads
$\Op_\cP$ spanned by the $\cP$-monoidal categories. The category $\Mon_\cP^\colax$ of 
$\cP$-monoidal categories and colax monoidal functors can be formally 
defined as follows. First of all, one defines $\Coop_{\cP^\op}$,
the category of $\cP^\op$-cooperads, as the category of functors
$p:C\to\cP^\op$ such that $p^\op:C^\op\to\cP$ is a $\cP$-operad.
The categories $\Coop_{\cP^\op}$ and $\Op_\cP$ are obviously 
equivalent but any $\cP$-monoidal category has both operadic and 
cooperadic realization that are intertwined by the passage from a monoidal category to its opposite. 
If $M$ is a $\cP$-monoidal category, its 
operadic realization is a $\cP$-operad $M^\otimes\to\cP$ that is
a cocartesian fibration obtained by Grothendieck construction
from the functor $\cP\to\Cat$ describing the $\cP$-monoidal structure on $M$. The equivalence between cocartesian fibrations over $\cP$ and the cartesian fibrations over $\cP^\op$ carries $M^\otimes$ to the cooperadic realization 
$^\otimes M\to\cP^\op$ of $M$. The embeddings 
$\Mon^\lax_\cP\to\Cat_{/\cP}$ and 
$\Mon^\colax_\cP\to\Cat_{/\cP^\op}$ preserve products.

The category $\Fun(B^\op,\Mon_\cP^\lax)$ is equivalent, by
Grothendieck construction, to the category of arrows
$p:X\to B\times\cP$, with components $p_B:X\to B$ and $p_\cP:X\to\cP$, satisfying the following properties.
\begin{itemize}
\item[1.] $p_B$ is a cartesian fibration, $p_\cP$ is a cocartesian fibration.
\item[2.] $p$ is a map of cartesian fibrations over $B$ and of cocartesian fibrations over $\cP$.
\item[3.] For any $b\in B$ the fiber $X_b=p_B^{-1}(b)\to\cP$
is a $\cP$-operad.
\item[4.] For any $\beta:b\to b'$ in $B$ the cartesian lifting 
$\beta^!:X_{b'}\to X_b$ is a map of $\cP$-operads.
\end{itemize}
The first two properties mean that $p:X\to B\times\cP$ is a lax bifibration in the sense of \cite{H.D}, 3.1.2.

Recall that $\Cat^\coc_{/B}$ (resp., $\Cat^\cart_{/B}$) denotes the
full subcategory of $\Cat_{/B}$ spanned by the cocartesian (resp., cartesian) fibrations.
\begin{prp}(see~\cite{H.EY}, 2.11.3)
\label{prp:monlax}
One has an equivalence $\Fun(B^\op,\Mon_\cP^\lax)=
\Alg_\cP(\Cat^\cart_{/B})$.
\end{prp}\qed

There is a similar description of $\Fun(B,\Mon_\cP^\colax)$.
This category is equivalent to the category of arrows
$q:X\to \cP^\op\times B$ such that
\begin{itemize}
\item[1.] $q_B$ is a cocartesian fibration, $q_{\cP^\op}$ is
a cartesian fibration.
\item[2.] $q$ is a map of cocartesian fibrations over $B$ and of
cartesian fibrations over $\cP^\op$.
\item[3.] For any $b\in B$ the fiber $X_b=q_B^{-1}(b)\to\cP^\op$
is a $\cP^\op$-cooperad.
\item[4.] For any $\beta:b\to b'$ in $B$ the cocartesian lifting 
$\beta_!:X_{b}\to X_{b'}$ is a map of $\cP$-cooperads.
\end{itemize}

The following result immediately follows from~\ref{prp:monlax}.
\begin{prp}
\label{prp:moncolax}
One has an equivalence $\Fun(B,\Mon_\cP^\colax)=
\Alg_\cP(\Cat^\coc_{/B})$.
\end{prp}
\qed

\ 

Let now $\cP$ and $\cQ$ be two operads. In the lemma below the
categories $\Mon_\cP^\lax$ and $\Mon_\cQ^\colax$ are endowed with 
the cartesian symmetric monoidal structure.
\begin{lem}
\label{lem:duality}
There is an equivalence of categories
$$
\Alg_\cQ(\Mon_\cP^\lax)=\Alg_\cP(\Mon_\cQ^\colax).
$$
\end{lem}
\begin{proof}
The category $\Alg_\cP(\Mon^\colax_\cQ)$ indentifies with the full
subcategory of functors $f:\cP\to\Cat_{/\cQ^\op}$ satisfying the Segal condition and such that for any $p\in\cP$ $f(p)$ is a (cooperad presentation of a) $\cP$-monoidal category. This can be equivalently described as the category
of $(\cQ^\op,\cP)$-lax bifibrations \cite{H.D} that is maps 
$q:X\to \cQ^\op\times\cP$ that are simultaneously map of cartesian fibrations over $\cQ^\op$
and of cocartesian fibrations over $\cP$, satisfying the listed above properties. Equivalently, this can be rewritten in terms of
$(\cQ,\cP)$-Gray fibrations~\cite{HHLN} 
$p:X\to\cP\times\cQ$ satisfying the properties 
\begin{itemize}
\item for any $q\in\cQ$ the fiber $X_q\to\cP$ is a $\cP$-monoidal category.
\item Any decomposition $q=q_1\oplus q_2$ gives rise to a cartesian diagram
$$
\xymatrix{
&X_q \ar[r]\ar[d] &X_{q_1}\ar[d]\\
&X_{q_2}\ar[r] &\cP
}.
$$
\end{itemize}
 Rewriting these as
functors $\cQ\to\Cat_{/\cP}$, we get the subcategory
$\Alg_\cQ(\Mon_\cP^\lax)$.
\end{proof}

\subsection{Monoidal structure on $\Mon_\cP^{\lax/\colax,L}$}
\label{ss:SM-Mon-laxcolax}

The category $\Mon_\cP^L=\Alg_\cP(\Cat^L)$ has a symmetric monoidal structure inherited from that on $\Cat^L$. In this subsection
we show that the lax and the colax versions of this category also have a symmetric monoidal structure.

We denote by $\Mon_\cP^{\lax,\times}$ the
cocartesian fibration over $\Fin_*$ defined by the cartesian SM sructure
on $\Mon^\lax_\cP$~\footnote{Here $\Mon_\cP^\lax$ consists of big
$\cP$-monoidal categories and lax $\cP$-monoidal functors.}.
The subcategory $\Mon_\cP^{\lax,L,\otimes}$ of $\Mon_\cP^{\lax,\times}$
is defined as follows. Its objects are the collections of $\cP$-monoidal
categories with colimits. An arrow $C_1\times\dots C_n\to C$ in $\Mon_\cP^{\lax,\times}$ belongs to $\Mon_\cP^{\lax,L,\otimes}$
if it preserves colimits in each argument.  The composition 
$p:\Mon_\cP^{\lax,L,\otimes}\to\Mon_\cP^{\lax,\times}\to\Fin_*$
is an operad.

\begin{prp}
\label{prp:Mon-lax-K-SM}
$p$ is a cocartesian fibration. This yields a structure
of SM category on $\Mon_\cP^{\lax,L}$. The embedding
$\Mon_\cP^L\to\Mon_\cP^{\lax,L}$ is a symmetric monoidal functor.
\end{prp}
\begin{proof}
Given $A_1,\dots,A_n$ in $\Mon_\cP^{\lax,L}$ let $\otimes A_i$
denote their tensor product in $\Mon_\cP^L$. We have to verify that
the map $\oplus A_i\to\otimes A_i$ is a cocartesian lifting of
the active arrow $\langle n\rangle\to\langle 1\rangle$, that is, 
that the natural map
\begin{equation}
\label{eq:j}
j:\Map_{\Mon_\cP^{\lax,L}}(\otimes A_i,B)\to\Map_{\Mon_\cP^{\lax,L,\otimes}}
(\oplus A_i,B)
\end{equation}
is an equivalence for any $B\in\Mon_\cP^{\lax,L}$. The source of $j$
is a subspace of 
$\Map_{\Op_\cT}(\otimes A_i,B)=\Map_\cP(\cP,\Funop_\cP(\otimes A_i,B))$.
We will use the following properties of $\Funop_\cP$, see~\cite{H.EY},
2.8.9. Let $A,\ B$ be $\cP$-monoidal categories. Then $\Funop_\cP(A,B)$
is a $\cP$-operad whose
objects over $p\in\cP_1$ are the functors $A_p\to B_p$; for an active arrow $\oplus p_j\to q$ in $\cP$ let $f_j:A_{p_j}\to B_{p_j}$
and $g:A_q\to B_q$. Then $\Map_{\Funop_\cP(A,B)}(\oplus f_j,g)$ identifies with the space 
\begin{equation}
\label{eq:mapop}
\Map_{\Fun(\prod A_{p_j},B_q)}(g\circ\alpha^A_!,\alpha^B_!\circ\oplus f_j),
\end{equation}
where the notation is expained by the diagram
$$
\xymatrix{
&\prod A_{p_j} \ar^{\alpha^A_!}[r]\ar^{\oplus f_j}[d] &A_q\ar^g[d]\\
&\prod B_{p_j} \ar^{\alpha^B_!}[r] &B_q
}.
$$

We define  $\Funop'_\cP(\otimes A_i,B)\subset\Funop_\cP(\otimes A_i,B)$
as the full suboperad spanned by the collections of functors 
$\otimes (A_i)_p\to B_p$
preserving the colimits. Then the source of (\ref{eq:j})
identifies with $\Map_\cP(\cP,\Funop'_\cP(\otimes A_i,B))$.

Similarly, we define $\Funop'_\cP(\prod A_i,B)\subset\Funop_\cP(\prod A_i,B)$ as the full suboperad spanned by the collections of functors
$\prod (A_i)_p\to B_p$ preserving the colimits in each argument. Then the target  of (\ref{eq:j})
identifies with $\Map_\cP(\cP,\Funop'_\cP(\prod A_i,B))$.

This implies that, in order to prove our assertion, it is sufficient to
prove that the map of operads
$$
\Funop'_\cP(\otimes A_i,B)\to\Funop'_\cP(\prod A_i,B)
$$
is an equivalence. It is sufficient to prove that this map defines an
equivalence of the spaces of colors (this is obvious) and of the 
spaces of operations $\Map(\oplus f_j,g)$. The verification is straighforward, based on the description (\ref{eq:mapop}).
\end{proof}

A similar symmetric monoidal structure can be defined on the category
$\Mon_\cP^{\colax,L}$ of $\cP$-monoidal categories with  colimits and colax monoidal functors. We define $\Mon_\cP^{\colax,L,\otimes}$ as the suboperad of $\Mon_\cP^{\colax,\times}$ whose objects
are collections of $\cP$-monoidal catgegories with colimits
and arrows preserving colimits in each argument. 
\begin{prp} 
\label{prp:Mon-colax-K-SM}
$\Mon_\cP^{\colax,L,\otimes}$ is a symmetric monoidal category so that
the embedding $\Mon_\cP^L\to\Mon_\cP^{\colax,L}$ becomes a symmetric monoidal functor.
\end{prp}
\begin{proof}
Passing to the opposite monoidal categories, we are back to a suboperad
of $\Mon_\cP^{\lax,\times}$, but this time working with limits
instead of colimits. The proof of \ref{prp:Mon-lax-K-SM}
now works without change.
\end{proof}

\subsection{The category $\Cat^{\coc,L}_{/B}$}
Here we define the category of cocartesian fibrations with the
fibers having colimits. 
\subsubsection{}Let $B$ be a category.
We denote by $\Cat^\coc_{/B}$ the full subcategory of $\Cat_{/B}$
spanned by the cocartesian fibrations $X\to B$. 

The subcategory $\Cat^{\coc,L}_{/B}$ of $\CAT^\coc_{/B}$ is spanned by the cocartesian fibrations
$p:X\to B$ classified by a functor $B\to\Cat^L\subset\CAT$. The arrows
in $\Cat^{\coc,L}_{/B}$ are those preserving the colimits on the fibers. 

\subsubsection{Relative presheaves} 
\label{sss:relativepre}
 
A category $X$ over $B$ is called flat if 
the functor $\times_BX$ preserves colimits. In this case the right adjoint functor $Y\mapsto\Fun^B(X,Y)$ is defined. Any cocartesian fibration is flat. One has
\begin{lem}
Let $X$ be a cocartesian fibration over $B$ classified by 
a functor $\cX:B\to\Cat$. Then $\Fun^B(X^\op,B\times\cS)$ is a 
a cocartesian fibration classified by the composition
\begin{equation}
\label{eq:PBX}
B\stackrel{\cX}{\to}\Cat\stackrel{P}{\to}\Cat^L.
\end{equation}
\end{lem}
Having in mind this description, we will denote by $P_B(X)$ the obtained cocartesian fibration. This is the category of relative presheaves in $X$.
\begin{proof}
If $X\to B$ is flat, $X^\op\to B^\op$ is also flat, and obviously
$\Fun^B(X,Y)^\op=\Fun^{B^\op}(X^\op,Y^\op)$. Passing to the opposite categories, we deduce from \cite{L.T}, 3.2.2.13, that 
$\Fun^B(X^\op,B\times\cS)$ is cartesian over $B$. Its fiber at
$b\in B$ is $\Fun(X_b^\op,\cS)$, so the cartesian fibration in
question is classified by the functor $B^\op\to\Cat$ carrying
$b$ to $\Fun(X_b^\op,\cS)$. This cartesian fibration is also cocartesian and as such it is classified by~(\ref{eq:PBX}).
\end{proof}

\begin{lem}
\label{lem:PBuni}
The functor $P_B:X\mapsto P_B(X)$ is left adjoint to the forgetful functor $\Cat^{\coc,L}_{/B}\to\Cat^\coc_{/B}$.
\end{lem}
\begin{proof}
Let $X\in\Cat^\coc_{/B}$ and $Y\in\Cat^{\coc,L}_{/B}$. We have to verify that the composition with the Yoneda embedding induces an
equivalence
$$
\Map_{\Cat^{\coc,L}_{/B}}(P_B(X),Y)\to
\Map_{\Cat^{\coc}_{/B}}(X,Y).
$$
A standard reasoning of \cite{L.T}, 5.1.5.5 proves that a map
$P_B(X)\to Y$ over $B$ preserves vertical colimits iff it is a
relative left Kan extension of its restriction to $X$. Since,
by \cite{L.T}, 4.3.2.14, any functor $X\to Y$ over $B$ has a relative
left Kan extension, this implies the result.
\end{proof}

\subsection{SM structure on $\Cat^{\coc,L}_{/B}$} 
The category $\Cat^\coc_{/B}$ has a cartesian SM structure. We denote by
$\Cat^{\coc,\times}_{/B}$ the corresponding category over $\Fin_*$.

Denote by $\Cat^{\coc,L,\otimes}_{/B}$
the subcategory of $\CAT^{\coc,\times}_{/B}$ whose objects over
$\langle n\rangle\in\Fin_*$ are collections $(X_i\to B_i)_{i=1,\ldots,n}$
of objects in $\Cat^{\coc,L}_{/B}$ and morphisms are those preserving 
colimits in each argument. The composition $p:\Cat^{\coc,L,\otimes}_{/B}
\to\Fin_*$ is obviously an operad.

We have
\begin{prp}
\label{prp:cat-coc-K-SM}
\begin{itemize}
\item[1.] $p$ is a cocartesian fibration. This yields a structure
of SM category on $\Cat^{\coc,L}_{/B}$.
\item[2.] For $X_1,\dots,X_n\in\Cat^{\coc,L}_{/B}$ 
and $b\in B$ one has
$$
(X_1\otimes\ldots\otimes X_n)_b=(X_1)_b\otimes\dots\otimes(X_n)_b.
$$
\end{itemize}
\end{prp}
\begin{proof}
To prove the first claim,
we present for each collection $X_1,\dots,X_n\in\Cat^{\coc,L}_{/B}$ 
 a universal arrow
$$
\alpha:\oplus_{i=1}^nX_i\to X
$$
in $\Cat^{\coc,L,\otimes}_{/B}$.  If $X_i\to X$
is classified by a functor $F_i:B_i\to\Cat^L$, we define $p:X\to B$
as the cocartesian fibration classified by the functor $F:B\to\Cat^L$
given by the formula $F(b)=F_1(b)\otimes\ldots\otimes F_n(b)$. 
The arrow $\oplus X_i\to X$ is 
obviously defined. We will now show that it is cocartesian in 
$\Cat^{\coc,L,\otimes}_{/B}$, that is 
that for any $Y\in\Cat^{\coc,L}_{/B}$ $\alpha$ induces an equivalence
\begin{equation}
\label{eq:hompq}
\alpha^*:\Map_{\Cat^{\coc,L}_{/B}}(X,Y)\to
\Map^\act_{\Cat^{\coc,L,\otimes}_{/B}}(\oplus X_i, Y).
\end{equation}

Let us first assume that $Y\to B$ is both cocartesian and cartesian
fibration~\footnote{We thank the referee for suggesting this idea of the proof 
of~\ref{prp:cat-coc-K-SM}.}. In this case we will be able to present the source and the 
target of (\ref{eq:hompq}) as subspaces of the space of sections of 
cocartesian fibrations
over $B$. The target space is embedded into
$$\Map^\act_{\Cat^{\coc,\otimes}_{/B}}(\oplus X_i, Y)
$$
that is the space of sections of $\Fun^B(\prod X_i,Y)\to B$
(here $\prod$ denotes the product over $B$)
which is a cocartesian fibration by \cite{L.T}, 3.2.2.13 classified by the functor carrying $b\in B$ to $\Fun(\prod (X_i)_b,Y_b)$. Similarly, the source space is embedded into $\Map_{\Cat^\coc_{/B}}(X,Y)$
that is the space of sections of the cocartesian fibration
$\Fun^B(X,Y)\to B$ classified by the functor $b\mapsto\Fun(\otimes (X_i)_b,Y_b)$. One readily sees that the source  and the target
of (\ref{eq:hompq}) both identify with the space of sections of
the cocartesian fibration classified by the functor 
$b\mapsto\Fun^L(\otimes (X_i)_b,Y_b)$. This proves the assertion 
in the case when $q:Y\to B$ is a cartesian fibration.

To prove that $(\ref{eq:hompq})$ is an equivalence for general $q$,
we embed $Y$ into $P_B(Y)$  as in~\ref{sss:relativepre}. This is
a full embedding. This allows one to
identify the left and the right hand sides of (\ref{eq:hompq}) with the subcategories of the similar expressions for $P_B(Y)$.
\end{proof}

\begin{lem}
\label{lem:PisSM}
The functor $P_B:\Cat^\coc_{/B}\to\Cat^{\coc,L}_{/B}$ is symmetric monoidal.
\end{lem}
\begin{proof}
We  define $\cC$ as the full subcategory of $\Fun([1],\CAT_{/B})$ spanned
by the arrows $X\to P$ where $X\in\Cat^\coc_{/B}$ and $P\in\Cat^{\coc,L}_{/B}$. Denote
by $p_0:\cC\to\Cat$ and $p_1:\cC\to\Cat^L$ the obvious projections.
$\cC$ is closed under finite products and we denote by $\cC^\times$
the corresponding category over $\Fin_*$. We denote by 
$\cD^\otimes\subset\cC^\times$ the subcategory whose objects
are collections of the arrows $f:X\to P$ inducing an equivalence 
$P_B(X)\to P$. Arrows in $\cD^\otimes$ are defined by the diagrams
$$
\xymatrix{
&\prod X_i\ar^a[r]\ar^{\prod f_i}[d] & X\ar^f[d]\\
&\prod P_i\ar^b[r] &P
}
$$
where $b$ preserves colimits in each argument. One has obvious functors
$p_0:\cD^\otimes\to\Cat^{\coc,\times}_{/B}$ and 
$p_1:\cD^\otimes\to\Cat^{\coc,L,\otimes}_{/B}$. We will now show that 
$\cD^\otimes$ is a SM category. Given a collection $f_i:X_i\to P_i$,
we define a map $u:\oplus f_i\to f$ in $\cD^\otimes$ lifting 
the active arrow $\langle n\rangle\to\langle 1\rangle$ in $\Fin_*$ as
follows. We put $X=\prod X_i$ and $a=\id_X$. We put $P=\otimes P_i$ with the canonical map $b:\prod P_i\to\otimes P_i$.  The map $u$ so defined
is a cocartesian lifting; this can be verified by induction in $n$,
using universality of $P_B$ proven in~\ref{lem:PBuni}. Now the functor
$p_0:\cD^\otimes\to\Cat^\times$ is a SM functor that is an equivalence
of categories. Therefore it is an equivalence of SM categories.
Finally $p_1$ is also symmetric monoidal, and this proves the assertion. 
\end{proof}

\begin{lem}
\label{lem:dualpresheaves}
Let $X\in\Cat^\coc_{/B}$. Then $P_B(X)$ is dualizable in
$\Cat^{\coc,L}_{/B}$ with dual $P_B(X^\op)$.
\end{lem}
\begin{proof}
Yoneda embedding $Y:X\to P_B(X)$ gives rise to a morphism
$X\times_BX^\op\to\cS$ in $\Cat^\coc_{/B}$. This induces a map
$e:P_B(X\times_B X^\op)=P_B(X^\op)\otimes P_B(X)\to\cS$ in 
$\Cat^{\coc,L}_{/B}$. The map $c:\cS\to P_B(X^\op)\otimes P_B(X)=P_B(X^\op\times X)$ is uniquely defined by the final section
of $P_B(X^\op\times X)$. The maps $c$ and $e$ are a unit and a counit of adjunction.

\end{proof}

\subsubsection{$\Cat^{\cart,L}_{/B}$}

Everything said about the category of cocartesian fibrations with colimits holds also for the category of cartesian fibrations with colimits. Passing to opposite categories establishes and equivalence
of $\Cat^\cart_{/B}$ with $\Cat^\coc_{/B^\op}$. This equivalence carries cartesian fibrations with colimits to cocartesian fibrations with 
limits. Fortunately, the proof of 
Proposition~\ref{prp:cat-coc-K-SM} is valid when one replaces the
colimits with the limits.

\

We will now deduce from~\ref{prp:monlax}  the version for the categories 
with colimits as follows. The equivalence~\ref{prp:monlax} restricts to 
the following.
\begin{crl}
\label{crl:monlax-K}
The equivalence~\ref{prp:monlax} induces an equivalence
$\Fun(B^\op,\Mon_\cP^{\lax,L})=
\Alg_\cP(\Cat^{\cart,L}_{/B})$.
\end{crl}
\begin{proof}
The left-hand side is a subcategory of $\Fun(B^\op,\Mon_\cP^\lax)$,
with the objects consisting of $\cP$-monoidal categories with colimits and arrows preserving these colimits. The embedding
$\Cat^{\cart,L,\otimes}_{/\cP}\to\Cat^{\cart,\times}_{/\cP}$
defines $\Alg_\cP(\Cat^{\cart,L}_{/B})$ as a the same 
subcategory of $\Alg_\cP(\Cat^{\cart}_{/B})$.
\end{proof}
Similarly, one has
\begin{crl}
\label{crl:moncolax-K}
The equivalence~\ref{prp:moncolax} induces an equivalence
$\Fun(B,\Mon_\cP^{\colax,L})=
\Alg_\cP(\Cat^{\coc,L}_{/B})$.
\end{crl}\qed
 
Lemma~\ref{lem:duality} has also a version for categories with colimits.
\begin{lem}
\label{lem:duality-K}
The equivalence~\ref{lem:duality} induces an equivalence
\begin{equation}
\label{eq:duality}
\Alg_\cQ(\Mon_\cP^{\lax,L})=\Alg_\cP(\Mon_\cQ^{\colax,L}).
\end{equation}
\end{lem}
\begin{proof}
Both the left and the right hand side are subcategories of
$\Alg_\cQ(\Mon_\cP^\lax)=\Alg_\cP(\Mon_\cQ^\colax)$. 
A $\cQ$-algebra in $\Mon_\cP^\lax$ represented by a functor 
$f:\cQ\to\Mon_\cP^\lax$ belongs to $\Alg_\cQ(\Mon_\cP^{\lax,L})$ 
if the following three properties are satisfied.
\begin{itemize}
\item[1.] For any $p\in\cP_1$, $q\in\cQ_1$ one has $f(q)_p\in\Cat^L$.
\item[2.] For $p=p_1\oplus\dots\oplus p_n$ with $r, p_i\in\cP_1$,
an active arrow $\alpha:p\to r$ in $\cP$, the composition 
$$
\prod_i f(q)_{p_1}\stackrel{\rho^{-1}}{\to} f(q)_p\stackrel{\alpha_!}{\to} f(q)_r
$$
preserves colimits in each argument. Here $\rho:f(q)_p\to
\prod f(q)_{p_i}$ is an equivalence since $f(q)$ satisfies the Segal condition in $p$.
\item[3.] For  $q=q_1\oplus\dots\oplus q_n$, $s,q_i\in\cQ_1$, $p\in\cP_1$
and an active arrow $\beta:q\to s$ in $\cQ$, the composition 
$$
\prod_i f(q_i)_p\stackrel{\rho^{-1}}{\to} f(q)_p\stackrel{\alpha_!}{\to} f(s)_p
$$
preserves colimits in each argument. Here 
$\rho:f(q)_p\to\prod f(q_i)_p$ is an equivalence since $f$ satisfies the Segal condition in $q$.
\end{itemize}
The same three conditions define the subcategory 
$\Alg_\cP(\Mon_\cQ^{\colax,L})$ of $\Alg_\cP(\Mon_\cQ^\colax)$.

\end{proof}

\subsection{The functor $\cQ:\Cat\times\Alg_\LM(\Cat^L)\to\Mon^\colax_{\TEN_\succ}$}
\label{ss:Q-SM}

In this subsection we prove the following
\begin{prp}
\label{prp:Q-SM}
The assignment $(X,\cM,\cB)\mapsto \cQ_{X,\cM,\cB}$
gives rise to a symmetric monoidal functor 
$\cQ:\Cat\times\Alg_\LM(\Cat^L)\to\Mon^{\colax,L}_{\TEN_\succ}$.
\end{prp}
Note that the SM structure on $\Alg_\LM(\Cat^L)$ is induced from that on $\Cat^L$ and the structure on $\Mon^{\colax,L}_{\TEN_\succ}$
is  defined as in~\ref{ss:SM-Mon-laxcolax}. In the rest of this
section we use the notation $\cT=\TEN_\succ$.
\begin{proof}
Recall \ref{sss:Q} that the $\cT$-monoidal category 
$\cQ_{X,\cM,\cB}$ is constructed in three steps, the first one 
assigning to $(X,\cM,\cB)$ the dualizable right $\cM$-module 
$\Fun(X,\cM)$, the second assigning to it the corresponding counit diagram (\ref{eq:eval-quiv}) and, finally (\ref{eq:eval-quiv-2}),
tensoring it with $\cB$.  
We will follow \ref{sss:Q} and \ref{crl:moncolax-K} and present a $\cT$-
algebra object in $\cC:=\Cat^{\coc,L}_{/B}$ with 
$B=\Cat\times\Alg_{\LM}(\Cat^L)$. 
We denote by $X\in\Cat^\coc_{/B}$ the tautological family defined by the
projection $B\to\Cat$, use $P_B(X)$ and $P_B(X^\op)$ instead of $P(X)$ and $P(X^\op)$ and define $\cM$ and $\cB$ to be the cocartesian 
fibrations classified by the projections of $B$ to the 
components of $\Alg_\LM(\Cat^L)$.  Lemma~\ref{lem:dualpresheaves},
gives us a $\cT$-algebra object in $\Cat^{\coc,L}_{/B}$,
that is, by~\ref{crl:moncolax-K}, a functor
\begin{equation}
\label{eq:q}
\cQ:\Cat\times\Alg_\LM(\Cat^L)\to\Mon^{\colax,L}_\cT
\end{equation}
carrying the triple $(X,\cM,\cB)$ to the $\cT$-monoidal category $\cQ_{X,\cM,\cB}$.

We will now show that the functor (\ref{eq:q}) canonically extends to
a symmetric monoidal functor. 
We will present $\cQ$ as a tensor product of two symmetric monoidal functors, $\cQ_0:\Cat\to\Mon_\cT^{\colax,L}$ and 
$\Pi:\Alg_\LM(\Cat^L)\to\Mon_\cT^{\colax,L}$.
The functor $\Pi$ is defined by the map of operads $\pi:\cT\to\LM$
carrying the colors $a,a',m,b$ to $A\in[\LM]$ and $n,k$ to $M\in[\LM]$.
The functor $\Pi$ is the composition of $\pi^*:\Alg_\LM(\Cat^L)\to
\Alg_\cT(\Cat^L)$ and the obvious SM embedding
$\Alg_\cT(\Cat^L)\to\Mon_\cT^{\colax,L}$.

It remains to construct the SM functor $\cQ_0$.
The $\cT$-monodial category $\cQ_0(X)$ is defined by 
the formulas
$$
A=A'=\cS,\ M=P(X),\ N=P(X^\op),\ B=\Quiv_X(\cS), K=\cS.
$$
Our aim is to canonically extend this to a SM functor.  
We construct a new category $\cC$
as follows. Its objects are collections consisting of a category $X$, a 
$\cT$-monoidal category $(M,B,N,K)$,  a functor $i:X\to M$ 
and an arrow $k:[0]\to K$ (that is, an object of $K$).
The category $\cC$ has a cartesian SM structure.
 
We now define a subcategory $\cD^\otimes$ of $\cC^\times$ as follows.
The objects of $\cD^\otimes$ are the (collection of) objects of $\cC$
satisfying the following properties.
\begin{itemize}
\item $(M,B,N,K)$ is a $\cT$-monoidal category with colimits.
\item $i:X\to M$ presents $M$ as $P(X)$ and $k\in K$ induces an equivalence $\kappa:\cS\to K$.
\item The right action of $B$ on $M$ determines an equivalence $B^\op=\End(M)$.
\item The composition $M\times N\to K\stackrel{\kappa^{-1}}{\to}\cS$ establishes $N$ as right dual to  $M$.
\end{itemize}
The morphisms are the maps $C_1\times\dots\times C_n\to C$ in $\cC^\times$
with $C,C_i$ satisfying the properties as above and such that the
corresponding maps $\prod M_i\to M$, etc., $\prod K_i\to K$, preserve
colimits in each argument.

One can easily see that the forgetful functor 
$\cD^\otimes\to\Cat^\times$ is an equivalence.
The composition $\cD^\otimes\to\cC^\times\to\Mon_\cT^{\colax,\times}$ 
induces a SM functor $\cD^\otimes\to\Mon_\cT^{\colax,L,\otimes}$.

\end{proof}
 
\section{Multiplicative structures}
\label{sec:wc-mult}

In this section we present a
monoidal version of the universality Theorem~\ref{thm:universal}.

Let $\cO$ be an operad. 
Assume that $\cM$ is an $\cO$-algebra object in $\Alg_\Ass(\Cat^L)$.
Equivalently, we assume that $\cM\in\Alg_{\cO\otimes\Ass}(\Cat^L)$. 
In this case many objects mentioned above acquire a structure of 
$\cO$-algebra~\footnote{This sentence becomes slightly imprecise
if $\cO$ is not monochrome: an $\cO$-algebra consists of
more than one object.}. For instance, the notion of $\cO$-monoidal
$\cM$-enriched category makes sense, as well as the notion of 
$\cO$-monoidal left $\cM$-module. We prove that if $\cA$ is an $\cO$-monoidal
$\cM$-enriched category, the enriched presheaves $P_\cM(\cA)$ 
form an $\cO$-monoidal left $\cM$-module category and Yoneda embedding
is an $\cO$-monoidal $\cM$-functor universal among such functors
to $\cO$-monoidal left $\cM$-modules with colimits.

\subsection{From associative algebras to left modules}

\subsubsection{}

Recall \cite{L.HA}, 2.10, that, given a bilinear map of operads $\mu:\cP\times\cQ\to\cR$, and an $\cR$-operad 
$\cX$, the $\cP$-operad
$p:\Alg^\mu_{\cQ/\cR}(\cX)\to\cP$ is defined as the object
representing the functor
$$
K\in\Cat_{/\cP}\mapsto
\Map_{\Cat^+_{/\cR^\natural}}(K^\flat\times\cQ^\natural,\cX^\natural),
$$
see \cite{H.EY}, 2.10.
In the case $\cX$ is $\cR$-monoidal, 
$\Alg^\mu_{\cQ/\cR}(\cX)$ is $\cP$-monoidal. In the case 
$\mu$ represents $\cR$ as a tensor product
of $\cP$ and $\cQ$, one has an equivalence
$$\Alg_\cP(\Alg^\mu_{\cQ/\cR}(\cX))=\Alg_\cR(\cX).$$
We will suppress the letter $\mu$ from the notation
if it is clear from the context.

We will need the following general claim about cocartesian fibrations.

\begin{lem}
\label{lem:map-coc}
Let 
$$
\xymatrix{
&\cP\ar^f[rr]\ar_p[rd]& &\cQ\ar^q[ld]\\
& &B &
}
$$
be a map of cocartesian fibrations over $B$.
Assume that
\begin{itemize}
\item[(1)] For any $b\in B$ $f_b:\cP_b\to\cQ_b$ is a cocartesian fibration.
\item[(2)] For any $\alpha:b\to b'$ in $B$ the functor
$$\alpha_!:\cP_b\to\cP_{b'}$$
carries $f_b$-cocartesian arrows to $f_{b'}$-cocartesian
arrows.
\end{itemize}
Then $f$ is a cocartesian fibration.
\end{lem}
\begin{proof}
By \cite{L.T}, 2.4.2.11, $f$ is a locally cocartesian fibration, with locally cocartesian arrows of the form
$u=u''\circ u'$ where $u'$ is $p$-cocartesian and
$u''\in\cP_{b'}$ $f_{b'}$-cocartesian, with $p(u):b\to b'$.
Condition (2) ensures that the composition of locally cocartesian arrows is locally cocartesian. This implies the claim.
\end{proof}

\

Let $\cO$ be an operad and let $\cC\in\Alg_{\cO\otimes\LM}(\Cat^L)$.

\begin{prp}
\label{prp:forgetfulcocartesian}
The forgetful functor 
\begin{equation}
\label{eq:relAlgLM}
\Alg_{\LM/\cO\otimes\LM}(\cC)\to
\Alg_{\Ass/\cO\otimes\LM}(\cC)
\end{equation}
is an $\cO$-monoidal cocartesian fibration.
\end{prp}
\begin{proof}
This result is very close to~\cite{L.HA}, 4.5.3.
We will apply Lemma~\ref{lem:map-coc} to the forgetful functor 
$f:\Alg_{\LM/\cO\otimes\LM}(\cC)\to
\Alg_{\Ass/\cO\otimes\LM}(\cC)$ 
over $B:=\cO$.

For $o\in\cO$ we denote $\cC_o$ the $\LM$-monoidal 
category obtained from $\cC$ by the base change 
along $\LM\stackrel{o}{\to}\cO\otimes\LM$.

The fiber $f_o$ of $f$ at $o\in\cO_1$ is a cartesian fibration
as it is a forgetful functor  
$\Alg_\LM(\cC_o)\to\Alg_\Ass(\cC_o)$.
Since these are categories with geometric realizations,
$f_o$ is also a cocartesian fibration. The same is true 
for $f_o$ at any $o\in\cO$ as cocartesian fibrations
are closed under products. This proves the condition (1)
of Lemma~\ref{lem:map-coc}. 
Let us verify the condition (2).
Given $\alpha:o\to o'$ in $\cO$, the functor
$$\alpha_!:\Alg_\LM(\cC_o)\to\Alg_\LM(\cC_{o'})$$
is induced by the colimit preserving $\LM$-monoidal functor
$\cC_\alpha:\cC_o\to\cC_{o'}$ induced by $\alpha$.
An arrow $(A,M)\to (B,N)$ in $\Alg_\LM(\cC_o)$
is $f_o$-cocartesian if it induces an equivalence
$B\otimes_AM\to N$. Thus, $\alpha_!$ preserves this property.
\end{proof}
The forgetful functor (\ref{eq:relAlgLM}) 
 is classified by a lax $\cO$-monoidal 
functor 
$$\LMod:\Alg_{\Ass/\cO\otimes\LM}(\cC)\to\Cat,$$
see~Prop. A.2.1. of \cite{H.R}.

In particular, we have the following.
\begin{crl}
\label{crl:lmod-o-algebra}
The functor $\LMod$ defined as above assigns an $\cO$-monoidal category of left modules
$\LMod_A(\cC)$ to an $\cO$-algebra object $A$ in 
$\Alg_{\Ass/\cO\otimes\LM}(\cC)$. A morphism of $\cO$-algebra
objects $A\to A'$ gives rise to an $\cO$-monoidal functor
$\LMod_A(\cC)\to\LMod_{A'}(\cC)$.
\end{crl}\qed
\begin{rem}
It is worthwhile, in order to keep track of what we are doing, to repeat the above construction in down-to-earth terms.

An $\cO$-algebra object in $\Alg_{\Ass/\cO\otimes\LM}(\cC)$
consists of a collection of associative algebra objects
$A_o$ in $\LM$-monoidal categories $\cC_o$, with operations
$\otimes A_{o_i}\to A_o$ corresponding to each operation
$\alpha:(o_1,\ldots,o_n)\to o$ in $\cO$. The category $\LMod_A(\cC)$
has components $\LMod_A(\cC)_o=\LMod_{A_o}(\cC_o)$, and 
the operation $\alpha:(o_1,\ldots,o_n)\to o$ in $\cO$
assigns to a collection $M_i\in\LMod_{A_{a_i}}(\cC_{a_i})$
the pushforward of $\otimes A_{o_i}$-module $\otimes M_i$ along
$\otimes A_{o_i}\to A_o$.
\end{rem}

We apply Corollary~\ref{crl:lmod-o-algebra} a number of times.

\subsection{$\cO$-monoidal left-tensored categories}
\subsubsection{}
For $\cM\in
\Alg_{\cO\otimes\Ass}(\Cat^L)$ the category
$\LMod_\cM$ is $\cO$-monoidal. An $\cO$-monoidal
left $\cM$-module $\cB$ is defined as an $\cO$-algebra
in $\LMod_\cM$.

An $\cO$-monoidal left $\cM$-module $\cB$ defines an 
$\cO\otimes\LM$-operad $(\cM,\cB)$. It makes sense to talk about
$\cO$-monoidal and lax $\cO$-monoidal functors between $\cO$-monoidal
left $\cM$-modules.

\begin{dfn}
\label{dfn-funOBBlax}
For $\cB,\cB'\in\Alg_\cO(\LMod_\cM)$ we define
$\Fun_{\LMod_\cM}^{\cO,\lax}(\cB,\cB')$ as the full subcategory
of the fiber of the forgetful functor
$$
\Alg_{(\cM,\cB)/\cO\otimes\LM}(\cM,\cB')\to
\Alg_{\cM/\cO\otimes\Ass}(\cM)
$$ 
at $\id_\cM$, spanned by the maps
of $\cO\otimes\LM$-operads $f=(\id_\cM,f_m):(\cM,\cB)\to(\cM,\cB')$
satisfying the conditions
\begin{itemize}
\item $f$ preserves cocartesian liftings of the arrows 
in $\LM$.
\item $f_m$ preserves colimits.
\end{itemize}
\end{dfn}
\begin{dfn}
\label{dfn-funOBB}
A lax $\cO$-monoidal  morphism of left $\cM$-modules
$f:(\cM,\cB)\to(\cM,\cB')$ is called $\cO$-monoidal if it preserves
cocartesian liftings of all arrows in $\cO\otimes\LM$.
\end{dfn}
The full subcategory of $\Fun^{\cO,\lax}_{\LMod_\cM}(\cB,\cB')$
spanned by $\cO$-monoidal arrows is denoted by
$\Fun^{\cO}_{\LMod_\cM}(\cB,\cB')$.

\subsection{$\cO$-monoidal enriched categories}
\subsubsection{}
\label{sss:Omec}

A lax SM functor 
$$ \quiv:\Op_\Ass\to\Fam^\cart\Op_\Ass,$$
as well as its relatives 
$ \quiv^\LM:\Op_\LM\to\Fam^\cart\Op_\LM$ and
$ \quiv^\BM:\Op_\BM\to\Fam^\cart\Op_\BM$,
have been constructed in ~\cite{H.EY}, 3.5.2.  
Since the functor $\Alg_\Ass:\Fam^\cart\Op_\Ass\to\Cat$
preserves the limits, the composition
$\Alg_\Ass\circ\quiv:\Op_\Ass\to\Cat$ 
carrying $\cM\in\Op_\Ass$ to the category $\PCat(\cM)$
of $\cM$-enriched precategories, is lax symmetric monoidal.
The same holds for the compositions
$\Alg_\LM\circ\quiv^\LM:\Op_\LM\to\Cat$ and
$\Alg_\BM\circ\quiv^\BM:\Op_\BM\to\Cat$.

Fix an operad $\cO$. Any $\cO$-algebra $\cM$ in $\Op_\Ass$
(for instance, $\cM\in\Alg_{\cO\otimes\Ass}(\Cat^L)$),
gives rise to an $\cO$-monoidal category $\PCat(\cM)$.

Similarly,  an $\cO$-algebra $\cM$ in $\Op_\LM$
(for instance, $\cM\in\Alg_{\cO\otimes\LM}(\Cat^L)$),
gives rise to an $\cO$-monoidal category $\PCat^\LM(\cM)$.
\begin{dfn}
Let $\cM\in\Alg_{\cO\otimes\Ass}(\Cat^L)$. An $\cO$-monoidal 
$\cM$-enriched precategory is an $\cO$-algebra object in $\PCat(\cM)$.
\end{dfn}

Note that a precategory $\cA\in\PCat(\cM)$ is an associative algebra in the family of planar operads $\Quiv_X(\cM)$ parametrized by $X\in\Cat$.
Thus, an $\cO$-algebra in $\PCat(\cM)$ has automatically an $\cO$-monoidal category $X$ of objects.

\subsubsection{Day convolution} 
\label{sss:lm-version}

Let $(\cM,\cB)\in\Alg_{\cO\otimes\LM}(\Cat^L)$. 
This means that $\cM$ is an $\cO\otimes\Ass$-monoidal
category with colimits and $\cB$ is an $\cO$-monoidal
category with colimits, left-tensored over $\cM$.

By~\ref{sss:Omec}, the category 
$\PCat^\LM(\cM,\cB)=\Alg_\LM(\Quiv^\LM(\cM,\cB))$ is $\cO$-monoidal.
 
By~\ref{prp:forgetfulcocartesian}, the forgetful functor
\begin{equation}
\label{eq:pcatlmtopcat}
\PCat^\LM(\cM,\cB)\to
\PCat(\cM)
\end{equation}
is an $\cO$-monoidal cocartesian fibration.
Therefore, it is classified by the lax $\cO$-monoidal
functor $\PCat(\cM)\to\Cat$
carrying $\cA\in\PCat(\cM)$ to $\Fun_\cM(\cA,\cB)$.

In particular, for any $\cO$-monoidal $\cM$-enriched
category $\cA$ the category $\Fun_\cM(\cA,\cB)$ is $\cO$-monoidal.
This is an enriched form of the Day convolution~\cite{L.HA}, 2.6.

Let us repeat that if $\cO$ is not monochrome, 
$\cM$ is actually a collection of monoidal categories,
$\PCat(\cM)$ is actually a collection of categories, etc.

\subsubsection{}
Let $X\in\Cat$ and let $\one_X$ denote the unit of the monoidal
category $\Quiv_X(\cM)$, see~\cite{H.EY}, 4.7.3. 
By definition, $\Fun_\cM(\one_X,\cB)=
\Fun(X,\cB)$. If now $\cM\in\Alg_{\cO\otimes\Ass}(\Cat^L)$, any 
$\cO$-monoidal category $X$ gives rise to an $\cO$-monoidal 
$\cM$-category $\one_X$. The enriched Day convolution defines an 
$\cO$-monoidal structure on $\Fun_\cM(\one_X,\cB)$. The internal Hom
in operads~\cite{L.HA}, 2.6 defines an $\cO$-monoidal structure on 
$\Fun(X,\cB)$. These two $\cO$-monoidal structures coincide, according to the following lemma.
\begin{lem}
The forgetful functor $\Fun_\cM(\one_X,\cB)\to\Fun(X,\cB)$ is an
equivalence of $\cO$-monoidal categories.
\end{lem}
\begin{proof}
The $\cO$-monoidal structure on $\Fun_\cM(\one_X,\cB)$ 
is induced from the identification
$$\Fun_\cM(\one_X,\cB)=\Alg_\LM(\Quiv^\LM_X(\cM,\cB))
\times_{\Alg_\Ass(\Quiv_X(\cM))}\{\one_X\}
$$
and an $\cO$-algebra structure on 
$$
\Quiv^\LM_X(\cM,\cB)=(\Quiv_X(\cM),\Fun(X,\cB).
$$
This immediately implies that the forgetful functor
$\Fun_\cM(\one_X,\cB)\to\Fun(X,\cB)$ is $\cO$-monoidal. Since it is
an equivalence, it is an $\cO$-monoidal equivalence.

\end{proof}

\begin{dfn} 
Given $(\cM,\cB)\in\Alg_{\cO\otimes\LM}(\Cat^L)$
and an $\cO$-monoidal $\cM$-enriched category $\cA$, 
a {\sl lax} $\cO$-monoidal $\cM$-functor $f:\cA\to\cB$ is an 
$\cO$-algebra in $\Fun_\cM(\cA,\cB)$.
\end{dfn}
We denote by $\Fun^{\cO,\lax}_\cM(\cA,\cB)=\Alg_\cO(\Fun_\cM(\cA,\cB))$
the category of lax $\cO$-monoidal $\cM$-functors from $\cA$ to $\cB$.

\subsubsection{$\cO$-monoidal $\cM$-functors}

Let $(\cM,\cB)\in\Alg_{\cO\otimes\LM}(\Cat^L)$
and let $\cA$ be an $\cO$-monoidal $\cM$-enriched category.

One has a canonical
$\cO$-monoidal functor $1:\one_X\to\cA$ for any $\cO$-monoidal enriched precategory with the $\cO$-monoidal category of objects $X$. It induces a  forgetful functor 
 $\Fun_\cM(\cA,\cB)\stackrel{1^*}{\to}\Fun_\cM(\one_X,\cB)=\Fun(X,\cB)$. It is automatically lax $\cO$-monoidal since it is right adjoint to an $\cO$-monoidal free module functor defined by
 $1:\one_X\to\cA$, see~\ref{crl:lmod-o-algebra}. $\cO$-algebras in 
 $\Fun(X,\cB)$ are lax $\cO$-monoidal functors from $X$ to $\cB$. 
\begin{Dfn} 
A   lax  $\cO$-monoidal $\cM$-functor $f:\cA\to\cB$ is called
$\cO$-monoidal if the lax $\cO$-monoidal functor $1^*(f):X\to\cB$ obtained from $f$ by forgetting the $\cA$-module structure, is $\cO$-monoidal.
\end{Dfn}
We denote by $\Fun^{\cO}_\cM(\cA,\cB)$
the category of $\cO$-monoidal $\cM$-functors from $\cA$ to $\cB$.

\subsection{$\cO$-monoidal Yoneda embedding}

\subsubsection{}
Assume $\cM$ is $\cO\otimes\Ass$-monoidal category in 
$\Cat^L$ and let $\cA$ be an 
$\cO$-monoidal $\cM$-enriched precategory, that is, an 
$\cO$-algebra in $\PCat(\cM)$.

We will now repeat the construction of Yoneda embedding
for $\cA$. 

Denote $\pi:\BM\to\Ass$ the canonical map of (planar) operads. For any planar operad $\cC$ (or a family of planar operads) the functor
$\pi^*:\Alg_\Ass(\cC)\to\Alg_\BM(\cC)$ carries an associative algebra $A$ to the $A$-$A$-bimodule $A$.

The folding functor $\phi:\Op_\BM\to\Op_\LM$ defined
in \cite{H.EY}, 3.6, preserves limits, so
the functor $\phi\circ\pi^*:\Alg_\Ass(\cC)\to
\Alg_\LM(\phi\circ\pi^*(\cC))$ also preserves limits,
and, therefore, carries $\cO$-algebras to $\cO$-algebras.

We apply this to $\cC:=\Quiv(\cM)$ and an $\cO$-algebra
$\cA$ in $\PCat(\cM)$.
We get an $\cO$-algebra in $\PCat^\LM(\cM\otimes\cM^\rev,
\cM)$.

According to \ref{sss:lm-version}, this defines a lax 
$\cO$-monoidal $\cM\otimes\cM^\rev$-functor 
$\tilde Y:\cA\boxtimes\cA^\op\to\cM$.

\subsubsection{}
\label{sss:mult}
One uses the adjoint associativity equivalence to deduce Yoneda embedding from the functor $\cA\boxtimes\cA^\op\to\cM$.

Recall~\cite{H.EY}, 6.1.7.
Below $\cB$ is a left $\cM\otimes\cM'$-module, $\cA\in\PCat(\cM)$ and $\cA'\in\PCat(\cM')$.
There is a canonical equivalence
\begin{equation}\label{eq:mult}
\Fun_{\cM\otimes\cM'}(\cA\boxtimes\cA',\cB)=
\Fun_\cM(\cA,\Fun_{\cM'}(\cA',\cB)).
\end{equation}

Let now $\cO$ be an operad, $\cM$ and $\cM'$ be in
$\Alg_{\cO\otimes\Ass}(\Cat^L)$ and $\cB$ be an $\cO$-monoidal category left-tensored over $\cM\otimes\cM'$.
Then the left and the right side of  (\ref{eq:mult}) 
are $\cO$-monoidal categories by~(\ref{sss:lm-version}).
\begin{Prp}
Under these assumptions (\ref{eq:mult}) 
is an an equivalence of $\cO$-monoidal categories.
\end{Prp}
\begin{proof}
Let $\cT=(\cM,\cB,\cM^{\prime\rev})$ be the $\BM$-monoidal
category defined by the left $\cM\otimes\cM'$-module $\cB$.
The equivalence~(\ref{eq:mult}) is constructed 
in~\cite{H.EY}, 6.1.7, from the equivalence
\begin{equation}
\label{eq:quiv=quivquiv}
\Funop_\BM(\BM_X\times\BM^\rev_{X'},\cT)=
\Quiv_X^\BM(\Quiv_{X'}^\BM(\cT^\rev)^\rev).
\end{equation}
Both left and right hand sides, cosidered as functors of $\cT$, preserve limits. Therefore, in case $\cT$ is $\cO$-monoidal, (\ref{eq:quiv=quivquiv}) is an equivalence of 
$\cO\otimes\BM$-monoidal categories.
\end{proof}
Proposition~\ref{sss:mult} and the fact that the functor 
$\tilde Y:\cA\boxtimes\cA^\op\to\cM$ is lax $\cO$-monoidal, 
immediately imply that the Yoneda embedding 
$Y:\cA\to P_\cM(\cA)$ is lax $\cO$-monoidal.

\begin{prp}
\label{prp:YOmonoidal}
$Y:\cA\to P_\cM(\cA)$ is  $\cO$-monoidal.
\end{prp}
\begin{proof}
We have to verify that the lax $\cO$-monoidal functor 
$Y:X\to P_\cM(\cA)=\LMod_{\cA^\op}(\Fun(X^\op,\cM))$ is 
$\cO$-monoidal.  

The unit $1:\one_X\to\cA$ gives rise to a decomposition
of $Y$ into
$$
X\stackrel{Y_\one}{\longrightarrow}P_\cM(\one_X)=\Fun(X^\op,\cM)
\stackrel{F}{\longrightarrow}P_\cM(\cA),
$$
where $F$ is the free $\cA^\op$-module functor, see~\cite{H.EY}, 6.2.6. The free module functor $F$ is $\cO$-monoidal by
\ref{crl:lmod-o-algebra}, whereas $Y_\one$ is the composition
of the usual $\infty$-categorical Yoneda embedding $Y_X:X\to P(X)$
and the functor $P(X)\to P_\cM(\one_X)=P(X)\otimes\cM$, so it is
also $\cO$-monoidal.

\end{proof}

\subsection{Universality of Yoneda embedding in the monoidal setting}

Recall that the composition with the Yoneda embedding
defines an equivalence~(\ref{eq:y*})
$$
Y^*:\Fun_\cM^L(P_\cM(\cA),\cB)\to
\Fun_\cM(\cA,\cB)
$$
for any left $\cM$-tensored category with colimits $\cB$.
In this subsection we present two $\cO$-monoidal versions of
this equivalence.

\subsubsection{}
Let $\cA$ be an $\cO$-monoidal $\cM$-enriched category.

In this case $P_\cM(\cA)$ is also $\cO$-monoidal (as a left
$\cM$-module) and the Yoneda embedding is an $\cO$-monoidal 
$\cM$-functor, see~\ref{prp:YOmonoidal}. 
In Theorem~\ref{thm:O-equivalence} below we show that the equivalence (\ref{eq:y*})
induces the equivalences
\begin{equation}
\label{eq:eq-O-lax}
Y^{\cO,\lax*}:\Fun^{\cO,\lax}_{\LMod_\cM}(P_\cM(\cA),\cB)\to
\Fun^{\cO,\lax}_\cM(\cA,\cB)
\end{equation}
and
\begin{equation}
\label{eq:eq-O}
Y^{\cO*}:\Fun^{\cO}_{\LMod_\cM}(P_\cM(\cA),\cB)\to
\Fun^{\cO}_\cM(\cA,\cB).
\end{equation}

\subsubsection{}
The assignment
$$\cB\mapsto\Fun_\cM(\cA,\cB),$$
as defined by the formula (\ref{eq:funM}),
defines a functor $\Alg_\cO(\LMod_\cM)\to\Alg_\cO(\Cat)$
which yields a canonical functor
\begin{equation}
\Fun^{\cO,\lax}_\cM(\cA,P_\cM(\cA))\times
\Fun^{\cO,\lax}_{\LMod_\cM}(P_\cM(\cA),\cB)\to
\Fun^{\cO,\lax}_\cM(\cA,\cB).
\end{equation}
Evaluating it at the $\cO$-monoidal Yoneda embedding $Y:\cA\to P_\cM(\cA)$, we get a map~(\ref{eq:eq-O-lax}).

The map~(\ref{eq:eq-O}) is its restriction.

\begin{thm}
\label{thm:O-equivalence}
Let $\cO$ be an operad, $\cM\in\Alg_{\cO\otimes\Ass}(\Cat^L)$ be an 
$\cO\otimes\Ass$-monoidal category with colimits.
Let, furthermore, $\cA$ be an $\cO$-monoidal $\cM$-enriched
category and let
$\cB\in\Alg_\cO(\LMod_\cM)$ be an $\cO$-monoidal 
category  with colimits left-tensored over $\cM$. Then
the functors~(\ref{eq:eq-O-lax}) and~(\ref{eq:eq-O}) are equivalences. 

\end{thm}

\begin{proof}
By Proposition~\ref{prp:Q-SM} $\cQ_{X,\cM,\cB}$ is an
$\cO$-algebra object in $\Mon_{\TEN_\succ}^{\colax,L}$.
By Lemma~\ref{lem:duality} it gives rise to a $\TEN_\succ$-algebra object 
$Q_{X,\cM,\cB}$ in $\Mon^{\lax,L}_\cO$. This yields,
for an $\cO$-monoidal $\cM$-enriched category $\cA$, 
similarly to~(\ref{eq:wc}), a lax $\cO$-monoidal functor
\begin{equation*}
\colim: P_\cM(\cA)\times\Fun_\cM(\cA,\cB)\to\cB.
\end{equation*}
By \ref{prp:YOmonoidal}, this induces a functor
\begin{equation}
\label{eq:Ocolim}
\colim:\Fun_\cM^{\cO,\lax}(\cA,\cB)\to
\Fun_{\LMod_\cM}^{\cO,\lax}(P_\cM(\cA),\cB).
\end{equation}

Let $[\cO]=\cO_1^\eq$ be the space of colors of the operad $\cO$. We have a forgetful functor
$G:\Alg_\cO(\Cat^\cK)\to(\Cat^\cK)^{[\cO]}$
that commutes with sifted colimits, see~\cite{L.HA}, 3.2.3.1.

By Proposition~\ref{prp:S-family}(3),  the
forgetful functor $G$ commutes with the weighted colimits.

Choose $o\in[\cO]$ and look at the following diagram 
\begin{equation}
\xymatrix{
&{\Fun_{\LMod_\cM}^{\cO,\lax}(P_\cM(\cA),\cB)}\ar@<1ex>[r]^{\quad Y^{\cO*}}
\ar[d]^{G_a}
&{\Fun^{\cO,\lax}_\cM(\cA,\cB)}\ar@<1ex>[l]^{\quad\colim}\ar[d]^{G_a} \\
&{\Fun_{\LMod_\cM}(P_{\cM_o}(\cA_o),\cB_o)}
\ar@<1ex>[r]^{\quad Y^*_o}
&{\Fun_{\cM_o}(\cA_o,\cB_o)}\ar@<1ex>[l]^{\quad\colim},
}
\end{equation}
with $Y^*_o$ induced by the Yoneda embedding for $\cA_o$ and $G_a$
denoting the $a$-component of the functor forgetting the $\cO$-algebra structure. Both squares of the diagram are homotopy commutative.

Since for all $o\in[\cO]$ the arrows $Y^*_o$ and $\colim$ are homotopy inverse, and since the forgetful functor $G$ is conservative, $Y^{\cO*}$ and $\colim$ are also homotopy inverse.

This proves that (\ref{eq:eq-O-lax}) is an equivalence. 
Let us prove that the map (\ref{eq:eq-O}) is also an equivalence. This amounts to verifying that a lax $\cO$-monoidal functor
$$
\Phi:P_\cM(\cA)\to\cB
$$
is $\cO$-monoidal whenever its composition $\phi$ with the embedding
$Y:X\subset P_\cM(\cA)$ is $\cO$-monoidal. The embedding $Y$ is
a composition
$$
X\stackrel{h}{\to}P(X)\stackrel{i}{\to}\Fun(X^\op,\cM)
\stackrel{F}{\to}P_\cM(\cA)
$$
where $h$ is the (non-enriched) Yoneda embedding, $i$ is induced
by the unit $\cS\to\cM$ and $F$ is the free $\cA$-module functor,
see~\cite{H.EY}, 6.2.6. If $\phi:X\to\cB$ is $\cO$-monoidal,
the composition $\Phi\circ F\circ i:P(X)\to\cB$ is $\cO$-monoidal
by Lemma~\ref{lem:Omon-1} below. Then the composition 
$\Phi\circ F:\Fun(X^\op,\cM)\to\cB$ is $\cO$-monoidal as 
$\Phi\circ F$ can be reconstructed from $\Phi\circ F\circ i$ as
the composition
$$
\Fun(X^\op,\cM)=\cM\otimes P(X)\stackrel{\id_M\otimes(\Phi\circ F\circ i)}{\longrightarrow}\cM\otimes\cB\to\cB.
$$
Finally, since any $\Phi\in P_\cM(\cA)$ is a colimit
of free $\cA$-modules, Lemma~\ref{lem:Omon-1} implies that $\Phi$ is
$\cO$-monoidal.
\end{proof}

\begin{lem}
\label{lem:Omon-1}
Let $\Phi:\cC\to\cD$ be an arrow in $\Mon_\cO^{\lax,L}$, 
$i:\cC_0\to\cC$ in $\Mon_\cO$ so that $\Phi\circ i:\cC_0\to\cD$ is
also in $\Mon_\cO$. Assume that $\cC_0$ generates $\cC$ by 
colimits. Then $\Phi$ is  in $\Mon^L_\cO$.
\end{lem}
\begin{proof}
Let $\alpha:x\to y$ be an arrow in $\cO$. We have to verify that $\Phi$
preserves cocartesian liftings of $\alpha$. Let $c\in\cC$ be over $x$.
We have $c=\colim\{i:I\to(\cC_0)_x\to\cC_x\}$. Since 
$\alpha_!:\cC_x\to\cC_y$ preserves colimits, 
$\alpha_!(c)=\colim\{\alpha_!\circ i\}$. Since $\Phi\circ i$ is in $\Mon_\cO$, the arrows $\Phi(c_i)\to\Phi(\alpha_!(c_i))$ are cocartesian liftings of $\alpha$. Since the functor $\alpha_!:\cD_x\to\cD_y$
preserves colimits, the arrow $\colim(\Phi(c_i))\to\colim(\Phi(\alpha_!(c_i)))$ is a cocartesian lifting of $\alpha$.
\end{proof}

\end{document}